\documentclass[reqno,10pt]{amsart}
\usepackage{graphicx,amsfonts,amssymb,amsmath,amsthm,url,amscd,comment}
\usepackage{color}
\usepackage{chngcntr}
\usepackage[usenames,dvipsnames]{xcolor}
\usepackage[normalem]{ulem}
\usepackage{pdfsync}
\usepackage{graphicx, amsmath, amssymb, amsfonts, amsthm, stmaryrd, tikz, amscd}
\usepackage[all]{xy}
\usetikzlibrary{matrix,arrows,decorations.pathmorphing}
\usepackage{graphicx, amsmath, amssymb, amsfonts, amsthm, stmaryrd, tikz, amscd}
\usepackage[all]{xy}
\usetikzlibrary{matrix,arrows,decorations.pathmorphing}

 \usepackage[hyperfootnotes=false, colorlinks, citecolor=	 RoyalBlue, urlcolor=blue, linkcolor=blue         ]{hyperref}

\theoremstyle{plain} 
\newtheorem{theorem}             {Theorem} 

\theoremstyle{definition}

\theoremstyle{plain} 
\newtheorem{proposition} {Proposition}

\theoremstyle{remark}

\newtheorem{definition}  {Definition}
\counterwithin*{definition}{subsubsection}
\newtheorem*{definition*}  {Definition}

\newtheorem*{example*}    {Example}

\newtheorem{remark}             {Remark}
\newtheorem*{remark*}            {Remark}

\newtheoremstyle{itplain} 
    {6pt}                    
    {5pt\topsep}                    
    {\itshape}                   
    {}                           
    {\itshape}                   
    {.}                          
    {5pt plus 1pt minus 1pt}                       
    {}  

\theoremstyle{itplain} 
\newtheorem{lemma}{Lemma}

\newtheorem*{lemma*}{Lemma}

\newtheorem*{corollary*} {Corollary} 

\theoremstyle{remark} 

\newtheorem*{lemmatest*}{Lemma}

\counterwithin*{lemma}{subsubsection}
\counterwithin*{remark}{subsubsection}
\counterwithin*{corollary}{subsubsection}

\usepackage{etoolbox}
\patchcmd{\section}{\scshape}{\bfseries}{}{}
\makeatletter
\renewcommand{\@secnumfont}{\bfseries}
\makeatother

\renewcommand{\Re}{\mathrm{Re}}

\renewcommand{\geq}{\geqslant}
\renewcommand{\leq}{\leqslant}

\numberwithin{equation}{section}

\DeclareMathOperator{\SL}{SL}
\DeclareMathOperator{\Mp}{Mp}
\DeclareMathOperator{\GL}{GL}

\DeclareMathOperator{\GO}{GO}
\DeclareMathOperator{\PGO}{PGO}

\DeclareMathOperator{\htt}{ht}

\DeclareMathOperator{\SO}{SO}

\DeclareMathOperator{\ad}{ad}
\DeclareMathOperator{\Ad}{Ad}

\DeclareMathOperator{\Hom}{Hom}

\DeclareMathOperator{\End}{End}
\DeclareMathOperator{\JL}{JL}

\def\eps{\varepsilon}

\def\PB{\operatorname{PB}}
\def\PGL{\operatorname{PGL}}
\def\Sob{\mathcal{C}}

\DeclareMathOperator{\Weil}{Weil}

\DeclareMathOperator{\bPB}{{\mathbf P}{\mathbf B}}

\DeclareMathOperator{\co}{co}

\def\O{\operatorname{O}}

\DeclareMathOperator{\ram}{ram}
\DeclareMathOperator{\fin}{fin}
\DeclareMathOperator{\cusp}{cusp}

\DeclareMathOperator{\pr}{pr}

\DeclareMathOperator{\reg}{reg}

\DeclareMathOperator{\ord}{ord}

\DeclareMathOperator{\res}{res}

\DeclareMathOperator{\vol}{vol}
\DeclareMathOperator{\nr}{nr}

\DeclareMathOperator{\tr}{tr}

\DeclareMathOperator{\supp}{supp}

\author{Paul D. Nelson}
\address{ETH Z{\"u}rich, Department of Mathematics, R{\"a}mistrasse 101, CH-8092, Z{\"u}rich, Switzerland}
\email{paul.nelson@math.ethz.ch}
\subjclass[2010]{Primary 11F70; Secondary 11F27, 58J51}

\date{\today}
\title{Quantum variance on quaternion algebras, II}
\hypersetup{
  pdfkeywords={},
  pdfsubject={},
  pdfcreator={Emacs 24.5.1 (Org mode 8.2.10)}}
\begin{document}

\begin{abstract}
  A method for determining quantum variance asymptotics on
  compact quotients attached to non-split quaternion algebras is
  developed in general and applied to ``microlocal lifts'' in the
  non-archimedean setting.  The results obtained are in the
  spirit of recent work of Sarnak--Zhao.

  The arguments involve a careful analytic study of the theta
  correspondence, the interplay between additive and
  multiplicative harmonic analysis on quaternion algebras, the
  equidistribution of translates of elementary theta functions,
  and the Rallis inner product formula.
\end{abstract}
\maketitle

\setcounter{tocdepth}{1} \tableofcontents

\section{Introduction}
\label{sec-1}
\subsection{Overview}
\label{sec-1-1}
The quantum variance problem
(see e.g. \cite[\S1]{nelson-variance-73-2},
\cite{MR848319},
\cite[\S15.6]{2009arXiv0911.4312Z}, \cite[\S4.1.3]{MR3204186},
\cite{ 2013arXiv1303.6972S, MR1361757,MR1465794,luo-sarnak-mass,MR2103474,MR2651907})
concerns
sums of the shape
\begin{equation}\label{eq:qv-sums-classical}
    \sum_{\varphi \in \mathcal{F}}
  \langle \varphi, \Psi_1 \varphi \rangle
  \langle \Psi_2 \varphi, \varphi \rangle.
\end{equation}
Here
$\Psi_1,\Psi_2$ are fixed mean zero functions on
the unit cotangent bundle of a Riemannian manifold $M$
with ergodic geodesic flow,
$\mathcal{F}$ traverses a sequence of
families of microlocal lifts of Laplace eigenfunctions with
eigenvalues in  $[0, T^2]$, and $T \rightarrow \infty$.
The problem
is to determine the leading order asymptotic
behavior of \eqref{eq:qv-sums-classical}.
The difficulty
of the problem may be appreciated
by comparing
the expected magnitude
$\asymp T$
for
\eqref{eq:qv-sums-classical}
for typical $\Psi_1 = \Psi_2$
with
the best known general upper bound
$O(T^{\dim(M)} / \log T)$ (see
e.g. \cite[\S1]{nelson-variance-73-2}
and references for details).

Although a mathematically rigorous solution to the problem seems
hopeless on general $M$, Sarnak--Zhao
\cite{2013arXiv1303.6972S} (following Luo--Sarnak
\cite{luo-sarnak-mass} and Zhao \cite{MR2651907}) managed to
solve it completely on 
$M = \SL_2(\mathbb{Z}) \backslash \mathbb{H}$
for $\mathcal{F}$ consisting of Hecke
eigenfunctions.
It is natural to seek analogous results on other arithmetic
quotients, such as the compact quotients
attached to orders in quaternion division algebras.  
The method of Luo, Sarnak and Zhao demonstrates the
tremendous power of \emph{parabolic Fourier expansions}, such as
the $q$-expansions $\sum a_n q^n$ enjoyed by classical
holomorphic modular forms on $\SL_2(\mathbb{Z})$ at the cusp
$\infty$, to establish results that are inaccessible by means of
semiclassical analysis or trace formulas alone.  Conversely,
their technique is fundamentally limited to \emph{split}
quotients, such
$\SL_2(\mathbb{Z}) \backslash \mathbb{H}$ and its congruence covers, on which such
expansions are available.

In this article, we develop systematically a method for studying
quantum variance on \emph{non-split} arithmetic quotients
arising from non-split quaternion algebras, in contrast to the
split matrix algebra $M_2(\mathbb{Q})$ underlying the quotient
$\SL_2(\mathbb{Z}) \backslash \mathbb{H}$ considered by Luo,
Sarnak and Zhao.  Our main result (Theorem
\ref{thm:main-result-for-microlocal-stuff}, stated in
\S\ref{sec-1-5})
concerns
families of ``microlocal lifts'' on non-split $p$-adic arithmetic
quotients attached to quaternion algebras.
We focus on $p$-adic
quotients to simplify the analysis; this passage does not affect
the fundamental difficulty of the problem, and is orthogonal to
the primary novelty that the quotient is non-split.  Our core
estimates (Theorem \ref{thm:main-estimate-general-variance},
stated in \S\ref{sec-4}) are developed in some generality.









\subsection{Trace formulas and linear statistics}\label{sec:trace-formula-linear}
Let $\mathbf{X} := \Gamma \backslash G$ be the
quotient by an arithmetic lattice of the points of a
semisimple $\mathbb{Q}$-group over a local field, such as the
real numbers, and let $\mathcal{F}$ be a ``large'' collection of
eigenfunctions $\varphi : \mathbf{X} \rightarrow \mathbb{C}$.
It
is natural to ask for the asymptotic statistics, as
``$\mathcal{F} \rightarrow \infty$'' in some sense, of the
random measure on $\mathbf{X}$ sending a test function $\Psi$ to
$\langle \varphi, \Psi \varphi \rangle$, where $\varphi \in
\mathcal{F}$
is sampled randomly with respect to (say) the normalized
counting measure.

The \emph{linear} statistics
of this random measure
are captured by the \emph{mean}
$\Psi \mapsto  \mathbb{E}_{\varphi \in \mathcal{F}} \langle \varphi, \Psi \varphi  \rangle$.
When $\mathcal{F}$ admits a nice harmonic-analytic description,
it can be (at least approximately) picked off by a convolution
kernel $f \in C_c^\infty(G)$.
The mean can then be studied using trace
formula techniques:
by integrating the pretrace formula
\begin{equation}\label{eq:intro-pretrace-fromula}
  \sum_{\gamma \in \Gamma} f(x^{-1} \gamma y)
  \approx
  \sum_{\varphi \in \mathcal{F} }
  \overline{\varphi(x)} \varphi(y) \quad (x,y \in \Gamma \backslash G)
\end{equation}
over the diagonal against $\Psi$,
one obtains an identity
\begin{equation}\label{eq:intro-pretrace-fromula-1}
  \sum_{\varphi \in \mathcal{F} }
  \langle \varphi, \Psi \varphi  \rangle
  \approx
  \int_{x \in \Gamma \backslash G}
  \Psi(x) \sum_{\gamma \in \Gamma} f(x^{-1} \gamma x)
\end{equation}
whose RHS may be studied
by methods
for bounding orbital integrals
much as in the ``Weyl's law'' case $\Psi \equiv 1$.
(For an example of such arguments, see \S\ref{sec:mean-statistics}.)

Higher-order statistics
such as the $n$-point correlations
\[
(\Psi_1,\dotsb,\Psi_n)
\mapsto
\mathbb{E}_{\varphi \in \mathcal{F}}
\langle \varphi, \Psi_1 \varphi  \rangle
\dotsb 
\langle \varphi, \Psi_n \varphi  \rangle
\]
are more mysterious.  The quantum variance problem concerns the
quadratic statistics \eqref{eq:qv-sums-classical} about which
trace formulas alone say little.

\subsection{Hecke multiplicativity and variance statistics}
Until this work and its
prequel, the only known asymptotic formulas for higher-order
statistics in this setting of \S\ref{sec:trace-formula-linear} were those of Luo--Sarnak--Zhao
concerning $\SL_2(\mathbb{Z}) \backslash \SL_2(\mathbb{R})$.
The
point of departure for their method is that
when the eigenfunctions $\varphi$ admit Fourier
expansions with coefficients $\lambda(n)$ enjoying
a ``doubling identity'' of the shape
\begin{equation}\label{eqn:hecke-multiplicativity}
  \lambda(m) \lambda(n) = \sum \lambda(\dotsb),
\end{equation}
one can try to reduce
variance statistics to linear
ones and apply trace formulas such as
\eqref{eq:intro-pretrace-fromula}.
This method does not apply when
such expansions are not available.

\subsection{Theta functions and variance statistics}\label{sec:theta-funct-vari}
When the space $\mathbf{X}$ arises from a quaternion algebra
$B$ 
(over $\mathbb{Q}$, say),
the Eichler/Shimizu theta correspondence provides an analogue
of the doubling identity \eqref{eqn:hecke-multiplicativity}
that suggests a natural strategy
for studying quantum variance.
We pursue this strategy here.
Let $\mathcal{F}$
be a family
of eigenfunctions
on $\mathbf{X}$.
Oversimplifying
for now,
Shimizu's theorem (see \cite[II.1]{MR783511})
says that
one can find
\begin{itemize}
\item a space $\mathbf{X} '$ (a congruence cover of
  $\PGL_2(\mathbb{Z})
  \backslash \PGL_2(\mathbb{R})$),
\item a function of three variables
  $\Theta : \mathbf{X} \times \mathbf{X} \times \mathbf{X} '
  \rightarrow \mathbb{C}$ (a theta kernel), and
\item for each $\varphi \in \mathcal{F}$,
  a function $\Phi_\varphi : \mathbf{X} ' \rightarrow
  \mathbb{C}$
  (an Eichler/Jacquet--Langlands lift)
\end{itemize}
with the property that
\begin{equation}\label{eq:shimizu-identity}
  \overline{\varphi(x)} \varphi(y)
  =
  \int_{z}
  \Phi_\varphi(z)
  \Theta(x,y;z)
  \quad \text{ for all $\varphi \in \mathcal{F}$
    and $x,y \in \mathbf{X}$.
  }
\end{equation}
By integrating the diagonal case $x=y$ of
\eqref{eq:shimizu-identity}
against $\Psi$,
one obtains
\begin{equation}\label{eq:identity-pre-parseval}
  \langle \varphi, \Psi \varphi  \rangle
=
\int_{z}
\Phi_\varphi(z)
\int_x \Psi(x)
\Theta(x,x;z).
\end{equation}
If the functions $\Phi_\varphi$ are orthogonal to one
another and the family $\mathcal{F}$ is sufficiently
``complete,'' then a cavalier application of Parseval's formula
to \eqref{eq:identity-pre-parseval}
suggests
that
\begin{equation}\label{eq:formula-seesaw}
  \sum_{\varphi \in \mathcal{F}}
(\int_z |\Phi_\varphi|^2 )^{-1}
\langle \varphi, \Psi_1 \varphi  \rangle
\langle \Psi_2 \varphi, \varphi  \rangle
=
\int_{z}
(\int_x \Psi_1(x)
\Theta(x,x;z))
\overline{(\int_y \Psi_2(y)
\Theta(y,y;z))}.
\end{equation}
The LHS \eqref{eq:formula-seesaw} may be understood as a
reasonable proxy for the quantum variance
\eqref{eq:qv-sums-classical} of $\mathcal{F}$ provided that the
weights $\int_z |\Phi_\varphi|^2$ are sufficiently uniform in
$\varphi$.
One aim of this article is to develop robust techniques for
determining the asymptotics of the RHS of
\eqref{eq:formula-seesaw}, which is not \emph{a priori} any
simpler to analyze than the LHS.  A second aim is to apply the
resulting machinery to an interesting family of automorphic forms.

We have oversimplified by neglecting
that the theta kernel
$\Theta$ produced by Shimizu's theorem may (and generally does)
depend upon the automorphic form $\varphi$.  For the above
argument to make sense, we need to choose one $\Theta$ that
works for every element of the family $\mathcal{F}$.  It is
natural instead to interpret \eqref{eq:formula-seesaw} as
\emph{defining} a (weighted) family $\mathcal{F}$ in terms of
$\Theta$.  A  third aim of this article is then to clarify in general
how to invert the association $\Theta \mapsto \mathcal{F}$.


\subsection{Microlocal lifts in the non-archimedean setting\label{sec:setting-overview}}
\label{sec-1-2}
Let $k$ be a non-archimedean local field
of characteristic zero.
Denote by $\mathfrak{o}$  its maximal order,
$\mathfrak{q}$ its maximal ideal,
and $q := \# \mathfrak{o}/\mathfrak{q}$.
For example,
one can take $(k,\mathfrak{o},\mathfrak{q},q)
:= (\mathbb{Q}_p,\mathbb{Z}_p, p \mathbb{Z}_p, p)$
for some rational prime number $p$.

Fix a totally real number field $F$ having $k$ as its
completion.
Fix a discrete cocompact subgroup $\Gamma$ of $\PGL_2(k)$
arising from a maximal order in a (non-split) totally definite
quaternion algebra over $F$
(see \S\ref{sec:strong-approx}
for details).  Set
\[\mathbf{X} := \Gamma \backslash \PGL_2(k).\]
To simplify the present exposition,
we assume
that $F$ has odd narrow class number,
so that $\mathbf{X}$
comes with a natural
family of commuting Hecke operators
(see \S\ref{sec:hecke-ops}), which also commute with the
right translation action by $\PGL_2(k)$.
Fix any $\PGL_2(k)$-invariant measure on $\mathbf{X}$;
denote by $\vol(\mathbf{X})$ the total volume
and define
$L^2(\mathbf{X})$ and $\langle, \rangle$ by integrating.
\begin{definition}\label{defn:eigenfunctions-intro}
  By
  an \emph{eigenfunction},
  we shall mean 
  a nonzero function $\varphi : \mathbf{X} \rightarrow \mathbb{C}$
  that is smooth (i.e., right-invariant under some open subgroup),
  is an eigenfunction under every Hecke operator, and
  generates an irreducible representation of $\PGL_2(k)$ under
  right translation.

  Some of the Hecke operators are involutions; an
  eigenfunction will be called \emph{even} if it has eigenvalue
  $+1$, rather than $-1$, under each such involution.
\end{definition}

\begin{definition}
  Let $N$ be a large integer and let
  $\omega : \mathfrak{o}^\times \rightarrow \mathbb{C}^\times$ be
  a unitary character of conductor $N$,
  thus $\omega$ is trivial on $1 + \mathfrak{q}^{N}$
  but not on $1 + \mathfrak{q}^{N-1}$.
  An eigenfunction
  $\varphi$ will be called a \emph{microlocal lift of orientation
    $\omega$} if it satisfies
  $\varphi(x g) = \omega(a^2/\det(g)) \varphi(x)$ for all
  $x \in \mathbf{X}$ and all
  \begin{equation}\label{eq:microlocal-lift-subgroup-defn-0}
    g = \begin{pmatrix}
      a & b \\
      c & d
    \end{pmatrix}
    \in
    \GL_2(\mathfrak{o}) \cap \begin{pmatrix}
      \mathfrak{o} & \frac{1}{2} \mathfrak{q}^{\lfloor N/2 \rfloor} \\
      \mathfrak{q}^{\lceil N/2 \rceil} & \mathfrak{o}
    \end{pmatrix}.
  \end{equation}
\end{definition}

There is a finite set $\mathcal{F}_\omega$ consisting of unit
vector microlocal lifts of orientation $\omega$ so that
$\{ \text{microlocal lifts of orientation $\omega$} \} =
\bigsqcup_{\varphi \in \mathcal{F}_\omega} \mathbb{C}^\times
\varphi$.
(see \S\ref{sec:deduction-main-thm-microlocal}
or
\cite{nelson-padic-que}).
Let $\mathcal{F}_N$ denote the union of
$\mathcal{F}_\omega$ over all $\omega$ of conductor $N$.

\begin{lemma*}
  There is an explicit $c > 0$
  so that for $N$ large enough,
  \begin{equation}\label{eq:family-size-booyah}
    |\mathcal{F}_N| = c q^{2 N}.
  \end{equation}
\end{lemma*}
\begin{proof}
  By a trace formula computation (see \S\ref{sec:mean-statistics}).
\end{proof}



\begin{remark*}
  From a representation-theoretic or microlocal-analytic
  perspective, $\mathcal{F}_N$ is analogous to the family of
  microlocal lifts of Hecke--Maass eigenforms on a compact
  arithmetic hyperbolic surface with Casimir eigenvalues
  $1/4 + t^2$, $t \in [T, 2 T]$, $T \approx q^{N}$.  For
  example, their limit measures enjoy diagonal invariance for
  local reasons.  We refer to \cite[Thm 25, Rmk
  26]{nelson-padic-que} for a discussion of the sense in which
  the microlocal lifts considered here are actually ``lifts.''
\end{remark*}

\subsection{Context: worst case behavior and mean statistics}
\label{sec-1-3}
It was shown recently in \cite{nelson-padic-que}
that microlocal lifts on $\mathbf{X}$
satisfy an analogue
of the arithmetic quantum unique ergodicity
theorem \cite{MR2680500, MR2195133}:
for each continuous $\Psi : \Gamma \backslash
\PGL_2(k) \rightarrow \mathbb{C}$,
\begin{equation}\label{eqn:padic-que-result-worst-case-booyah}
  \lim_{N \rightarrow \infty}
  \max_{\varphi \in \mathcal{F}_N}
  |
  \langle \varphi, \Psi \varphi \rangle
  - \langle \Psi,1 \rangle / \langle 1, 1 \rangle
  |
  = 0.
\end{equation}
The Lindel{\"o}f hypothesis
predicts more precisely that
\begin{equation}\label{eq:lindelof-prediction}
  \langle \varphi, \Psi \varphi \rangle
  = \langle \Psi,1 \rangle /\langle 1,1 \rangle
  + O(q^{-N/2 + o(N)})
  \text{ for $\Psi$ fixed and  $\varphi \in \mathcal{F}_N$}, 
\end{equation}
and it is generally believed
that this prediction is essentially optimal.

\begin{lemma*}\label{lem:mean-stats-bizooyah}
  For each continuous $\Psi : \Gamma \backslash
  \PGL_2(k) \rightarrow \mathbb{C}$,
  \begin{equation}\label{eq:mean-stats-bizooyah}
    \lim_{N \rightarrow \infty}
    |\mathcal{F}_N|^{-1}
    \sum_{\varphi \in \mathcal{F}_N}
    \langle \varphi, \Psi \varphi \rangle
    = \langle \Psi,1 \rangle /\langle 1,1 \rangle.
  \end{equation}
\end{lemma*}
\begin{proof}
  This follows
  by averaging \eqref{eqn:padic-que-result-worst-case-booyah};
  alternatively, one can argue as in \S\ref{sec:trace-formula-linear}
  using the pretrace formula (see \S\ref{sec:mean-statistics}).
\end{proof}


\subsection{Definition of variance sums}
\label{sec-1-4}
Consider for each large natural number $N$ the random measure on
$\mathbf{X}$ sending a continuous function
$\Psi : \mathbf{X} \rightarrow \mathbb{C}$ to
$\langle \varphi, \Psi \varphi \rangle$, where
$\varphi \in \mathcal{F}_N$ is sampled uniformly at random.
The lemma of \S\ref{lem:mean-stats-bizooyah} says that this
random
measure has expectation tending
as $N \rightarrow \infty$
to the invariant probability measure on
$\mathbf{X}$.
The most natural definition of the variance
of that random measure
would then be the bilinear form
on smooth mean zero functions
$\Psi_1,\Psi_2 : \mathbf{X} \rightarrow
\mathbb{C}$
given by
\[
|\mathcal{F}_N|^{-1}
\sum_{\varphi \in \mathcal{F}_N}
\langle
\varphi, \Psi_1 \varphi 
\rangle
\langle \Psi_2 \varphi , \varphi 
\rangle.
\]
For technical reasons related to our method
(which the reader may infer already from \eqref{eq:formula-seesaw}),
we consider instead
the slightly modified (and unnormalized) sums
\begin{equation}\label{eq:variance-defn-for-refined-intro}
  V_N(\Psi_1,\Psi_2)
  :=
  \sum_{\varphi \in \mathcal{F}_N}
  \iota_\varphi
  \langle
  \varphi, \Psi_1 \varphi 
  \rangle
  \langle \varphi \Psi_2, \varphi 
  \rangle,
\end{equation}
where
$\iota_\varphi
:=
L^{(S)}(\ad
\varphi,1)$
is a standard ``harmonic weight''
(see
\S\ref{sec:standard-l-function},
\S\ref{sec:variance-statistics}).
Such weights are positive and mild in
that their size is $\iota_{\varphi} = q^{o(N)}$
(see \eqref{eq:HL}) and their mean essentially
constant. Although there are well-developed
techniques for removing them as in \cite{2013arXiv1303.6972S},
we retain them here for simplicity.

It follows from the compactness of $\mathbf{X}$ that every smooth function on $\mathbf{X}$
is a finite linear combination of eigenfunctions,
so in studying
\eqref{eq:variance-defn-for-refined-intro},
we may assume by linearity
that $\Psi_1,\Psi_2$
are eigenfunctions.
It is natural to assume further that
\begin{itemize}
\item $\Psi_1, \Psi_2$ are \emph{even},
  as otherwise $V_N(\Psi_1,\Psi_2) = 0$ by parity considerations,
  and that
\item $\Psi_1,\Psi_2$ are \emph{strongly of mean zero} in 
  that they are orthogonal to the finite collection of
  one-dimensional subrepresentations of $L^2(\mathbf{X})$, as
  otherwise $V_N(\Psi_1,\Psi_2) = 0$ by basic properties of the
  families $\mathcal{F}_N$ (see \cite[Lem 51]{nelson-padic-que}).
\end{itemize}

\subsection{Statement of main result}
\label{sec-1-5}
By \eqref{eq:family-size-booyah} and
\eqref{eq:lindelof-prediction}, one expects $V_N(\Psi_1,\Psi_2)$
to have magnitude at most $O(q^{N})$.  It should be possible to
confirm this expected upper bound using the Cauchy--Schwarz
inequality, the triple product formula, and well-developed
techniques for averaging families of $L$-functions as in
\cite{2013arXiv1303.6972S}, but the true asymptotics of
$V_N(\Psi_1,\Psi_2)$ are much subtler: in the off-diagonal case
that $\Psi_1, \Psi_2$ generate distinct irreducible
representations $\pi_1 \neq \pi_2 \subseteq L^2(\mathbf{X})$,
one expects the \emph{signs} of the quantities
$\langle \varphi, \Psi_1 \varphi \rangle$ and
$\langle \varphi, \Psi_2 \varphi \rangle$ to vary independently,
suggesting additional cancellation in
\eqref{eq:variance-defn-for-refined-intro}.
The primary novelty in the following result
is that we detect such cancellation in the off-diagonal case; a
secondary novelty is that in the diagonal case, we determine the
main term.
\begin{theorem}\label{thm:main-result-for-microlocal-stuff}
  For even eigenfunctions $\Psi_1, \Psi_2$ strongly of mean zero,
  the limit
  \begin{equation}\label{eq:main-result-intro-1}
    \lim_{N \rightarrow \infty}
    q^{-N} V_N(\Psi_1, \Psi_2)
  \end{equation}
  exists.
  The limit is zero unless $\Psi_1,\Psi_2$
  generate the same irreducible representation
  $\pi \subseteq L^2(\mathbf{X})$.
  In that case, it is
  equal to
  \begin{equation}\label{eq:main-result-intro-main-term}
    c_0 L^{(S)}(\pi,\tfrac{1}{2})
    \int_{h \in N(H)}
    (\int_{x \in \mathbf{X}}
    \Psi_1(x h)
    \overline{\Psi_2(x)}),
  \end{equation}
  where  (see \S\ref{sec:deduction-main-thm-microlocal} 
  for details)
  \begin{itemize}
  \item $c_0$ is an explicit positive constant,
  \item $S$ denotes the set of ``bad places,''
  \item $L^{(S)}(\pi,\tfrac{1}{2})$
    denotes the central $L$-value without Euler factors
    in $S$, and
  \item $N(H)$
    denotes the normalizer
    of the diagonal torus $H$ in $\PGL_2(k)$.
  \end{itemize}
\end{theorem}
\begin{remark}
  The integral \eqref{eq:main-result-intro-main-term}
  converges as written
  (see \S\ref{sec:local-Xi},
  \S\ref{sec:local-convergence-lemmas}).
\end{remark}
\begin{remark}
  As in \cite{2013arXiv1303.6972S}, the identity
  \eqref{eq:main-result-intro-main-term} admits an intriguing
  semiclassical interpretation whereby the arithmetical values
  $L^{(S)}(\pi,\tfrac{1}{2})$ quantify the deviation of the
  asymptotic quantum variance from the (symmetrized)
  classical variance of the diagonal flow.
\end{remark}
\begin{remark}
  We establish the stronger
    assertion that
    $q^{-N} V_N(\Psi_1, \Psi_2)$
    differs from its limit
    by $O(N q^{-N})$ (see \S\ref{sec:deduction-main-thm-microlocal}).
    We expect that our method is capable of refining that error
    term to $O(q^{-N})$
    and proving that such an estimate is essentially best
    possible (see \cite[\S6.5]{nelson-variance-73-2}),
    but we do not pursue such refinements here.
  \end{remark}
  \begin{remark}
    We note as in \cite{MR2103474}
    that
    Theorem \ref{thm:main-result-for-microlocal-stuff} confirms the Lindel{\"o}f prediction
    \eqref{eq:lindelof-prediction} on average and
    implies that it is essentially optimal if it is true.
  \end{remark}
  \begin{remark}   The off-diagonal case is that in which the method of
    Luo--Sarnak--Zhao requires a cusp.  It is conceivable that one
    could establish the diagonal case of Theorem
    \ref{thm:main-result-for-microlocal-stuff} by averaging the
    triple product formula as in \cite{2013arXiv1303.6972S}.  Our
    method does not use the triple product formula.
  \end{remark}

  \subsection{Comparison with the prequel}
  We highlight some differences
  between the results and aims of \cite{nelson-variance-73-2}
  and those of this article.
  \begin{enumerate}
  \item In \cite{nelson-variance-73-2}, the observables
    $\Psi_1, \Psi_2 : \mathbf{X} \rightarrow \mathbb{C}$ were
    restricted to be right-invariant by the maximal compact
    subgroup $K := \PGL_2(\mathfrak{o})$.  In this article, we
    study arbitrary fixed observables on the full ``phase space''
    $\mathbf{X}$
    rather than on the ``configuration space''
    $\mathbf{X} / K$.
    The simplifying restriction of the prequel
    allowed us to get by with some \emph{ad hoc} computations in
    places where a more systematic approach is required here.

    This jump in complexity is analogous to that from
    \cite{MR2103474, MR2651907} to \cite{2013arXiv1303.6972S},
    but the manner in which the new complexity
    is addressed differs
    completely
    (in a stronger sense
    than that the methods themselves
    differ completely):  In \cite{2013arXiv1303.6972S},
    phase space observables were treated by an inductive technique
    involving weight isotypic vectors, while the methods developed
    here apply directly to general phase space observables. 
  \item In the prequel, we considered families $\mathcal{F}_N'$ of
    balanced newvectors.  Here, we consider families
    $\mathcal{F}_N$ of microlocal lifts.  The methods developed in
    Part 1 of this article apply robustly to a large class of
    families;
    given those methods, neither
    $\mathcal{F}_N$ nor $\mathcal{F}_N'$ is much more difficult to
    analyze than the other.
    We focused in \cite{nelson-variance-73-2} on newvector
    families because of their familiarity.  We focus here on the
    families $\mathcal{F}_N$ because of their strong analogy with
    those in the motivating work of Luo--Sarnak--Zhao and because
    the formulas \eqref{eq:main-result-intro-main-term}
    for the main term for $\mathcal{F}_N$ are more
    aesthetically appealing than those for $\mathcal{F}_N'$.
  \item The aim of \cite{nelson-variance-73-2} was to introduce
    the method by application to the simplest non-trivial
    non-split case of the quantum variance problem and in the
    most elementary language possible.\footnote{
      In the prequel, the base field is unimportantly
      restricted to
      be $\mathbb{Q}_2$ (rather than $\mathbb{Q}_p$ for an odd
      prime $p$) in order to reduce the set of bad
      primes from $\{2, p\}$ to $\{2\}$, noting that $2$ is
      always bad when studying half-integral weight forms.
    }  The aim here is
    instead to develop the technique as clearly as possible in
    its natural generality.
    We hope the two articles serve complementary purposes.
  \item It may be worth recording that for the reasons indicated
    in the previous three points, the two articles have
    essentially no logical overlap.
  \end{enumerate}






\subsection{Discussion of method}
\label{sec-1-6}
The general strategy of \S\ref{sec:theta-funct-vari}
applies
in the setting of Theorem 
\ref{thm:main-result-for-microlocal-stuff}:
modulo a preliminary technical partition
of the family $\mathcal{F} = \mathcal{F}_N$,
we construct $f = f_N$ for
which \eqref{eq:intro-pretrace-fromula}
holds and then
$\Theta$ for which \eqref{eq:formula-seesaw}
holds.
(For a precise definition
of $f$, see the end of \S\ref{sec:appl-micro-local-prelims}.)
The integral 
$\int_z \Psi_i(z) \Theta(x,x;z)$ does not define
a theta lift of $\Psi_i$ in the traditional sense,
but instead decomposes as a sum of products
$\theta_i(z) h_i(z)$, where $\theta_i$ is a variant
of the Jacobi theta function and $h_i$ is a theta lift
of $\Psi_i$.
The RHS of \eqref{eq:formula-seesaw} then
decomposes as a sum of inner products
\begin{equation}\label{eq:ip-4-theta-before-rearr}
  \langle \theta_1 h_1, \theta_2 h_2 \rangle
\end{equation}
Suppose we can approximate each such inner product
by
\begin{equation}\label{eq:ip-4-theta-after-rearr}
  \langle \theta_1, \theta_2 \rangle \langle h_1, h_2 \rangle.
\end{equation}
The Rallis inner product formula \cite{2012arXiv1207.4709T,
  MR2837015}
for theta lifts applies to
\eqref{eq:ip-4-theta-after-rearr};
summing it up,
we obtain
\begin{equation}\label{eq:asymptotic-after-rallis-in-intro-sketch}
  \sum_{\varphi \in \mathcal{F}}
(\int |\Phi_\varphi|^2 )^{-1}
\langle \varphi, \Psi_1 \varphi  \rangle
\langle \Psi_2 \varphi, \varphi  \rangle
\approx
(\ast)
\int_{h \in G}
I_f(h)
(\int_{x \in \mathbf{X}}
\Psi_1(x h)
\overline{\Psi_2(x)}),
\end{equation}
where:
\begin{itemize}
\item $\approx$ means up to the error incurred
  by
  replacing
  each term
  \eqref{eq:ip-4-theta-before-rearr}
  with
  \eqref{eq:ip-4-theta-after-rearr};
\item 
  $(\ast)$ means ``modify by a central $L$-value as in
  Theorem \ref{thm:main-result-for-microlocal-stuff};'' and
\item 
  $I_f(h) := \int_{g \in G} \mathfrak{S} f(h^{-1} g h)
  \overline{\mathfrak{S} f(g)}$,
  where $\mathfrak{S} f(g)
  := (f(g) + f(g - \tr(g)))/2$.
\end{itemize}
To complete the proof of Theorem
\ref{thm:main-result-for-microlocal-stuff},
it suffices now to show that
\begin{enumerate}
\item[(i)]
  $I_f : G \rightarrow \mathbb{C}$ tends 
  to the ``delta distribution'' on the normalizer $N(H)$ as
  the parameter $N$ tends to $\infty$, and that
\item[(ii)]  the error hidden by
  $\approx$ is negligible.
\end{enumerate}
Problem (i) is purely local.
Problem (ii) involves
both local and global difficulties;
a critical global
input to its solution
was developed in \cite{nelson-theta-squared}.

To indicate in more detail
how this works,
assume for simplicity
that $k$ arises
as a completion of $\mathbb{Q}$.
The proof may then be summarized
by the sequence
\begin{align}
  q^{-N} V_N(\Psi_1,\Psi_2)
  &= \label{eqn:proof-outline-1}
    \langle \theta(z) h_{1}(q^{2 N} z),
    \theta(z) h_2(q^{2 N} z) \rangle
  \\
  &=\label{eqn:proof-outline-2}
    \langle |\theta|^2(z), \overline{h_1} h_2(q^{2 N} z) \rangle
  \\
  &=\label{eqn:proof-outline-3}
    \langle |\theta|^2, 1 \rangle \langle 1, \overline{h_1}  h_2
    \rangle
    + O(N q^{-N})
  \\
  &= \label{eqn:proof-outline-4}
    c_0 L^{(S)}(\pi,1/2)
    \int_{h \in N(H)}
      (\int_{x \in \mathbf{X}}
  \Psi_1(x h)
  \overline{\Psi_2(x)})
    + O(N q^{-N}),
\end{align}
where
$\theta$ is essentially
the weight $1/2$ Jacobi theta function,
$z$ denotes an integration parameter in the upper half-plane,
$h_1,h_2$ are \emph{fixed}
weight $3/2$ theta lifts of $\Psi_1,\Psi_2$
to some congruence cover of
$\Gamma_0(4) \backslash \mathbb{H}$,
and
inner products are taken on congruence covers of
$\Gamma_0(4) \backslash \mathbb{H}$ with respect 
to the natural probability measures.
The first step \eqref{eqn:proof-outline-1} was discussed in
\S\ref{sec:theta-funct-vari} in high level terms.
The obvious second step \eqref{eqn:proof-outline-2}
forms the cornerstone of the argument;
see \cite{MR2373356}
for related discussion.  The third step \eqref{eqn:proof-outline-3}
is a variant of the equidistribution of Hecke operators; we
have discussed it extensively in \cite{nelson-variance-73-2,
  nelson-theta-squared}.  The fourth step
\eqref{eqn:proof-outline-4} is an explication of the Rallis
inner product formula followed by the asymptotic analysis of
the integral $I_f(h)$ discussed above.


\subsection{The shape of the key inner product}
The precise shape of the RHS of
\eqref{eqn:proof-outline-1} (namely, the ``separation''
of $\theta(z)$ and $h_i(q^{2 N} z)$
by the dilation
$z \mapsto q^{2 N} z$)
is evidently crucial to the success
of the method, so we sketch how it arises.
Let $\phi = \phi_N : M_2(k) \rightarrow \mathbb{C}$
denote the Schwartz--Bruhat function
related to the kernel $f = f_N$
by
\begin{equation}
  \phi(x) :=
  \begin{cases}
    0 & \text{ if } x \notin \GL_2(k) \\
    1_{\mathfrak{o}^\times}(\det(x)) f(\pr(x)) & \text{ if } x \in \GL_2(k),
  \end{cases}
\end{equation}
where $\pr : \GL_2(k) \rightarrow \PGL_2(k)$
denotes the canonical projection.
Let
$\mathcal{F}$ denote the Fourier transform
on $M_2(k)$.  (We hope the dual use of
this symbol for Fourier transforms
and families introduces no confusion.)
Then
\begin{equation}\label{eq:key-approximation-for-fourier-transform-of-conv-kernel-sketch}
  \mathcal{F} \phi
  (
  \begin{pmatrix}
    d + a & b  \\
    c  & d - a
  \end{pmatrix}
  )
  \approx
  \begin{cases}
    1 & |d| = O(1); |b|, |c| = O(q^{N/2}); |a| \asymp q^N \\
    0 & \text{otherwise}
  \end{cases}
\end{equation}
The detailed statement and proof of the computation
\eqref{eq:key-approximation-for-fourier-transform-of-conv-kernel-sketch}
may be found in \S \ref{sec:fourier-transf-conv-kern}.  The key
features are the shape of the support and the uniform
smoothness under simultaneous dilation of the parameters
$a,b,c$.  Ignoring the less important variable $d$, one can
think of
$\mathcal{F} \phi$
as capturing roughly the Fourier transform of the pullback of
$f$ to the Lie algebra of $\PGL_2(k)$.
We note in passing that the subgroup $N(H)$ 
arises eventually from
\eqref{eq:key-approximation-for-fourier-transform-of-conv-kernel-sketch}
as the ``normalizer of the limiting support'' of
$\mathcal{F} \phi$.


Consider now the Hecke twisted pretrace formula
\begin{equation}\label{eq:hecke-twisted-pretrace}
  \sum_{\gamma \in M_n} f(x^{-1} \gamma y) =\sum_{\varphi \in
    \mathcal{F}_N} \sqrt{n} \lambda_\varphi(n) \overline{\varphi
    (x)} \varphi(y),
\end{equation}
where $M_n$ is as in the classical definition of the Hecke operator
$T_n$ and $\sqrt{n} \lambda_{\varphi}(n)$ denotes the Hecke eigenvalue.
Let $R \subseteq B$ denote the maximal order underlying
the construction of $\Gamma$.
Taking $x=y=:g$ in \eqref{eq:hecke-twisted-pretrace}
and summing against $e(n z)$ over positive integers $n$ having only ``good''
prime divisors gives
\[
\sum_{\gamma \in R}
\phi(g^{-1} \gamma g)
e(\det(\gamma) z)
=
\sum_{\varphi \in \mathcal{F}_N}
|\varphi|^2(g)
\Phi_\varphi(z)
\]
where $e(z) := e^{2 \pi i z}$ and
$\Phi_\varphi(z) := \sum \sqrt{n} \lambda_\varphi(n) e(n z)$
denotes an Eichler/Jacquet--Langlands lift of $\varphi$.
Suppose henceforth that $k = \mathbb{Q}_p$.
By the inversion formula for theta functions,
we obtain
\begin{align*}
  \sum_{\varphi \in \mathcal{F}_N}
  |\varphi|^2(g)
  \Phi_\varphi(-1/z)
  &\approx
    \sum_{\gamma \in R[1/p]}
    \mathcal{F} \phi(g^{-1} \gamma g)
    e(\det(\gamma) z) \\
  &\approx
    \sum_{m,\beta}
    \mathcal{F} \phi(m + g^{-1} \beta g)
    e(m^2 z)
    e(\det(\beta) z),
\end{align*}
where the sum is over $m \in \mathbb{Z}[1/p]$
and $\beta \in R[1/p]^0$.
(Here $\approx$ means ``up to unimportant
inaccuracies.'')
By integrating against $\Psi(g)$, we obtain
\begin{equation}\label{eq:key-sketch-identity-pre-parseval}
  \sum_{\varphi \in \mathcal{F}_N}
\Phi_{\varphi}(-1/z)
\langle \varphi, \Psi \varphi  \rangle
= \sum_{m,\gamma}
e(m^2 z) e(\det(\beta) z)
\int_{g \in \Gamma \backslash G}
\Psi(g) \mathcal{F} \phi(m + g^{-1} \beta g).
\end{equation}
The multiplicity one theorem on $\mathbf{X}$ implies that for
$\varphi, \varphi ' \in \mathcal{F}_N$,
\begin{equation}\label{eq:consequence-of-mult-one-in-sketch}
  \langle \Phi_\varphi, \Phi_{\varphi '} \rangle
  \approx \iota_{\varphi} \text{ if } \varphi = \varphi ',
  \quad = 0 \text{ otherwise}.
\end{equation}
By 
\eqref{eq:key-approximation-for-fourier-transform-of-conv-kernel-sketch}
and the assumption that $\Psi$ is fixed,
one deduces that
\begin{equation}\label{eqn:key-observation-for-proof-outline-1}
  \int_{g \in \Gamma \backslash G}
  \Psi(g) \mathcal{F} \phi(m + g^{-1} \beta g)
  \approx 1_{\mathbb{Z}_p}(m)
  \phi_{\Psi}''(p^N \beta)
\end{equation}
where $\phi_{\Psi}''$ is independent of $N$.  (A ``toy version''
of such reasoning: if a function
 on $\mathbb{R}^3$ is smooth with respect to polar
coordinates and supported away from the origin, then it is smooth.)
By Parseval applied to
\eqref{eq:key-sketch-identity-pre-parseval},
\eqref{eq:consequence-of-mult-one-in-sketch} and
\eqref{eqn:key-observation-for-proof-outline-1}, one arrives at
an identity of the form \eqref{eqn:proof-outline-1}.







\subsection{Organization of this paper}
\label{sec-1-7}
The main result is Theorem
\ref{thm:main-result-for-microlocal-stuff}, stated
somewhat informally above
and more precisely (and generally) as
Theorem \ref{thm:main-variance-microlocal-adelic-formulation} in
the final section \S\ref{sec:deduction-main-thm-microlocal} of
this paper, which also contains the proof.

To present the proof as clearly as possible, we separate
the difficulties concerning general
families from those specific to the $\mathcal{F}_N$
considered above.
The former may be found
in Part I, the latter in Part II.
The two parts are weakly coupled.

The main result of Part I is
Theorem \ref{thm:main-estimate-general-variance},
stated in
\S\ref{sec-4}.
Its proof
involves local (\S\ref{sec-2})
and global (\S\ref{sec-3})
preliminaries.
Its conclusion reduces the quantum variance problem to
local problems of three sorts, which are treated
 in Part II for the
families $\mathcal{F}_N$:
\begin{enumerate}
\item[(\S\ref{sec:appl-micro-local-prelims})]
  The comparatively
  mild difficulties associated with producing a convolution kernel $f$
  that picks off a given family $\mathcal{F}$,
  as in
  \eqref{eq:intro-pretrace-fromula}.
\item[(\S\ref{sec:local-estimates-main-statements})]
  Those
  associated with determining the main term
  (i.e., asymptotically evaluating distributions
  such as the $I_f$ considered in
  \S\ref{sec-1-6}).
\item[(\S\ref{sec:local-estimates-error})]
  Those concerned with
  bounding error terms (i.e., differences between inner products
  as in \eqref{eq:ip-4-theta-before-rearr} and
  \eqref{eq:ip-4-theta-after-rearr}).  An important point is
  that Theorem \ref{thm:main-estimate-general-variance} reduces
  this global problem to a local one.
\end{enumerate}
For the latter two problems, the key input is a careful analysis
(\S\ref{sec:fourier-transf-conv-kern})
of the ``Fourier transform'' of the convolution kernel $f$.

The reader might first study carefully
\S\ref{sec:general-notation},
\S\ref{sec-4-1}, \S\ref{sec:main-general-estimate-key-defns},
\S\ref{sec-4-3},
\S\ref{sec:general-estimates-specialized-single-place},
\S\ref{sec:element-attached-to-N-sigma},
\S\ref{sec-7-1},
\S\ref{sec-8-2},
\S\ref{sec:9-setting}
and
\S\ref{sec:variance-statistics}, skipping or skimming any
sections before or between; we have made some effort to keep
those sections essentially self-contained.

\subsection{General notation\label{sec:general-notation}}
\label{sec-1-8}
For a quaternion algebra $B$ over a field or adele ring $A$, 
we denote by $\iota : B \rightarrow B$ the main involution, by
$\nr : B \rightarrow A$ the reduced norm
$\nr(x) := x x^{\iota}$, by $\tr : B \rightarrow A$ the reduced
trace $\tr(x) := x + x^{\iota}$, and by
\[B^0 := \{x \in B : \tr(x) = 0\}\] the subspace of traceless
quaternions.
We employ the notations
\[
n(x) := \begin{pmatrix}
  1 & x \\
  0 & 1
\end{pmatrix},
\quad
a(y) :=
\begin{pmatrix}
  y & 0 \\
  0 & 1
\end{pmatrix},
\quad
t(y) :=
\begin{pmatrix}
  y & 0 \\
  0 & y^{-1}
\end{pmatrix},
\quad
n'(x) :=
\begin{pmatrix}
  1 & 0 \\
  x & 1
\end{pmatrix}
\]
\[
\Ad(g) x := g x g^{-1},
\quad
\Ad(g) f(x) := f(g^{-1} x g)
\]
and
\[
\mathfrak{S} f(x) := \frac{f(x) + f(x - \tr(x))}{2}
\]
whenever they make sense.
For example, this is the case if 
$g$ belongs to the unit group $B^\times$ of a quaternion algebra $B$ over $A$ as above
and $x$ belongs to (resp. $f$ is a function on)
one of the sets $B^\times/A^\times, B, B^0$.
(As
we explain below,
one may interpret $\mathfrak{S}$ as ``projection onto $\O_1$-invariants.'')

We define the right regular representation
$\rho_{\reg}(g) f(x) := f(x g)$
whenever it makes sense.

Given a local (resp. global) field $F$ and a nontrivial unitary
character $\psi$ of $F$ (resp. of $\mathbb{A}/F$)
and an element $a \in F^\times$,
we denote by $\psi^a$ the nontrivial unitary character
with the same domain as $\psi$
given by
$\psi^a(x) := \psi(a x)$.

For a finite-dimensional vector space $V$ over a local field or
adele ring, we denote by $\mathcal{S}(V)$ the space of
Schwartz--Bruhat functions $\phi : V \rightarrow \mathbb{C}$,
topologized as usual
(see e.g. \cite[\S11]{MR0165033}).


Let $G$ be a group over an adele ring or a finite product of
local fields.
We let $C_c^\infty(G)$ denote the space of smooth compactly supported
functions;
as usual, smooth
means infinitely differentiable
(resp. locally constant)
with respect to the archimedean
(resp. non-archimedean) variables.
Assume that we have equipped $G$ with a Haar
measure.  Let $\pi$ be a smooth representation of a group that
contains $G$.
Let $f \in C_c^\infty(G)$.  We then define the operator
$\pi(f) \in \End(\pi)$ by
$\pi(f) v := \int_{g \in G} f(g) \pi(g) v$.

The use of Vinogradov notation is standard: $A = O(B)$,
$A \ll B$ and $B \gg A$ each signify that $|A| \leq c |B|$ for
some ``constant'' $c$, with dependencies indicated by subscripts;
$A \asymp B$ signifies that $A \ll B \ll A$.

We write $1_E$ for the characteristic function of a subset $E$
of some set $X$.
For an assertion $A$,
we set $1_A := 1$ if $A$ is true and $1_A := 0$ if $A$ is false.

We set $\mathbb{C}^{(1)} := \{ z \in \mathbb{C}^{\times} : |z| =
1\}$.

We adopt the convention that main results (Theorems
\ref{thm:main-result-for-microlocal-stuff},
\ref{thm:main-estimate-general-variance},
\ref{thm:main-variance-microlocal-adelic-formulation} in
\S\ref{sec-1-5}, \S\ref{sec-4-3},
\S\ref{sec:variance-statistics}, respectively) are called
theorems, the most important intermediary results original to
this article (Propositions
\ref{prop:main-error-estimate-global-adelic-general},
\ref{prop:after-extracting-main-term},
\ref{prop:harmonic-analytic-isolation-1},
\ref{prop:key-fourier-estimate-microlocal-kernel},
\ref{prop:desired-main-term-identity-do-it-up},
\ref{prop:local-error-estimates-stmt} in \S\ref{sec-3-5-5},
\S\ref{sec:main-identity},
\S\ref{sec:element-attached-to-N-sigma}, \S\ref{sec-6-3},
\S\ref{sec-7-1}, \S\ref{sec-8-2}) are called
propositions, and everything else (including deep cited work) is
called a lemma.

\newpage
\part{Quantum variance and theta functions}


\section{Local preliminaries}
\label{sec-2}
The purpose of this section is collect local definitions,
notation and identities for later use.  The notation introduced
here should be self-descriptive with the exception of that for
the similitude Weil representation $\Omega$ defined in
\S\ref{sec:defn-local-omega}.

Let $k$ be a local field of characteristic $\neq 2$, thus $k$ is either $\mathbb{R}$,
$\mathbb{C}$ or a finite extension of $\mathbb{Q}_p$ or
(for $p \neq 2$) of
$\mathbb{F}_p(t)$.
The assumption on the characteristic
is relevant only for sections discussing the Weil representation.

Let
$\psi : k \rightarrow \mathbb{C}^{(1)}$ be a nontrivial unitary
character of $k$,
and let $B$ be a quaternion algebra over $k$.
Set $G := B^\times/ k^\times$.
When $k$ is non-archimedean, let $R \subset B$ be a maximal
order.

\subsection{Generalities}
\label{sec-2-1}
\subsubsection{The number field}
\label{sec-2-1-1}
We denote by $|.| := |.|_k : k \rightarrow \mathbb{R}_{\geq 0}$
the normalized absolute value,
so that $d(c x) = |c| \, d x$ for $c \in k$ and any Haar measure $d x$ on $k$.


Let $\zeta_k(s)$ denote the local zeta function,
thus $\zeta_k(s) = \pi^{-s/2} \Gamma(s/2),
2 (2 \pi)^{-s} \Gamma(s)$ or $(1-q^{-s})^{-1}$
as $k = \mathbb{R}, \mathbb{C}$,
or a non-archimedean local field
with residue field of cardinality $q$.

Recall that $B$ is \emph{split} if it is isomorphic to the algebra $M_2(k)$ of $2 \times 2$
matrices.  Otherwise, $B$ is called \emph{non-split} or \emph{ramified}; in that case, it is the unique (up to isomorphism)
quaternion division algebra over $k$, and the group $G$ is compact.

When $k$ is non-archimedean,
we denote by
$\mathfrak{o}$ the maximal order, by
$\mathfrak{q}$ the maximal ideal,
by
$q := \# \mathfrak{o} / \mathfrak{q}$
the cardinality of the residue field,
by $\varpi \in \mathfrak{q} = \varpi \mathfrak{o}$
a uniformizer (thus $|\varpi| = q^{-1}$),
and by $\Delta_{\psi}$
the absolute conductor
of
$\psi : k \rightarrow \mathbb{C}^{(1)}$,
thus $\Delta_{\psi} = q^d$
if $\psi$ is trivial on
$\mathfrak{q}^{-d}$
but not
on $\mathfrak{q}^{-d-1}$.
Recall that $\psi$ is \emph{unramified} if $\Delta_{\psi} = 1$.

\subsubsection{Measures\label{sec:local-measures}}
\label{sec-2-1-2}
For $X \in \{k,B^0,B\}$,
define the
perfect pairing $\langle , \rangle : X \otimes X \rightarrow k$
by $\langle x,y \rangle := x y$ if $X = k$
and by $\langle x, y \rangle := \tr(x^{\iota} y)$ if $X = B^0,B$.
Equip $X$ with the Haar measure $d x$
for which the Fourier transform
$\mathcal{F} : \mathcal{S}(X) \rightarrow \mathcal{S}(X)$
defined by
$\mathcal{F} f(\xi)
:=
\int_{x \in X}
f(x)
\psi(\langle x, \xi  \rangle)
\, d x$
satisfies the inversion formula
\begin{equation}\label{eq:fourier-inversion-for-normalization}
  \mathcal{F} \mathcal{F} f(x) = f(-x).
\end{equation}
Equip $k^\times$ with the Haar measure
$\int_{k^\times } f
:=
\int_{x \in k^\times}
f(x) \, \frac{d x}{|x|}$.

The quotient $k^\times / k^{\times 2}$ is finite.
We equip it with the Haar measure
compatible with the squaring map,
so that for $f \in C_c(k^\times)$,
\begin{equation}\label{eqn:compatibility-squaring-map-local-measure}
\int_{x \in k^\times}
f(x)
\,
\frac{d x}{|x|}
= 
\int_{y \in k^\times/k^{\times 2}}
(\int_{z \in k^\times}
f(y z^2) \, \frac{d z}{|z|}).
\end{equation}
For $f : k^\times / k^{\times 2} \rightarrow
\mathbb{C}$,
one has explicitly
$\int_{x \in k^\times  / k^{\times 2}}
f(x)
=
\frac{|2|_k}{2}
\sum_{x \in k^\times / k^{\times 2}}
f(x)$.

Equip $G$ with the Haar
measure $d g$ for which the integral formula
\begin{equation}\label{eq:integarl-formula-for-integrating-over-B}
  \int_{x \in B}
  f(x) \, d x
  = \int_{g \in G}
  (\int_{z \in k^\times}
  |\nr(z g)|^2 f(z g) \, \frac{d z}{|z|}) \, d g
\end{equation}
holds for $f \in C_c(B)$.
When $B = M_2(k)$
is split, so that $G = \PGL_2(k)$,
a direct calculation
with differential forms
gives for $f \in C_c(G)$
that
\begin{equation}\label{eqn:explicit-M2-integral-formula}
  \int_G f
=
\int_{x_1,x_2 \in k}
\int_{y \in k^\times}
f(
n'(x_1)
n(x_2)
a(y)
)
\, d x_1 \, d x_2 \, \frac{d y}{|y|}.
\end{equation}

\subsubsection{Volume formulas\label{sec:local-vol-formulas}}
\label{sec-2-1-3}
Assume (for all but the final assertion of \S\ref{sec-2-1-3})
that $k$ is non-archimedean.  Write $\vol(E \subseteq X)$ to
denote the volume of $E$ with respect to the measure that we
have defined on $X$.  Let $J \leq G$ denote the image of
$R^\times$; if $B$ is split, then $J$ is a maximal compact
subgroup of $G$, otherwise it has index $2$ in the compact group
$G$.  Abbreviate
$\vol(\mathfrak{o}) := \vol(\mathfrak{o} \subseteq k)$,
$\vol(\mathfrak{o}^\times ) := \vol(\mathfrak{o}^\times
\subseteq k^\times )$,
$\vol(J) := \vol(J \subseteq G)$, 
$\vol(R) := \vol(R \subseteq B)$ and
$\Delta := \Delta_{\psi}$.  Let $\Delta_{B}$ denote the reduced
discriminant, thus $\Delta_{B_\mathfrak{p}} = 1$ or $q$
according as $B$ splits or ramifies.  Set
$\zeta_B(s) := \zeta_k(2 s) \zeta_k(2 s - 1)$ if $B$ splits and
$\zeta_B(s) := \zeta_k(2 s)$ otherwise.
\begin{lemma*}\label{lem:local-vol-formulas}
  ~\begin{enumerate}
  \item[(i)]
    $\vol(\mathfrak{o}) = \Delta^{-1/2}$,
    $\vol(\mathfrak{o}^\times)
    = \zeta_k(1)^{-1} \Delta^{-1/2}$.
  \item[(ii)]
    $\vol(R) = \Delta_B^{-1} \Delta^{-4/2}$,
    $\vol(J) = \zeta_k(1) \zeta_B(1)^{-1} \Delta_B^{-1}  \Delta^{-3/2}$.
  \item[(iii)]
    If $B$ is split, then
    \[\frac{\vol(R)}{\vol(J) \Delta^{-1/2}}
    = \zeta_k(2).\]
  \item[(iv)] If $k$ is real,
    $B$ is non-split
    and $\psi(x) = e^{2 \pi i x}$,
    then $\vol(G) = 4 \pi^2$.
  \end{enumerate}
\end{lemma*}
\begin{proof}
  For (i)---(iii),
  we may reduce by dimensional analysis
  to the case $\Delta = 1$.
  The required formulas
  then follow
  from \eqref{eqn:compatibility-squaring-map-local-measure}
  applied to $f = 1_\mathfrak{o}$ or $f = 1_R$
  and
  by \eqref{eq:integarl-formula-for-integrating-over-B}
  applied to $f = 1_{1 + \varpi R}$
  (see \cite[Lem 2.4.3]{MR580949} for details).
  For (iv), set $f(x) := e^{- 2 \pi \nr(x)}$. Apply
  \eqref{eq:fourier-inversion-for-normalization} to see that
  $\int_B f = 1$.
  Apply
  \eqref{eq:integarl-formula-for-integrating-over-B} and the
  substitution $z \mapsto z / (2 \pi \nr(g))^{1/2}$ to deduce
  that
  $(2 \pi)^2 = \vol(G) \int_{z \in \mathbb{R}^\times} |z|^4
  e^{-|z|^2} \, \frac{d z}{|z|} = \vol(G)$.
\end{proof}  

\subsubsection{Cartan
  decomposition}\label{sec:cartan-decomposition}
Suppose $B = M_2(k)$, so that $G = \PGL_2(k)$.  Let $K \leq G$
be the standard maximal compact subgroup.
Then
$G = \cup_{y \in k^\times : |y| \leq 1} K a(y) K$.
When $k$ is
non-archimedean,
one has
for $f \in C_c(K \backslash G / K)$ the integral
formula
\begin{equation}\label{eq:cartan-decomp-integral-formula}
  \int_{G}
  f
  =
  \vol(K)
  \sum_{m \geq 0}
  q^m
  (1 + 1_{m > 0} q^{-1})
  f(a(\varpi^m)).
\end{equation}

\subsubsection{The $\Xi$-function\label{sec:local-Xi}}
\label{sec-2-1-4}
Given a maximal compact subgroup $K \leq G$, let
$\Xi : G \rightarrow \mathbb{R}_{>0}$ denote the Harish--Chandra
function relative to $K$:
\begin{itemize}
\item   If $B$ is non-split, then
$\Xi \equiv 1$.  
\item If $B$ is split, then $\Xi(g) = \langle g v, v \rangle$
  where $v$ is a $K$-invariant unit vector in the unitary
  induction of the trivial character of a Borel subgroup of $G$
  (see \cite{MR946351}).
\end{itemize}
The following properties of $\Xi$
are relevant for us:
\begin{enumerate}
\item It satisfies $\Xi(1) = 1$, and is bi-$K$-invariant.
\item If $B$ is split, then under any identification $G = \PGL_2(k)$,
  one has
  $\Xi(a(y)) \asymp \log(t)/t^{1/2}$ with $t := |y| + |y|^{-1}$.
\item
  Let $\pi$ be an irreducible unitary
  representation of $G$.
  If $B$ is split, assume that $\dim(\pi) > 1$.
  Then there exists $\delta > 0$
  so that for $v_1, v_2 \in \pi$,
  one has $\langle g v_1, v_2  \rangle
  \ll_{v_1,v_2} \Xi(g)^{\delta}$
  for
  all $g \in G$.
  (See for instance \cite[\S2.5.1]{michel-2009}
  for a more precise assertion).
\end{enumerate}

\subsubsection{Convergence
  lemmas\label{sec:local-convergence-lemmas}}
The estimates collected here are standard.
\label{sec-2-1-5}
\begin{lemma}\label{lemma:cheap-matrix-coeff-schwartz-space-B-estimate-via-Xi}
  Either let $X$ be one of the spaces $B^0, B$
  and take $\phi_1, \phi_2 \in \mathcal{S}(X)$,
  or let $X = G$ and take $\phi_1,\phi_2 \in C_c^\infty(G)$.
  For $g \in G$,
  one then has
  \[\langle \Ad(g) \phi_1, \phi_2 \rangle_{L^2(X)}
    \ll_{\phi_1,\phi_2} \Xi(g)^2.\]
\end{lemma}
\begin{proof}
  We treat the case $X = B^0$; the other cases are similar.
  The estimate is trivial unless $B$ is split,
  so assume
  that $B = M_2(k)$.
  By
  the Cartan decomposition, we may assume that $g = a(y)$
  with $|y| \leq 1$.  It suffices then (ignoring logarithmic
  factors)
  to show that
  \[
    \int_{a,b,c \in k}
    \overline{\phi_1}
    (
    \begin{pmatrix}
      a & y^{-1} b \\
      y c & -a
    \end{pmatrix}
    )
    {\phi_2}
    (
    \begin{pmatrix}
      a &  b \\
      c & -a
    \end{pmatrix}
    )
    \ll_{\phi_1,\phi_2} |y|.
  \]
  For this,
  we
  substitute
  $b \mapsto y b$
  and appeal to the rapid decay of $\phi_1,\phi_2$.
\end{proof}

\begin{lemma}\label{lemma:convergence-Xi-along-G-and-H}
  Let $\delta > 0$.
  \begin{enumerate}
  \item The integral
    $\int_{g \in G} \Xi^{2+\delta}(g)$
    converges.
  \item   Let $E$ be a separable quadratic
    subalgebra of $B$.
    Let $H \leq G$ denote the image of $E^\times$.
    Equip $H$ with some Haar measure.
    Then
    the integral
    $\int_{h \in H} \Xi^\delta(h)$
    converges.
  \end{enumerate}
\end{lemma}
\begin{proof}
  We may assume that $B = M_2(k)$ is split and that $E$ is the
  split diagonal torus, as otherwise the groups involved are
  compact.  The convergence of the second integral then
  follows from that of
  $\int_{y \in k^\times} \min(|y|,|y|^{-1})^\delta$.  For the
  first integral, we integrate using the Cartan decomposition;
  the convergence follows similarly.
\end{proof}

\subsubsection{Conventions}
By a \emph{representation} of a $k$-group,
we shall always mean
\begin{itemize}
\item a smooth representation, if $k$ is non-archimedean, and otherwise
\item the space of smooth vectors in a unitary representation.
\end{itemize}

\subsection{Weil representations\label{sec:local-weil-reps}}
\label{sec-2-2}

\subsubsection{Quadratic spaces\label{sec:local-quadratic-spaces}}
\label{sec-2-2-1}

Let $V$ be a quadratic space over $k$, thus $V$ is a
finite-dimensional $k$-vector space equipped with a
non-degenerate quadratic form $q_V : V \times V \rightarrow k$.
We denote by $b_V : V \otimes V \rightarrow k$
the associated non-degenerate bilinear form  given by $b_V(x,y)
:= q_V(x+y) - q_V(x) - q_V(y)$.

Recall that
$\O(V) := \{g \in \GL(V) : q_V(g x) = q_V(x) \text{ for all } x
\in V\}$,
$\GO(V) := \{g \in \GL(V) : \text{ there exists
  $\lambda \in k^\times$ so that } q_V(g x) = \lambda q_V(x)
\text{ for all } x \in V\}$,
and $\SO(V) := \SL(V) \cap \O(V)$.  The group $\GO(V)$ contains
the subgroup
$k^\times$ of scalar operators, and we set
$\PGO(V) := \GO(V) / k^\times$.

Let $\mu_V$ denote the measure on $V$
that is $(\psi,b_V)$-self dual,
i.e., 
that for which
$\mathcal{F} : \mathcal{S}(V) \rightarrow \mathcal{S}(V)$
defined by
$\mathcal{F} \phi(\xi)
:=
\int_{x \in V}
\phi(x)
\psi(b_V(x,\xi))
\, d \mu_V(x)$
satisfies
$\mathcal{F} \mathcal{F} \phi(x) = \phi(-x)$.

The following examples of quadratic spaces are relevant for us:
\begin{enumerate}
\item $V = B$,
  $q_V = \nr$,
  so that $b_V(x,\xi) = \tr(x^{\iota} \xi) = \langle x,\xi
  \rangle$.
\item $V = B^0$,
  $q_V$ the restriction of $\nr$.
  The natural map $\Ad : G \rightarrow \SO(B^0)$
  is an isomorphism.
\item $V = k$, regarded as a subspace of $B$, and $q_V$ the
  restriction of $\nr$, thus $q_V(x) = x^2$ and
  $b_V(x,y) = 2 x y$ for $x \in V$.
  In this case,
  we denote the orthogonal group by $\O_1(k) := \O(V) \cong \{\pm 1\}$.
\end{enumerate}
For $V = B, B_0$,
the measure $d \mu_V(x)$ coincides
with
$d x$ as defined in \S\ref{sec:local-measures}.

\subsubsection{Metaplectic group}
\label{sec-2-2-2}
Let $\Mp_2(k)$ denote the metaplectic double cover of
$\SL_2(k)$.
It is convenient
to identify $\Mp_2(k)$ with $\SL_2(k) \times \mu_2$,
where
$\mu_2 := \{\pm 1\}$,
with the
group law given by
$(s_1,\zeta_1) (s_2,\zeta_2) = (s_1 s_2, \zeta_1 \zeta_2
c(s_1,s_2))$
for a cocycle
$c : \SL_2(k) \times \SL_2(k) \rightarrow \{\pm 1\}$
as in \cite[p.19]{MR0424695} or \cite[\S4.4]{nelson-theta-squared}.
Thus $\mu_2$ is a central subgroup of $\Mp_2(k)$, and one has a
short exact sequence
$1 \rightarrow \mu_2 \rightarrow \Mp_2(k) \xrightarrow{\pr}
\SL_2(k) \rightarrow 1$.
\subsubsection{Weil representation\label{sec:local-weil-repn}}
\label{sec-2-2-3}
For a quadratic space $V$,
one has the Weil representation
\cite{MR0165033} on the Schwartz--Bruhat space $\mathcal{S}(V)$:
\[\rho_{\Weil}^{\psi,V} : \Mp_2(k)
\times \O(V) \rightarrow \GL(\mathcal{S}(V)).\]
This representation is continuous \cite[\S39]{MR0165033}
for the standard topologies on all spaces involved
and
extends to a unitary representation on $L^2(V) := L^2(V,\mu_V)$.

For the remainder of \S\ref{sec:local-weil-repn},
abbreviate $\rho := \rho_{\Weil}^{\psi,V}$.
For $s \in \Mp_2(k)$
or $g \in \O(V)$
we abbreviate
$\rho(s) := \rho(s,1)$
and
$\rho(g) := \rho(1,g)$;
one then has
$\rho(s)
\rho(g)
=
\rho(g)
\rho(s)$.

Elements $\zeta$ of the central subgroup $\mu_2$
of $\Mp_2(k)$ act by the
scalar operators
$\rho(\zeta) = (-1)^{\dim(V)}$,
so $\rho$ factors through $\SL_2(k)$ if and only if
$\dim(V)$ is even.

There is a quartic character
$\chi_{\psi,V} : k^\times \rightarrow \mathbb{C}^{(1)}$
and an eighth root of unity $\gamma_{\psi,V} \in \mathbb{C}^{(1)}$
so that,
abbreviating $n(b) := (n(b),1),
t(a) := (t(a),1), w := (\begin{pmatrix}
  & 1 \\
  -1 & 
\end{pmatrix}, 1) \in \Mp_2(k)$
for
$b \in k$, $a \in k^\times$,
one has for
$\phi \in \mathcal{S}(V)$,
$x \in V$
that
\begin{align*}
  \rho(n(b))
  \phi(x)
  &= 
    \psi(b q_V(x)) \phi(x),
  \\
  \rho(t(a))
  \phi(x)
  &= 
    \chi_{\psi,V}(a) |a|^{\dim(V)/2} \phi(a x),
  \\
  \rho(w) \phi(x)
  &=
    \gamma_{\psi,V} \mathcal{F} \phi(x).
\end{align*}
If $V = M_2(k)$, then $\chi_{\psi,V}$ is trivial and
$\gamma_{\psi,V} = 1$.

Elements $g$ of the orthogonal group $\O(V)$
act by
$\rho(g) \phi(v) := \phi(g^{-1} v)$.
Suppose that $V = B^0$,
so that $\Ad : G \xrightarrow{\cong} \SO(B^0)$.
For $g \in G$
and $\phi \in \mathcal{S}(B^0)$,
the function $\Ad(g) \phi$
as defined in \S\ref{sec:general-notation}
agrees with $\rho(\Ad(g)) \phi$:
both send $x \in B^0$
to $\phi(g^{-1} x g)$.
\subsubsection{Factorization\label{sec:factorization-weil-repn}}
\label{sec-2-2-4}
Let $V$ be a quadratic space that admits an orthogonal
decomposition $V = V' \oplus V''$
as a sum of two quadratic spaces.
(The relevant example is when $V = B, V' = k, V'' = B^0$.)

Recall the dense inclusion
$\mathcal{S}(V') \otimes \mathcal{S}(V'') \rightarrow
\mathcal{S}(V)$ obtained by identifying $\phi ' \otimes  \phi '' \in
\mathcal{S}(V') \otimes  \mathcal{S}(V'')$
with the function
$V' \oplus V'' \ni \alpha ' + \alpha '' \mapsto \phi '(\alpha ') \phi '' (\alpha
'')$.

Given continuous linear operators $T, T', T''$ on
$\mathcal{S}(V), \mathcal{S}(V'), \mathcal{S}(V'')$,
respectively, write $T = T' \otimes T''$ to denote that
$T (\phi ' \otimes \phi '') = T' \phi ' \otimes T'' \phi ''$
for all $\phi ' \in \mathcal{S}(V'), \phi ''  \in
\mathcal{S}(V'')$.
In this sense, one has
$\rho_{\Weil}^{V,\psi}(s) = \rho_{\Weil}^{V',\psi}(s) \otimes
\rho_{\Weil}^{V'',\psi}(s)$ for all $s \in
\Mp_2(k)$.

We denote by
$1 \otimes \rho_{\Weil}^{V'',\psi}(s)$
the operator
on $\mathcal{S}(V)$
sending
$\phi ' \otimes \phi ''$
to $\phi ' \otimes \rho_{\Weil}^{V'',\psi}(s) \phi ''$.

\subsubsection{Extension to similitudes\label{sec:defn-local-omega}}
\label{sec-2-2-5}
The following definitions
were inspired by \cite[I.3]{MR783511}.
Let $\Omega$
denote the space of functions
$\phi : k^\times \times
B \rightarrow \mathbb{C}$ satisfying the conditions:
\begin{itemize}
\item
  For each $t \in k^\times$,
  the function $\phi[t] : B
  \rightarrow \mathbb{C}$ given by
  $\phi[t](x) := \phi(t,x)$
  belongs to the Schwartz--Bruhat space
  $\mathcal{S}(B)$.
\item One has $\phi(z^2 t, x) = \phi(t, z x)$ for all $t,z \in
  k^\times$,
  $x \in B$.
\end{itemize}
Let
$\rho_{\Weil} : \PGL_2(k) \times \PGO(B)
\rightarrow \GL(\Omega)$
denote the representation
characterized
by the identities:
for $s \in \SL_2(k), y \in k^\times, g \in \GO(B)$,
\begin{align*}
  (\rho_{\Weil}(s) \phi)[t]
  &= \rho_{\Weil}^{\psi^t,B}(s) (\phi[t]),
  \\
  (\rho_{\Weil}(a(y)) \phi)[t]
      &=
        |y| \phi[t y],
  \\
  (\rho_{\Weil}(g) \phi)(t,x)
  &=
    \phi(\lambda(g) t, g^{-1} x)
\end{align*}
where  $\lambda : \GO(B) \rightarrow k^\times$
denotes the similitude factor.

\begin{remark*}
  More generally, if $V$ is any even-dimensional quadratic
  space, then the representation $\rho_{\Weil}^{\psi,V}$ factors
  through $\SL_2(k) \times \O(V)$.  One can induce it to a
  representation of $\GL_2(k) \times \GO(V)$ on
  $\mathcal{S}(k^\times \times V)$, whose isomorphism class is
  independent of $\psi$.  By taking coinvariants for the action
  by the center, one arrives at a representation of
  $\PGL_2(k) \times \PGO(V)$.  In the relevant case that
  $V = B$, the representation obtained in that way is realized
  by $\Omega$.  Our global discussion concerns the restriction
  of $\Omega$ to $\SL_2(k) \times \O_1(k) \times \O(B^0)$,
  which embeds as the ``even subspace'' of
  $\oplus_{\tau \in k^{\times} / k^{\times 2}}
  \rho_{\Weil}^{\psi^{\tau},F} \otimes
  \rho_{\Weil}^{\psi^{\tau},B^0}$.
\end{remark*}



Equip $\Omega$
with the invariant hermitian norm $\|.\|_{\Omega}$ given by
\begin{equation}\label{eqn:inner-product-on-Omega-1}
  \|\phi\|^2_{\Omega}
  :=
  \int_{t \in k^\times / k^{\times 2}}
  |t|^2
  \int_{x \in B}
  |\phi|^2(t,x),
\end{equation}
or equivalently (by \eqref{eq:integarl-formula-for-integrating-over-B}, \eqref{eqn:compatibility-squaring-map-local-measure}),
\begin{equation}\label{eqn:inner-product-on-Omega-2}
  \|\phi\|^2_{\Omega}
  = 
  \int_{g \in G}
  |\nr(g)|^2
  \int_{t \in k^\times}
  |t|^2
  |\phi|^2(t,g)
  \, \frac{d t}{|t|} \, d g.
\end{equation}

Define
$\mathfrak{S} : \Omega \rightarrow \Omega$
and
$\Ad(g) : \Omega \rightarrow \Omega$
($g \in G$)
by
applying the general definition (\S\ref{sec:general-notation})
to the second coordinate,
so that for $\phi \in \Omega$
and $(t,x) \in k^\times \times B$,
one has
$(\mathfrak{S} \phi)[t] = \mathfrak{S} (\phi[t])$,
$\mathfrak{S} \phi(t,x) =
(\phi(t,x) + \phi(t,\tr(x)- x))/2$,
$\Ad(g) \phi = \rho_{\Weil}(\Ad(g)) \phi$, 
$(\Ad(g) \phi)[t] = \Ad(g) (\phi[t])$,
$\Ad(g) \phi(t,x) =
\phi(t,g^{-1} x g)$.


\subsubsection{The distinguished element}\label{sec:dist-elem}
Suppose temporarily that $k$ is non-archimedean
and that $B \cong M_2(k)$ is split;
similar considerations apply to non-split $B$,
but we do not need them.
The \emph{distinguished element} $\phi^0 \in \Omega$ (with respect to the chosen maximal order $R \subset B$)
is then defined
by
\begin{equation}
  \phi^0(t,x)
  :=
  \frac{
    \int_{z \in k^\times}
    1_{R}(z x)
    1_{\mathfrak{o}^\times}(z^{-2} t)
    \, \frac{d z}{|z|}
  }
  {
    \int_{z \in k^\times}
    1_{\mathfrak{o}^\times}(z)
    \, \frac{d z}{|z|}
  }.
\end{equation}
Note that $\phi^0$ takes values in $\{0,1\}$.
One verifies directly that $\phi^0$ is $\PGL_2(\mathfrak{o})
\times K'$-invariant,
where $K' \leq \PGO(B)$
denotes the image of the $\O(B)$-stabilizer of $R$.
Moreover,
$\mathfrak{S} \phi^0 = \phi^0$. 

\begin{lemma}\label{lem:norm-of-distinguished-vector-in-Omega-local}
  $\|\phi^0\|^2_{\Omega}
  = \vol(R)$.
\end{lemma}
\begin{proof}
  Since $\phi$ takes values in $\{0 ,1\}$,
  one has
  $\|\phi^0\|^2_{\Omega}
  =
  \int_{t \in k^\times / k^{\times 2}}
  |t|^2
  \int_{x \in B}
  \phi^0(t,x)$.
  By expanding the definition
  of $\phi^0$ and using that $\int_{x \in B} 1_R(z x)
  = |z|^{-4} \vol(R)$
  and that
  $|t|^2 |z|^{-4} 1_{\mathfrak{o}^\times}(z^{-2} t)
  = 1_{\mathfrak{o}^\times}(z^{-2} t)$,
  our task
  reduces
  to showing that
  $\int_{t \in k^\times / k^{\times 2}}
  \int_{z \in k^\times} 
  1_{\mathfrak{o}^\times}(z^{-2} t)
  \, \frac{d z}{|z|}
  =
  \int_{z \in k^\times} 
  1_{\mathfrak{o}^\times}(z^{-2} t)
  \, \frac{d z}{|z|}$,
  as follows from \eqref{eqn:compatibility-squaring-map-local-measure}.
\end{proof}

Let $K \leq G$
denote the image of $R^\times$.
We may then fix an identification
$B = M_2(k)$
under which $G = \PGL_2(k)$,
$R = M_2(\mathfrak{o})$,
$K = \PGL_2(\mathfrak{o})$.
\begin{lemma}\label{lem:formula-for-how-Ad-acts-on-distinguish-element-inner-products}
  Let $\phi_1, \phi_2 \in \mathbb{C} \phi^0$.
  Let $g \in K a(\varpi^m) K$ for some $m \in \mathbb{Z}_{\geq
    0}$
  (see \S\ref{sec:cartan-decomposition}).
  Then $\langle \Ad(g) \phi_1, \phi_2 \rangle_{\Omega}
  = q^{-m} \langle \phi_1, \phi_2 \rangle_{\Omega}$.
\end{lemma}
\begin{proof}
  We expand the definitions and use that
  $\vol(g R g^{-1} \cap R) = q^{-m} \vol(R)$.
\end{proof}

\subsection{Generic representations of
  $\operatorname{PGL}_2$}\label{sec-2-3}
We refer to \cite[\S4.4, \S4.6]{MR1431508}
for details on and proofs of the facts collected here.
Let $\pi$ be an irreducible representation of
$\PGL_2(k)$.
Recall that $\pi$ is \emph{generic}
if it admits a Whittaker model
$\mathcal{W}(\pi,\psi)$, consisting of
$W : \PGL_2(k) \rightarrow \mathbb{C}$ satisfying
$W(n(x) g) = \psi(x) W(g)$.
It then admits a Kirillov model
$\mathcal{K}(\pi,\psi)$, consisting of
$W : k^\times \rightarrow \mathbb{C}$ of the form
$W(y) := W'(a(y))$ for some $W' \in \mathcal{W}(\pi,\psi)$.  The
vector space $\mathcal{K}(\pi,\psi)$ is independent of $\psi$
and contains $C_c^\infty(k^\times)$.
Recall that $\pi$ is \emph{unramified} if the space
$\pi^{\PGL_2(\mathfrak{o})}$ of $\PGL_2(\mathfrak{o})$-invariant
vectors in $\pi$ is nonzero, and that in that case,
$\dim(\pi^{\PGL_2(\mathfrak{o})}) = 1$.

Suppose for the remainder of \S\ref{sec-2-3}
that $\pi$ is generic and unramified.
Let $\psi^0$ be
an unramified unitary character of $k$.
There is then a unique
$\PGL_2(\mathfrak{o})$-invariant vector $W_{\pi}^0$ in the Kirillov
model $\mathcal{K}(\pi,\psi^0)$ of $\pi$ for which
$W_{\pi}^0(1) = 1$.
There is a unique unordered pair $\{\alpha, \beta \}$ of
complex
numbers,
the \emph{Satake parameters} of $\pi$,
so that for $y \in k^\times$ with $|y| = q^{-n}$,
\begin{equation}\label{eq:explicit-formula-W-pi-0}
  W_{\pi}^0(y)
  =
  |y|^{1/2}
  \sum _{\substack{
      i, j \in \mathbb{Z}_{\geq 0}:
      i + j = n
    }
  }
  \alpha^{i} \beta^j
  =
  1_{\mathfrak{o}^\times}(y) |y|^{1/2}
  \frac{\alpha^{n+1} - \beta^{n+1}}{\alpha - \beta }.
\end{equation}
One has in general $\alpha \beta = 1$;
if moreover $\pi$ is unitary,
then either $|\alpha| = |\beta| = 1$
or $\alpha,\beta \in (-q^{1/2}, q^{1/2}) \subseteq
\mathbb{R}$.
The \emph{adjoint $L$-factor}
is defined for $s \in \mathbb{C}$
by
\[
L(\ad \pi,s)
:= (1 - \alpha \beta^{-1} q^{-s})^{-1}
(1 - q^{-s})^{-1}
(1 - \alpha^{-1} \beta  q^{-s})^{-1}.
\]
\begin{lemma*}\label{lem:local-computation-norm-of-whittaker-newvector}
  If $\pi$ is unitary
  and $\Re(s) \geq 0$,
  then
  $L(\ad \pi, s)$ is finite, and one has
  the identity
  \begin{equation}\label{eq:local-computation-norm-of-whittaker-newvector}
    \int_{y \in k^\times}
    |W_{\pi}^0(y)|^2
    |y|^s
    \, \frac{d y}{|y|}
    = \frac{L(\ad \pi,1 + s)}{\zeta_k(2 + 2 s)}
    \Delta_{\psi}^{-1/2}
    \frac{\zeta_k(1 + s)}{\zeta_k(1)}
  \end{equation}
  in which the LHS converges absolutely.
\end{lemma*}
\begin{proof}
  This is a standard calculation.\footnote{
  Since we lack a convenient
  reference sharing our measure normalizations, we record the
  proof.  The LHS
  of
    \eqref{eq:local-computation-norm-of-whittaker-newvector}
    expands to
  $\vol(\mathfrak{o}^\times) \sum_{n \geq 0} |a_n|^2 x^n$, where
  $a_n := (\alpha^{n+1} - \beta^{n+1})/(\alpha-\beta)$ and
  $x := q^{-(1+s) n}$.  One has
  $\vol(\mathfrak{o}^\times) = \zeta_k(1)^{-1}
  \Delta_\psi^{-1/2}$.
  The identity
  $\sum a_n x^n = (1 - \alpha x)^{-1} (1 - \beta x)^{-1}$ and
  \cite[Lem 1.6.1]{MR1431508} imply that
  $\sum_{n \geq 0} |a_n|^2 x^n = (1 - x^2) (1 - \alpha
  \overline{\alpha } x)^{-1} (1 - \alpha \overline{\beta }
  x)^{-1} (1 - \beta \overline{\alpha} x)^{-1} (1 - \beta
  \overline{\beta} x)^{-1}$.
  Invoking the consequence
  $\{\alpha \overline{\alpha }, \alpha \overline{\beta }, \beta
  \overline{\alpha}, \beta \overline{\beta }\} = \{1, \alpha
  \beta^{-1}, \beta \alpha^{-1}, 1\}$
  of the assumed unitarity of $\pi$, we obtain
  $\sum_{n \geq 0} |a_n|^2 x^n = \zeta_k(2 + 2 s)^{-1} L(\ad
  \pi,1 + s) \zeta_k(1 + s)$, as required.
}
\end{proof}

\subsection{Representations of $G$}
\label{sec-2-4}
Let $\pi$ be an irreducible representation of $G$.
Define
a compact open subgroup $J \leq G$
in the following two cases:
\begin{itemize}
\item if $k$
  is non-archimedean,
  take for $J \leq G$ the image of the unit group $R^\times$
  of the chosen maximal order $R \subseteq B$;
\item if $k$ is real and $B$ is non-split,
  set $J := G$.
\end{itemize}
In either case, set
$\vol(J) := \int_{g \in G} 1_J(g)$
and $e_J := \vol(J)^{-1} 1_J \in C_c^\infty(G)$.




\subsubsection{Hecke kernels and theta kernels\label{sec:hecke-kernels-local}}
\label{sec-2-4-1}
Assume that $k$ is non-archimedean.
For $y \in k^\times$,
the \emph{normalized Hecke kernel}
$T_y \in C_c^\infty(J \backslash G / J)$
is defined to be the element with the property
that
$|y|^{-1} \vol(J) T_y$ is the characteristic
function of the image in $G$ of the subset
$\{b \in R : |\nr(b)| = |y| \}$ of $B^\times$.
For example, if $y \in \mathfrak{o}^\times$,
then $T_y = e_J$.
\begin{lemma*}\label{lem:relation-distinguished-elt-hecke-kernel}
  Let $y \in k^\times$, $g \in G$.
  Choose $\tilde{g} \in B^\times$
  with image $g$.
  Then
  \[
  \rho_{\Weil}(a(y))
  \phi^0(\nr(\tilde{g})^{-1}, \tilde{g})
  =
  |y| \phi^0(y \nr(\tilde{g})^{-1}, \tilde{g})
  = \vol(J)
  T_y(g)
  \]
  where $\phi^0 \in \Omega$ is
  the distinguished element (\S\ref{sec:dist-elem}).
\end{lemma*}
\begin{proof}
  We must verify that
  \begin{equation}\label{eq:formula-for-hecke-kernel}
    |y|^{-1} \vol(J) T_y(g) =
    \frac{
      \int_{z \in k^\times}
      1_R(z \tilde{g}) 1_{\mathfrak{o}^\times}(y \nr(z \tilde{g})^{-1})
      \, \frac{d z}{|z|}
    }
    {
      \int_{z \in k^\times}
      1_{\mathfrak{o}^\times}(z)
      \, \frac{d z}{|z|}
    }.
  \end{equation}
  Let $g \in G$.
  The RHS of \eqref{eq:formula-for-hecke-kernel} is independent
  of $\tilde{g}$, and both sides take values in
  $\{0,1\}$.
  The RHS
  of
  \eqref{eq:formula-for-hecke-kernel}
  is nonzero
  iff its integrand
  is nonzero for some $z \in k^\times$, i.e.,
  iff
  for some $z \in k^\times$
  the element
  $b := z \tilde{g}$ lies in $R$ and $|\nr(b)| = |y|$,
  i.e., iff  the LHS
  of
  \eqref{eq:formula-for-hecke-kernel}
  is nonzero.


\end{proof}

\subsubsection{Hecke functionals and standard $L$-factors}
\label{sec-2-4-2}
Continue to assume that $k$ is non-archimedean.
Recall that $\pi$ 
is \emph{unramified} if the space
$\pi^{J}$ of
$J$-invariant vectors in $\pi$ is nonzero;
it is known then that
$\dim(\pi^J) = 1$.

Suppose for remainder of \S\ref{sec-2-4-2}
that $\pi$ is unramified.  
There is then a unique functional
$\lambda_\pi : C_c^\infty(J \backslash G / J) \rightarrow
\mathbb{C}$
so that $\pi(T) v = \lambda_\pi(T) v$ for all
$T \in C_c^\infty(J \backslash G / J), v \in \pi^J$.
We may evaluate this functional on the elements
$T_y$ attached above to $y \in k^\times$:
\begin{itemize}
\item If $B$ is split,
  then there is a unique unordered pair $\{\alpha,\beta\}$
  of complex numbers
  (the \emph{Satake parameters})
  satisfying $\alpha \beta = 1$
  so that with $|y| = q^{-n}$,
  \begin{equation}\label{eq:explicitf-romula-lambda-pi-tY}
      \lambda_\pi(T_y)
  = 
  |y|^{1/2}
  \sum _{\substack{
      i, j \in \mathbb{Z}_{\geq 0}:
      i + j = n
    }
  }
  \alpha^{i} \beta^j
  =
  1_{\mathfrak{o}^\times}(y)
  |y|^{1/2}
  \frac{\alpha^{n+1} - \beta^{n+1}}{\alpha - \beta }
\end{equation}
(see e.g. \cite[\S4.6]{MR1431508}).
\item If $B$ is non-split,
  then there is a unique unramified
  quadratic character $\eta$ of $k^\times$
  so that
  $\lambda_\pi(T_y)
  = |y| \eta(y)$.
\end{itemize}
The \emph{standard $L$-factor}
is then the meromorphic function defined for $s \in \mathbb{C}$
by
\[
L(\pi,s) := \begin{cases}
  (1 - \alpha q^{-s})^{-1} (1 - \beta q^{-s})^{-1} & \text{ if $B$ is split}, \\
  (1 - \eta(\varpi) q^{-s-1/2})^{-1}  & \text{ if $B$ is non-split}.
\end{cases}
\]
\subsubsection{The local Jacquet--Langlands correspondence}
\label{sec-2-4-3}
The Jacquet--Langlands lift $\pi_{\JL}$ of $\pi$
is an irreducible representation of $\PGL_2(k)$
attached to $\pi$.
The following properties of the association
$\pi \mapsto \pi_{\JL}$
are relevant for us:
\begin{itemize}
\item $\pi_{\JL}$ is generic if (and only if) either
  \begin{itemize}
  \item $B$ is non-split, or
  \item $B$ is split and $\dim(\pi) > 1$.
  \end{itemize}
\item If $B$ is split, then $\pi_{\JL}$ corresponds to $\pi$
  under the isomorphism $G \cong \PGL_2(k)$.  In particular, if
  $\pi$ is unramified, then so is $\pi_{\JL}$, and Satake
  parameters (see
  \S\ref{sec-2-3}, \S\ref{sec-2-4-2}) are
  preserved.
\end{itemize}
Assume now that
that $k$ is non-archimedean, that $B$ is split,
that
$\pi$ is unramified, and that
$\dim(\pi) > 1$.
Then $\pi_{\JL}$ is generic and unramified.
Let
$W^0_\pi : k^\times \rightarrow \mathbb{C}$
denote the function
attached to $\pi_{\JL}$
in \S\ref{sec-2-3}.
By \eqref{eq:explicit-formula-W-pi-0}
and
\eqref{eq:explicitf-romula-lambda-pi-tY},
\begin{equation}\label{eqn:key-local-identity-relating-whittaker-and-hecke}
  W_\pi^0(y) = \lambda_\pi(T_y).
\end{equation}

\subsubsection{Local integrals}
\label{sec:local-integrals-for-rallis-ipf}

Assume first that $k$ is non-archimedean, that $\pi$ is unramified,
and that $\pi$ is unitary.
Retain the notation of \S\ref{sec-2-4-2}.
\begin{lemma}\label{lem:local-rallis-ipf-unram-calc}
  Suppose that $B$ is split and that
  $\dim(\pi) > 1$.
  \begin{enumerate}
  \item[(i)] $|\alpha|, |\beta| < q^{1/2}$.
    In particular, $L(\pi,\tfrac{1}{2})$ is finite.
  \item[(ii)] Let $\phi_1, \phi_2$ belong to the line $\mathbb{C} \phi^0$
    spanned by the distinguished element
    $\phi^0 \in \Omega$.
    Let $v_1, v_2 \in \pi^J$.
    Then the identity
    \begin{equation}\label{eqn:local-rallis-ipf-unram-calc}
      \int_{g \in G}
      \langle \Ad(g) \phi_1, \phi_2 \rangle
      \langle g v_1, v_2 \rangle
      =
      \frac{L(\pi,\frac{1}{2})}{\zeta_k(2)}
      \vol(J) 
      \langle \phi_1, \phi_2 \rangle
      \langle v_1, v_2 \rangle
    \end{equation}
    holds, with the LHS converging absolutely.
  \end{enumerate}
\end{lemma}
\begin{proof}
  For (i), see \cite[Thm 4.6.7]{MR1431508}.  For (ii), the
  convergence follows from \S\ref{sec:local-convergence-lemmas}.
  Let $\{\alpha,\beta\}$
  denote the Satake parameters of $\pi$ and set
  $t_1 := \alpha q^{-1/2}, t_2 := \beta q^{-1/2}$, so that
  $L(\pi,\tfrac{1}{2})^{-1} = (1 - t_1) (1-t_2)$.  The
  Macdonald formula \cite[Thm 4.6.6]{MR1431508} says that
  $\langle g v_1, v_2 \rangle = (u_1 t_1^m + u_2 t_2^m)
  \langle v_1, v_2 \rangle$, where
  \[
  u_1 :=
  \frac{1}{1 + q^{-1}}\frac{1 - q^{-1} \beta/\alpha
  }{1 - \beta / \alpha 
  },
  \quad
  u_2 :=
  \frac{1}{1 + q^{-1}}\frac{1 - q^{-1} \alpha / \beta }{1 -
    \alpha / \beta }.
  \]
  By the Cartan decomposition and Lemma
  \ref{lem:formula-for-how-Ad-acts-on-distinguish-element-inner-products},
  we obtain
  \[\int_{g \in G} \langle \Ad(g) \phi_1, \phi_2 \rangle
  \langle g v_1, v_2 \rangle =
  \vol(J) 
  \langle \phi_1, \phi_2 \rangle \langle v_1, v_2 \rangle
  \Sigma,\]
  where
  $\Sigma := \sum_{i=1,2} \sum_{m \geq 0} (1 + 1_{m>0} q^{-1})
  t_i^m$.  We compute
  that $\sum_{i=1,2} u_i (1 + q^{-1} t_i) ( 1 - t_i)^{-1} =
  L(\pi,\tfrac{1}{2}) \Sigma '$
  with
  $\Sigma ' := \sum_{i=1,2} u_i (1 + q^{-1} t_i) ( 1 -
  t_{i'})^{-1}$, $\{i, i'\} = \{1,2\}$.  Direct
  calculation gives $\Sigma ' = \zeta_k(2)^{-1}$, as
  required.
\end{proof}

Suppose now that $B$ is non-split,
so that $\pi$ is the one-dimensional
representation corresponding to the character
$G \ni g \mapsto \eta(\nr(g)) \in \{\pm 1\}$,
as in \S\ref{sec-2-4-2}.
Let $v_1,v_2 \in \pi$.
Recalling that $[G:J] = 2$, we have
\begin{equation}\label{eqn:local-rallis-ipf-integral-nonsplit-finite}
  \int_{g \in G} \langle \Ad(g) e_J, e_J \rangle_{L^2(G)}
  \langle g v_1, v_2 \rangle
  =
  \langle v_1, v_2 \rangle
  \cdot 
  \begin{cases}
    0 & \text{ if $\eta$ is nontrivial,} \\
    2 & \text{ if $\eta$ is trivial.}
  \end{cases}
\end{equation}

Suppose finally that $k \cong \mathbb{R}$, that $B$ is non-split,
and that $\pi$ is trivial.
For $v_1,v_2 \in \pi$, one then has
\begin{equation}\label{eqn:local-rallis-ipf-integral-nonsplit-real}
  \int_{g \in G} \langle \Ad(g) e_J, e_J \rangle_{L^2(G)}
  \langle g v_1, v_2 \rangle
  =
  \langle v_1, v_2 \rangle.
\end{equation}


\section{Global preliminaries}
\label{sec-3}
In this section we collect those preliminaries for the proof of
Theorem \ref{thm:main-estimate-general-variance}
whose discussion makes sense independently of that proof.

Let $F$ be a number field with adele ring $\mathbb{A}$,
let $B$ be a quaternion algebra over $F$,
and let $\psi$ be a nontrivial unitary character of
$\mathbb{A}/F$.

\subsection{Generalities}
\label{sec-3-1}
\subsubsection{Notation}
\label{sec-3-1-1}
We denote by $\mathcal{O}_F$ or simply $\mathcal{O}$ the ring of
integers in $F$.  We denote by $\mathfrak{p}$ a place of $F$,
finite or infinite.  A subscripted $\mathfrak{p}$ denotes
completion; for example, $\mathcal{O}_\mathfrak{p}$ denotes the
ring of integers of $F_{\mathfrak{p}}$ if $\mathfrak{p}$ is finite.  For a finite set of
places $S$, a subscripted $S$ denotes a product taken over $S$,
such as in
$F_S := \prod_{\mathfrak{p} \in S} F_\mathfrak{p}, B_S :=
\prod_{\mathfrak{p} \in S} B_\mathfrak{p}$.

The character $\psi$ factors as
$\psi(x) = \prod \psi_\mathfrak{p}(x_\mathfrak{p})$, where
$\psi_\mathfrak{p}$ is a nontrivial unitary character of
$F_\mathfrak{p}$.

For a place $\mathfrak{p}$,
let $\zeta_\mathfrak{p} := \zeta_{F_\mathfrak{p}}$ denote the
local Euler factor.
Let $\xi_F(s) := \prod \zeta_\mathfrak{p}(s)$ denote
the Dedekind zeta function (absolutely convergent for $\Re(s) > 1$) and
$\xi_F^*(1) := \res_{s \rightarrow 1} \xi_F(s)$ its residue.  
For a finite set $S$ of places that contains
the infinite places, let $\zeta_F^{(S)}(s) :=
\prod_{\mathfrak{p} \in S}
\zeta_\mathfrak{p}(s)$ denote the partial Dedekind zeta function.

\subsubsection{Groups}
For an algebraic $F$-group $\mathbf{G}$ we write
$G := \mathbf{G}(F), G_\mathfrak{p} :=
\mathbf{G}(F_\mathfrak{p})$,
$G_\mathbb{A} := \mathbf{G}(\mathbb{A})$,
$G_S := \mathbf{G}(F_S) = \prod_{\mathfrak{p} \in S} G_\mathfrak{p}$,
$[G] := G \backslash G_\mathbb{A}$.  This notation applies
notably to
the $F$-group $\bPB^\times$ given by
$\bPB^\times(A) := (B \otimes_F A)^\times/A^\times$
and also
to the $F$-groups
${{\mathbf P}{\mathbf G}{\mathbf L}}_2$,
${{\mathbf S}{\mathbf L}}_2$.
We similarly abbreviate $[\Mp_2] := \SL_2(F) \backslash \Mp_2(\mathbb{A})$ (see \S\ref{sec:global-metaplectic-gp}).

\subsubsection{Measures\label{sec:global-measures}}
\label{sec-3-1-2}
When $\mathbf{G}$ is semisimple, we equip $G_\mathbb{A}$ and  $[G]$ with
Tamagawa measures.
Then $\vol([\SL_2]) = 1$ and
$\vol([\PGL_2]) = \vol([\PB^\times]) = 2$. 
Denote by $\langle , \rangle_{G}$ the
corresponding inner product on $L^2([G])$; we omit the subscripted
$G$ if it is clear by context.  

For each
place $\mathfrak{p}$, the character $\psi_\mathfrak{p}$ induces
(via the recipe of \S\ref{sec:local-measures})  a Haar measure on $F_\mathfrak{p}$,
$B_\mathfrak{p}$,
$F_\mathfrak{p}^{\times} / F_{\mathfrak{p}}^{\times
  2}$,
$\PB^\times_\mathfrak{p}$; we equip $\mathbb{A}, B_\mathbb{A}$,
$\mathbb{A}^{\times} / \mathbb{A}^{\times 2}$ and
$\PB^\times_\mathbb{A}$ with the corresponding restricted
product measures.
(This defines the Tamagawa measure on $\PB^\times_\mathbb{A}$.)
The quotient measures on $\mathbb{A}/F$ and
$B_\mathbb{A}/B$ are then probability measures.
We likewise equip finite products
such as $F_S$ or $\PB^\times_S$ with product measures.

We equip $\mathbb{A}^\times$ with the regularized
product of the measures constructed in \S\ref{sec:local-measures}:
for a factorizable function
$f = \prod f_\mathfrak{p} \in C_c^\infty(\mathbb{A}^\times)$
for which $f_\mathfrak{p} = 1_{\mathcal{O}_\mathfrak{p}^\times}$
for almost all finite primes $\mathfrak{p}$,
we set
\[
\int_{y \in \mathbb{A}^\times}
f(y) \, \frac{d y}{|y|}
:=
\frac{1}{\xi_F^*(1)}
\prod_{\mathfrak{p}}
\zeta_\mathfrak{p}(1)
\int_{y \in F_\mathfrak{p}^\times}
f_\mathfrak{p}(y)
\, \frac{d y}{|y|}.
\]
We thereby
obtain a quotient Haar $\frac{d y}{|y|}$ on
$\mathbb{A}^\times / F^\times$ whose pushforward under
$|.| : \mathbb{A}^\times / F^\times \rightarrow
\mathbb{R}^\times_+$
is the standard Haar measure $\frac{d t}{|t|}$ on
$\mathbb{R}^\times_+$, where $d t$ denotes Lebesgue measure.

The quotient measure on
the discrete group $F^{\times } / F^{\times 2}$
compatible with the squaring map
is half the counting measure,
i.e.,
for finitely-supported $f : F^\times \rightarrow \mathbb{C}$,
one has
$\sum_{x \in F^\times } f(x)
= \frac{1}{2} \sum_{y \in F^\times / F^{\times 2}}
(\sum_{z \in F^\times} f(y z^2) )$.
On $\mathbb{A}^\times / F^\times \mathbb{A}^{\times 2}$,
we take the quotient measure
induced by the exact sequence 
$1 \rightarrow F^\times / F^{\times 2}
\rightarrow 
\mathbb{A}^\times / \mathbb{A}^{\times 2} \rightarrow
\mathbb{A}^\times / F^\times \mathbb{A}^{\times 2} \rightarrow
1$,
where $F^{\times} / F^{\times 2}$ is equipped with half the
counting measure.
Thus for $f \in C_c(\mathbb{A}^\times / \mathbb{A}^{\times 2})$,
\begin{equation}\label{eq:integral-formula-involving-squares}
  \int_{y \in \mathbb{A}^\times / F^\times \mathbb{A}^{\times  2}}
  \frac{1}{2} \sum_{a \in F^\times / F^{\times 2}}
  f(a y)
  = \int_{\mathbb{A}^\times / \mathbb{A}^{\times 2}} f.
\end{equation}
By decomposing the Haar on $\mathbb{A}^\times$ in two ways,
one
finds
for $f \in C_c(\mathbb{A}^\times / F^\times)$
that
$\int_{\mathbb{A}^\times / F^\times} f = \int_{x \in
  \mathbb{A}^\times / F^{\times } \mathbb{A}^{\times 2}} \int_{y
  \in \mathbb{A}^\times / F^\times} f(x y^2)$;
moreover, $\vol(\mathbb{A}^{\times} / F^\times \mathbb{A}^{\times 2}) = 2$.
Finally, for $f \in C_c^\infty([\PGL_2])$,
\begin{equation}\label{eq:sl2-vs-pgl2-integrals}
  \int_{[\PGL_2]} f
  =
  \int_{y \in \mathbb{A}^\times /  F^\times \mathbb{A}^{\times 2}}
  \int_{s \in [\SL_2]}
  f(s a(y)).
\end{equation}

\subsubsection{The $\Xi$-function\label{sec:Xi-global}}
\label{sec-3-1-3}
Fix a maximal compact subgroup
$K = \prod K_\mathfrak{p} \leq \PB^\times_{\mathbb{A}}$.
Let $\Xi : \PB^\times_{\mathbb{A}} \rightarrow \mathbb{C}$
be the product
$\Xi(g) := \prod \Xi_\mathfrak{p}(g_\mathfrak{p})$
of the functions $\Xi_\mathfrak{p}$ on $\PB^\times_\mathfrak{p}$
attached in \S\ref{sec:local-Xi}
to the factors $K_\mathfrak{p}$.

\subsubsection{Conventions}\label{sec:conventions-cusp-forms}
A \emph{cusp form} is a smooth vector
in the Hilbert space $L^2_{\cusp}([G])$
of square-integrable cuspidal functions.
A \emph{cuspidal automorphic representation}
$\pi$ of $G_\mathbb{A}$ is the space of smooth vectors
in an irreducible subrepresentation of
$L^2_{\cusp}([G])$.

\subsection{Automorphic forms on $\PGL_2$}
\label{sec-3-3}
\subsubsection{Fourier expansions}\label{sec:aut-forms-fourier-exp}
\label{sec-3-3-1}
Let $\varphi : [\PGL_2] \rightarrow \mathbb{C}$ be a
smooth function.  It admits the Fourier expansion
$\varphi(n(x) a(y)) = c_{\varphi}(y) + \sum_{\tau \in F^\times} \psi(\tau x)
W_{\varphi}(\tau y)$,
where $c_{\varphi}(y) := \int_{x \in \mathbb{A}/F} \varphi(n(x) a(y))$
denotes the constant term and
\[W_{\varphi}(y) := \int_{x \in \mathbb{A}/F} \psi(- x) \varphi(n(x)
a(y))\]
denotes the (diagonal restriction of) the Whittaker function.
The standard Borel subgroup of $\PGL_2(\mathbb{A})$ has dense image in
$[\PGL_2]$, so $\varphi$ is determined by the values
$\varphi(n(x) a(y))$ for
$x \in \mathbb{A}, y \in \mathbb{A}^\times$.
Recall that $\varphi$ is \emph{cuspidal} if $c_{\varphi} = 0$;
in that case, $\varphi$ is
determined by $W_\varphi$.

\subsubsection{Kirilov model}
\label{sec-3-3-2}
Let $\pi \subseteq L^2([\PGL_2])$ be a cuspidal automorphic
representation; it is (the smooth completion of) a restricted
tensor product $\otimes \pi_\mathfrak{p}$, where
$\pi_\mathfrak{p}$ is a generic for every $\mathfrak{p}$ and
unramified for almost all finite $\mathfrak{p}$.  Let
$\mathcal{K}(\pi,\psi) := \{W_\varphi : \varphi \in \pi \}$.
The natural map $\pi \rightarrow \mathcal{K}(\pi,\psi)$ is a
linear isomorphism under which the pure tensors in $\pi$
correspond to the factorizable functions
$W(y) = \prod W_\mathfrak{p}(y_\mathfrak{p})$, where $W_\mathfrak{p}$ belongs to
the local Kirillov model
$\mathcal{K}(\pi_\mathfrak{p},\psi_\mathfrak{p})$ and satisfies
$W_\mathfrak{p} = W_{\pi_\mathfrak{p}}^0$ (see \S\ref{sec-2-4-3}) for almost all finite
$\mathfrak{p}$.
\begin{lemma*}\label{lem:inner-product-formula-petersson-vs-kirillov}
  Let $S$ be a finite set of places of $F$
  that contains all infinite places
  as well as any places $\mathfrak{p}$
  at which
  $\pi$ ramifies.
  Let $\varphi \in \pi$.
  The integral $I(s) := \int_{y \in \mathbb{A}^\times}
  |W_\varphi(y)|^2 |y|^s \, \frac{d y}{|y|}$
  converges absolutely for complex numbers $s$ with positive real part,
  extends to a meromorphic function on the complex plane,
  and satisfies
  \begin{equation}\label{eqn:rs-ipf}
    2 \res_{s \rightarrow 0} I(s) = \|\varphi\|^2.
  \end{equation}
\end{lemma*}
\begin{proof}
  See for instance \cite[Lem 2.2.3]{michel-2009}.\footnote{
    The cited reference
    normalizes measures
    differently than we do.
    For the convenience of the reader,
    we sketch here the proof that
    \eqref{eqn:rs-ipf} is normalized correctly, taking for
    granted the meromorphicity of $I(s)$.  By unfolding,
    $I(s) = \int_{y \in \mathbb{A}^\times / F^\times} \int_{x
      \in \mathbb{A}/F} |\varphi(n(x) a(y))|^2 |y|^s \,
    \frac{d y}{|y|}$.
    Because the pushforward under
    $|.| : \mathbb{A}^\times / F^\times \rightarrow
    \mathbb{R}_+^\times$
    of $\frac{d y}{|y|}$ is the standard Haar measure
    $\frac{d t}{t}$ on $\mathbb{R}^\times_+$ (see
    \S\ref{sec:global-measures}), it follows that
    $\res_{s \rightarrow 1} I(s) = \lim_{t \rightarrow 0}
    \mathbb{E}_{y : |y| = t} \int_{x \in \mathbb{A}/F}
    |\varphi(n(x) a(y))|^2$
    where $\mathbb{E}$ denotes an average with respect to the
    probability measure invariant by the norm one
    idele class group.  By the equidistribution of the horocycle flow and
    the normalization $\vol(\mathbb{A}/F) = 1$, that limit is
    the integral of $|\varphi|^2$ with respect to the
    probability Haar on $[\PGL_2]$.  By the consequence
    $\vol([\PGL_2]) = 2$ of our measure normalizations, we
    obtain \eqref{eqn:rs-ipf}.}
\end{proof}

\subsubsection{Adjoint $L$-function}
\label{sec:adjoint-l-function}
For $\pi$, $S$ as in
the lemma of \S\ref{sec-3-3-2}, the
\emph{partial adjoint $L$-function} is
defined for
$\Re(s) > 1$ by the absolutely-convergent Euler product
$L^{(S)}(\ad \pi,s) := \prod_{\mathfrak{p} \notin S} L(\ad
\pi_\mathfrak{p},s)$;
it continues meromorphically
to the complex plane,
and is holomorphic for (at least) $\Re(s) \geq 1$
(see \cite{MR533066}).

\subsection{Automorphic forms on $\PB^\times$}
\label{sec-3-4}

\subsubsection{Jacquet--Langlands lifts}
\label{sec-3-4-1}
Let
$\pi = \otimes \pi_\mathfrak{p} \subseteq L^2([\PB^\times])$ be
a cuspidal automorphic representation with
$\dim(\pi) > 1$.
By \cite[Prop 4]{MR0333081},
\begin{equation}\label{eqn:generic-at-each-place-if-not-1-diml}
  \text{$\dim(\pi_\mathfrak{p}) > 1$ for any prime $\mathfrak{p}$ at
    which $B$ splits.}
\end{equation}
The Jacquet--Langlands lift
$\pi_{\JL} = \otimes \pi_{\JL,\mathfrak{p}} \subseteq
L^2([\PGL_2])$
is the unique cuspidal automorphic automorphic
representation
for which $\pi_{\JL,\mathfrak{p}} = (\pi_\mathfrak{p})_{\JL}$
for each place $\mathfrak{p}$.
If $\mathfrak{p}$ is a finite prime at which $B$ splits and for
which $\pi_\mathfrak{p}$ is unramified, then
$(\pi_{\JL})_\mathfrak{p}$ is unramified.  The association
$\pi \mapsto \pi_{\JL}$ is injective.


\subsubsection{The pretrace formula\label{sec:pretrace-formula}}
\label{sec-3-4-2}
Assume that $B$ is non-split,
so that $[\PB^\times]$ is compact.
Fix a maximal compact subgroup $K$ of $\PB^\times_\mathbb{A}$.
The pretrace formula
asserts that for $f \in C_c^\infty(\PB_\mathbb{A}^\times)$
and $x \in \PB^\times_\mathbb{A}$,
\begin{equation}\label{eq:pretrace-formula-general}
  \sum_{\pi}
  \sum_{\varphi}
  \overline{\varphi (x)}
  \pi(f) \varphi(x)
  = \sum_{\gamma \in \PB^\times} f(x^{-1} \gamma x),
\end{equation}
where $\pi$ traverses the irreducible subrepresentations of
$L^2([\PB^\times])$ and $\varphi$ traverses an orthonormal basis
$\mathcal{B}(\pi)$ of $\pi$ consisting of $K$-isotypic vectors.
Only finitely many summands on the RHS of
\eqref{eq:pretrace-formula-general} are nonzero, while the
condition on $\mathcal{B}(\pi)$ implies that the LHS of
\eqref{eq:pretrace-formula-general} converges absolutely, or
indeed, rapidly: Let
$C(\pi) := \prod_\mathfrak{p} C((\pi_\mathfrak{p})_{\JL}) \in
\mathbb{R}_{\geq 1}$
denote the analytic conductor of $\pi$ (see e.g.  \cite[\S3.1.8,
\S4.1.4]{michel-2009}).  Then for each $A \geq 0$, one
has\footnote{
  For instance, this follows (in overkill fashion)
  from the proof of \cite[Thm 9.1]{MR1616155}.
}
\begin{equation}\label{eq:rapid-convergence-of-pretrace-formula}
  \sum_{\pi} C(\pi)^A
  \sum_{\varphi}
  |\overline{\varphi}(x) \pi(f) \varphi(x)|
  < \infty.
\end{equation}
Note also that 
there exists $A_0 > 3$
so that
(see e.g. \cite[(2.15)]{michel-2009}).
\begin{equation}\label{eq:polynomial-growth-of-reps}
  \sum_{\pi} C(\pi)^{-A_0}
  < \infty
\end{equation}


Let $S$ be a set of places containing the infinite ones.
Let $R \subseteq B$ be a maximal order.
For each $\mathfrak{p} \notin S$,
let $J_\mathfrak{p} \leq \PB^\times_\mathfrak{p}$
denote the image of $R_\mathfrak{p}^\times$,
as in \S\ref{sec-2-4},
and set $J := \prod_{\mathfrak{p} \notin S} J_\mathfrak{p}$.
Suppose that
$f = f_S \otimes (\otimes_{\mathfrak{p} \notin S}
T_{y_\mathfrak{p}})$
for some $f_S \in C_c^\infty(\PB_S^\times)$
and
$y \in \mathbb{A}^\times$,
where $T_{y_\mathfrak{p}}$ is the Hecke kernel
as defined in \S\ref{sec:hecke-kernels-local} relative to $J_\mathfrak{p}$.
The formula
\eqref{eq:pretrace-formula-general}
then
specializes to
\begin{equation}
  \sum_\pi
  (\sum_{\varphi}
  \overline{\varphi(x)}
  \pi(f_S) \varphi(x))
  \prod_{\mathfrak{p} \notin S}
  \lambda_{\pi_\mathfrak{p}}(T_{y_\mathfrak{p}}) 
  =
  \sum_{\gamma \in \PB^\times}
  f_S(x_S^{-1} \gamma x_S)
  \prod_{\mathfrak{p} \notin S}
  T_{y_\mathfrak{p}}(x_\mathfrak{p}^{-1} \gamma x_\mathfrak{p}),
\end{equation}
where $\pi \subseteq L^2([\PB^\times])$ now traverses the subrepresentations
that are unramified outside $S$ (i.e., that contain a nonzero
$J$-fixed vector)
and
$\varphi$ traverses
an orthonormal basis
of $K$-isotypic vectors
for the $J$-fixed subspace $\pi^J$
of $\pi$.


\subsubsection{$L$-functions}\label{sec:standard-l-function}
Let $\pi \subseteq L^2([\PB^\times])$
be a
cuspidal automorphic representation
with $\dim(\pi) > 1$.
Let  $S$ be a finite set of places
containing
all infinite places
as well as any places
at which either $B$ or $\pi$ ramifies.

The \emph{partial standard $L$-function}
is
defined for
$\Re(s) > 1$ by the absolutely-convergent Euler product
$L^{(S)}(\pi,s) := \prod_{\mathfrak{p} \notin S}
L(\pi_\mathfrak{p},s)$;
it continues meromorphically
to the complex plane,
and is holomorphic for (at least) $\Re(s) \geq 1/2$
(see e.g. \cite[\S3.5]{MR1431508}).

Set $L^{(S)}(\ad \pi,s) := L^{(S)}(\ad \pi_{\JL}, s)$
(see \S\ref{sec:adjoint-l-function}).
By \cite{HL94} (cf. \cite[\S 2.9]{2009arXiv0904.2429B}),
one has
\begin{equation}\label{eq:HL}
 C(\pi)^{-\eps}  \ll_{\eps} L^{(S)}(\ad \pi,1) \ll_{\eps}
 C(\pi)^{\eps}
 \text{ for each } \eps > 0.
\end{equation}

\subsection{Theta functions}
\label{sec-3-2}
\subsubsection{Metaplectic group\label{sec:global-metaplectic-gp}}
\label{sec-3-2-1}
Let $\Mp_2(\mathbb{A})$ denote the metaplectic double cover
of $\SL_2(\mathbb{A})$;
it fits into a short exact sequence
$1 \rightarrow \mu_2 \rightarrow \Mp_2(\mathbb{A})
\xrightarrow{\pr} \SL_2(\mathbb{A})$.
We may identify it with
$\SL_2(\mathbb{A}) \times \mu_2$
with the group law given by
$(s_1,\zeta_1) (s_2,\zeta_2) = (s_1 s_2, \zeta_1 \zeta_2
c(s_1,s_2))$,
where $c$ is the product of the cocycles
from \S\ref{sec:global-metaplectic-gp}.
We identify $\SL_2(F)$
with its image under the unique splitting $\SL_2(F)
\hookrightarrow \Mp_2(\mathbb{A})$.

We may similarly define $\Mp_2(F_S)$
as a double cover of $\SL_2(F_S)$
for any collection $S$ of places of $F$.
\subsubsection{Quadratic spaces}
\label{sec-3-2-2}
One defines quadratic spaces $V$ over $F$ as in
\S\ref{sec:local-quadratic-spaces}.  The relevant examples are
still $V = B, B^0, F$.  We equip $V_\mathbb{A}$ with the
$(\psi,b_V)$-self dual measure $\mu_V$.  That measure is the
product of the measures $\mu_{V_\mathfrak{p}}$ on the local
spaces $V_{\mathfrak{p}}$ attached to $\psi_\mathfrak{p}$, and
is independent of $\psi$: it assigns volume one to a fundamental
domain for $V_\mathbb{A}/V$.


\subsubsection{Weil representation\label{sec:weil-repn-global}}
\label{sec-3-2-3}
For a quadratic space $V$ over $F$,
the Schwartz--Bruhat space $\mathcal{S}(V_\mathbb{A})$ factors
as the (completed) restricted tensor product
$\mathcal{S}(V_\mathbb{A}) = \otimes
\mathcal{S}(V_\mathfrak{p})$.
The Weil representation
$\rho_{\Weil}^{\psi,V} : \Mp_2(\mathbb{A})
\times \O(V_\mathbb{A}) \rightarrow
\GL(\mathcal{S}(V_\mathbb{A}))$
is given by
$\rho_{\Weil}^{\psi,V} = \otimes
\rho_{\Weil}^{\psi_{\mathfrak{p}},V_{\mathfrak{p}}}$
in the evident sense.

We may similarly define
a Weil representation
$\rho_{\Weil}^{\psi,V} : \Mp_2(F_S)
\times \O(V_S)
\rightarrow \GL(\mathcal{S}(V_S))$
for a finite set $S$ of places of $F$.

\subsubsection{Theta kernels\label{sec:theta-kernels}}
\label{sec-3-2-5}
Let $V$ be a quadratic space over $F$.
For $\phi \in
\mathcal{S}(V_\mathbb{A})$,
$s \in \Mp_2(\mathbb{A})$
and $g \in \O(V_\mathbb{A})$,
set
$\theta_{\psi}(\phi)(s,g)
:=
\sum_{x \in V}
\rho_{\Weil}^{\psi,V}(s,g) \phi(x)$.
The sum converges absolutely
and defines a smooth function
$\theta_{\psi}(\phi) : [\Mp_2]
\times [\O(V)] \rightarrow \mathbb{C}$.
We employ notation such as
$\theta_{\psi}(\phi;s,g) := \theta_{\psi}(\phi)(s,g)$.
Observe that
\begin{equation}\label{eq:equivariance-theta-kernel}
  \theta_{\psi}(\phi;s s', g g')
  = \theta_{\psi}(\rho_{\Weil}^{\psi,V}(s',g') \phi;s,g)
\end{equation}
\subsubsection{Elementary theta functions\label{sec:elem-theta-fns}}
\label{sec-3-2-6}
Let $V = F$, regarded as a quadratic subspace of $B$ as in
\S\ref{sec:local-quadratic-spaces}.  In that case, we abbreviate $\O_1(F) := \O(V) \cong \{\pm 1\}$.
For
$\phi \in \mathcal{S}(V_\mathbb{A}) = \mathcal{S}(\mathbb{A})$,
we denote also by $\theta_{\psi}(\phi)$ the elementary theta function on
$[\Mp_2]$ obtained by restricting
to the first factor
the theta kernel defined in \S\ref{sec:theta-kernels}, thus
$\theta_{\psi}(\phi)(s) := \theta_{\psi}(\phi)(s,1)
=
\sum_{x \in F}
\rho_{\Weil}^{\psi,F}(s) \phi(x)$.
By \eqref{eq:equivariance-theta-kernel},
\begin{equation}\label{eq:equivariance-etf}
  \rho_{\reg}(s) \theta_{\psi}(\phi)
  = \theta_{\psi}(\rho_{\Weil}^{\psi,F}(s) \phi)
  \text{ for }
  s \in \Mp_2(\mathbb{A}).
\end{equation}
The $\O_1(F)$-invariance
of the theta kernel says that
for $\phi \in \mathcal{S}(\mathbb{A})$,
\begin{equation}\label{eqn:O1-invariance-elementary-theta-fn}
  \text{$\theta_{\psi}(\phi) = \theta_{\psi}(\phi_-)$
    with $\phi_-(x) := \phi(-x)$.}
\end{equation}

\subsubsection{Ternary theta lifts}
\label{sec-3-2-7}
Suppose $V = B^0$.
Given $\Psi : [\PB^\times] \rightarrow \mathbb{C}$
and $\phi \in \mathcal{S}(B^0_\mathbb{A})$
and
$s \in \Mp_2(\mathbb{A})$,
set
$\theta_{\psi}(\phi,\Psi;s)
:=
\int_{g \in [\PB^\times]}
\Psi(g)
\theta_{\psi}(\phi;s,\Ad(g))$
where $\Ad : \PB^\times_\mathbb{A} \xrightarrow{\cong}
\SO(B^0_\mathbb{A})$
is the isomorphism
induced by the notation of \S\ref{sec:general-notation}.
If $\Psi$ is a cusp form, then the integral
converges absolutely and defines
a cusp form
$\theta_{\psi}(\phi,\Psi) : [\Mp_2] \rightarrow \mathbb{C}$.
By \eqref{eq:equivariance-theta-kernel},
\begin{equation}\label{eq:equivariance-ttf}
  \rho_{\reg}(s) \theta_{\psi}(\phi,\Psi)
  = \theta_{\psi}(\rho_{\Weil}^{\psi,B^0}(s) \phi, \Psi)
  \text{ for }
  s \in \Mp_2(\mathbb{A}),
\end{equation}
\begin{equation}\label{eq:equivariance-ttf-2}
  \theta_{\psi}(\Ad(g)\phi, \rho_{\reg}(g) \Psi)
  =
  \theta_{\psi}(\phi, \Psi)
  \text{ for }
  g \in \PB^\times_\mathbb{A}.
\end{equation}

\subsubsection{Factorization}
\label{sec-3-2-8}
If the quadratic space $V$ decomposes as the direct sum
$V' \oplus V''$ of quadratic subspaces,
then the factorization of the Weil representation (\S\ref{sec:factorization-weil-repn})
implies the factorization of theta functions: for $g = g' \times
g''
\in \O(V'_\mathbb{A}) \times \O(V_\mathbb{A} '')
\leq \O(V_\mathbb{A})$
and
$\phi = \phi ' \otimes \phi '' \in \mathcal{S}(V_\mathbb{A})$
with $\phi ' \in \mathcal{S}(V'_\mathbb{A}),
\phi '' \in \mathcal{S}(V_\mathbb{A} '')$ (see \S\ref{sec:factorization-weil-repn}),
\begin{equation}\label{eqn:Factorization-of-theta-fns}
  \theta_{\psi}(\phi;s,g)
  =
  \theta_{\psi}(\phi ';s,g')
  \theta_{\psi}(\phi '';s,g'').
\end{equation}
With $\phi '_-$
as in \eqref{eqn:O1-invariance-elementary-theta-fn}
and notation as in \S\ref{sec:general-notation},
one has
\begin{equation}\label{eq:action-of-S-on-pure-tensors-for-B-0}
  \Ad(g) \phi = \phi ' \otimes \Ad(g) \phi ''
\end{equation}
\begin{equation}\label{eq:action-of-S-on-pure-tensors-for-B}
  \mathfrak{S} \phi = \phi '_- \otimes \phi '',
\end{equation}

\subsection{Equidistribution of products of pairs of elementary theta functions\label{sec:equid-prod-pairs-theta}}
\label{sec-3-5}
The purpose of this section is to recall and apply some results
from \cite{nelson-theta-squared}.  Let
$\tau_1,\tau_2 \in F^\times$.  Throughout this section we regard
$\psi,\tau_1,\tau_2,F,B$ as fixed: implied constants may depend
upon them without explicit mention.  We assume also (for
technical convenience)
that $B$ is
non-split.

\subsubsection{Some asymptotic notation}\label{sec:some-asympt-notat}
Given a topological vector space $\mathcal{S}$, we
adopt the convention (similar to ``big O notation'') of denoting by
$\mathcal{C}(\phi)$
any quantity depending continuously upon
$\phi \in \mathcal{S}$; the continuity is assumed uniform in all
auxiliary parameters except those explicitly labelled ``fixed.''
The space $\mathcal{S}$ itself is always regarded as fixed, of course.  This
convention applies in particular to Schwartz--Bruhat spaces of
finite-dimensional vector spaces over local fields, over finite
products of local fields, or over adele rings.

Similarly to the ``$\eps$-convention''
of analytic number theory,
we allow the precise meaning
of $\mathcal{C}(\phi)$
to change from one occurrence to the next.
When we specifically
wish to distinguish between
several such quantities,
we use the notation
$\mathcal{C} '(\phi), \mathcal{C} ''(\phi)$, and so on.

For example, let $V$ be a vector space over $F$
(always assumed finite-dimensional).  Let $\hat{F}$
denote the ring of finite adeles, so that
$\mathbb{A} = F_\infty \times \hat{F}$ with
$F_\infty := \prod_{\mathfrak{p} | \infty} F_\mathfrak{p}$.
Similarly, write $V_\mathbb{A} = V_\infty \times \hat{V}$.
The
Schwartz--Bruhat space $\mathcal{S}(V_\mathbb{A})$
factors as
the algebraic tensor product $\mathcal{S}(V_\infty) \otimes
\mathcal{S}(\hat{V})$.
Suppose given 
some quantities $a(\phi;t_1,t_2)$
and $b(\phi;t_1,t_2)$ depending
upon $\phi \in \mathcal{S}(V_\mathbb{A})$
and some auxiliary parameters $t_1,t_2$.
The notation
\begin{equation}\label{eq:example-a-phi-bounded-by-C-phi}
  a(\phi;t_1,t_2)
  \ll b(\phi;t_1,t_2) \mathcal{C}(\phi) \text{ for fixed $t_2$ }
\end{equation}
means that
for each $t_2$ and each $\phi_f \in \mathcal{S}(\hat{V})$
there is a
finite collection $\mathcal{P}$ of polynomials
on $V_\infty$
and a finite collection $\mathcal{D}$ of translation-invariant
differential operators
on $V_\infty$
(thus $\mathcal{D}$ consists of linear combinations
of monomials
$\frac{\partial }{\partial x_{i_1}}
\dotsb \frac{\partial }{\partial x_{i_n}}$
with respect to some coordinates $x_j : V_\infty \rightarrow \mathbb{R}$)
so that for all $\phi_\infty \in \mathcal{S}(V_\infty)$
and all $t_1$,
\[
|
a(\phi_\infty \otimes \phi_f;t_1,t_2)
|
\leq 
|b(\phi_\infty \otimes \phi_f;t_1,t_2)|
\sum_{P \in \mathcal{P}}
\sum_{D \in \mathcal{D}}
\|P D \phi_\infty \|_{L^\infty(V^\infty)}.
\]
One could just as well
write ``the
functional
$\mathcal{S}(V_\mathbb{A}) \ni \phi \mapsto
a(\phi;t_1,t_2)/b(\phi;t_1,t_2)$
is defined and continuous, uniformly in $t_1$,'' but that would
be stilted.

\begin{remark*}
  It would be ``better'' to work not with the
  notation $\mathcal{C}$ but instead with a system of adelic
  Sobolev norms on $\mathcal{S}(V_\mathbb{A})$,
  such as those attached by the recipe of
  \cite[\S2]{michel-2009} and \cite[\S4.6,
  \S5.3]{nelson-theta-squared} to the basic Weil representation
  of the metaplectic group of the symplectic space
  $V^*_\mathbb{A} \oplus V_\mathbb{A}$.  The advantage of
  doing so would be to obtain polynomial
  dependence in Theorem
  \ref{thm:main-result-for-microlocal-stuff} upon
  the
  levels of $\Psi_1, \Psi_2$.  The disadvantages would be to
  lengthen the article and introduce technical overhead
  irrelevant to our primary aims.
\end{remark*}


\subsubsection{Simple estimates for lattice sums}\label{sec:simple-estimates-for-lattice-sums}
\begin{lemma}\label{lem:cheap-lattice-sum-bound-Rn-Zn}
  Let $n \in \mathbb{Z}_{\geq 0}$.
  Let $A \geq 0$ and $t_0 > 0$ be fixed.
  For each $\phi \in \mathcal{S}(\mathbb{R}^n)$
  and $t > t_0$,
  one has
  $\sum_{v \in \mathbb{Z}^n - \{0\}} |\phi(t v)|  \ll |t|^{-A} \mathcal{C}(\phi)$.
\end{lemma}
\begin{proof}
  The LHS of
  the required estimate 
  is bounded by
  \[C
  |t|^{-A}
  \sup_{x \in \mathbb{R}^n}
  |x|^{n+1+A} |\phi(x)|\]
  with $C := |t_0|^{-(n+1)} \sum_{v \in \mathbb{Z}^n - \{0\}} |v|^{-(n+1+A)} < \infty$.
\end{proof}
\begin{lemma}\label{lem:cheap-lattice-sum-bound-adelic}
  Let $V$ be a vector space over $F$.
  Let $A \geq 0$ and $t_0 > 0$ be fixed.
  For $\phi \in \mathcal{S}(V_\mathbb{A})$
  and
  $y \in \mathbb{A}^\times$ with $|y| > t_0$,
  one has
  $\sum_{v \in V - \{0\}}
  |\phi(y v)|
  \ll |y|^{-A} \mathcal{C}(\phi)$.
\end{lemma}
\begin{proof}
  Observe first that the action of $\mathbb{A}^\times$ on
  $\mathcal{S}(V_\mathbb{A})$ by dilation is continuous.  Let
  $\mathbb{A}^{(1)} := \{y \in \mathbb{A}^\times : |y| = 1\}$
  denote the subgroup of norm one ideles.  By the compactness of
  $\mathbb{A}^{(1)} / F^\times$ and the previous observation, it
  suffices to consider the case that $y_\mathfrak{p} = 1$
  for all finite $\mathfrak{p}$ 
  and $y_\mathfrak{p} = t$ for all infinite $\mathfrak{p}$,
  where $t \in \mathbb{R}_+^\times$
  satisfies
  $t > t_1 := t_0^{1/[F:\mathbb{Q}]}$.
  Each $\phi_f \in \mathcal{S}(\hat{V})$
  is bounded and satisfies
  $\supp(\phi_f) \cap V \subseteq L$
  for some lattice $L \subseteq V$,
  so it suffices to show
  for each fixed $A \geq 0$
  and
  fixed lattice $L \subseteq V$
  that for all $\phi \in \mathcal{S}(V_\infty)$
  and $t  > t_1$,
  one has
  $\sum_{v \in L - \{0\}} |\phi(t v)| = O(|t|^{-A} \mathcal{C}(\phi))$.
  By choosing a $\mathbb{Z}$-basis of $L$,
  we reduce to Lemma \ref{lem:cheap-lattice-sum-bound-Rn-Zn}.
\end{proof}

\subsubsection{Simple estimates for theta functions}
\label{sec-3-5-4}
Recall that the \emph{Iwasawa decomposition} asserts that each
$s \in \SL_2(\mathbb{A})$ may be written in the form
$s = n(x) t(y) k$, where
$x \in \mathbb{A}, y \in \mathbb{A}^\times$ and $k$ belongs to
the standard maximal compact subgroup of $\SL_2(\mathbb{A})$.
The decomposition is not unique,
but the quantities $x$ and $|y|$
depend only upon $s$.

Recall that the \emph{height} function
$\htt: [\Mp_2] \rightarrow \mathbb{R}_{>0}$ factors through
$\htt: [\SL_2] \rightarrow \mathbb{R}_{>0}$ where it is given
for $g \in [\SL_2]$ by
$\htt(g) := \max_{\gamma \in \SL_2(F)} \htt_\mathbb{A}(\gamma
g)$,
where
$\htt_{\mathbb{A}} : \SL_2(\mathbb{A}) \rightarrow
\mathbb{R}_{>0}$
is defined with respect to the Iwasawa decomposition
$s = n(x) t(y) k$ by $\htt_{\mathbb{A}}(s) := |y|^{1/2}$.
One has $\int_{[\SL_2]} \htt^{1-\eps} < \infty$
for $\eps > 0$.
\emph{Reduction theory}
says
that the
image of $\htt$ is bounded from below by some $c > 0$ depending
only upon $F$.


Recall that the nontrivial unitary character
$\psi$ of $\mathbb{A}/F$ is regarded as fixed.
\begin{lemma}\label{lem:crude-bound-for-cuspidal-theta-functions}
  Let $A \geq 0$ be fixed.  Let
  $\Psi \in L^1( [\PB^\times])$
  with $\langle \Psi, 1 \rangle = 0$.
  Let
  $\phi \in \mathcal{S}(B_\mathbb{A}^0)$.
  For
  $s \in \Mp_2(\mathbb{A})$,
  one has
  $\theta_{\psi}(\phi,\Psi;s)
  \ll \htt(s)^{-A}
  \mathcal{C}(\phi)
  \|\Psi \|_{L^1}$.
\end{lemma}
\begin{proof}
  Since $\Psi$ has mean zero,
  \[
  \theta_{\psi}(\phi,\Psi;s)
  = \int_{g \in [\PB^\times]}
    \Psi(g)
    \sum_{v \in V - \{0\}}
    \rho_{\Weil}^{\psi,V}(s,\Ad(g))
    \phi(v).
  \]
  By the Iwasawa decomposition
  and reduction theory, we
  may assume that $s = n(x) t(y) k$
  with $|y| \gg 1$.
  Since $B$ is non-split,
  we may fix a compact
  subset $U$ of $\PB^\times_\mathbb{A}$ containing a fundamental
  domain for $[\PB^\times]$.
  Then
  \[
    |\theta_{\psi}(\phi,\Psi;s)|
    \leq \|\Psi \|_{L^1}
    |y|^{3/2}
    \sup_{g \in U}
    \sum_{v \in B^0 - \{0\}}
    |\rho_{\Weil}^{\psi,B^0}(k,\Ad(g))
    \phi(y v)|.
  \]
  Since the Weil representation is continuous (\cite[\S39]{MR0165033}),
  we may reduce to
  the case $k = 1$ and $g = 1$, in which the required
  estimate follows from Lemma
  \ref{lem:cheap-lattice-sum-bound-adelic} of
  \S\ref{sec:simple-estimates-for-lattice-sums}.
\end{proof}

\begin{lemma}\label{lem:crude-bound-for-elementary-theta-functions}
  For 
  $\phi \in \mathcal{S}(\mathbb{A})$ and $s \in
  \Mp_2(\mathbb{A})$,
  one has
  $\theta_{\psi}(\phi;s)
    \ll
    \htt(s)^{1/4}
    \mathcal{C}(\phi)$.
\end{lemma}
\begin{proof}
  We argue as in the proof of Lemma
  \ref{lem:crude-bound-for-cuspidal-theta-functions}, but take
  into account the contribution from $0 \in F$ to the definition
  of $\theta_{\psi}(\phi)$.
\end{proof}


\subsubsection{Main estimate: the case of pure tensors}
\label{sec-3-5-3}
\begin{lemma*}\label{lem:main-result-for-theta-squared}
  Let $\phi_1', \phi_2' \in \mathcal{S}(\mathbb{A})$
  and $\phi_1'', \phi_2'' \in \mathcal{S}(B_\mathbb{A}^0)$.
  Let $\Psi_1, \Psi_2 : [\PB^\times] \rightarrow \mathbb{C}$ be
  integrable functions of mean zero.
  Let $\tau_1, \tau_2 \in F^\times$ be fixed.
  Abbreviate $\theta_i :=
  \theta_{\psi^{\tau_i}}(\phi_i')$
  and $h_i := \theta_{\psi^{\tau_i}}(\phi_i'',\Psi_i)$.
  Then for all $s \in \Mp_2(\mathbb{A})$,
  \[
  \langle
  \theta_1 \cdot \rho_{\reg}(s) h_1,
  \theta_2 \cdot \rho_{\reg}(s) h_2
  \rangle
  = \langle \theta_1, \theta_2 \rangle
  \langle h_1, h_2 \rangle
  + O(\Xi(s) \prod_{i=1,2} \mathcal{C}(\phi_i')
  \mathcal{C}(\phi_i'')
  \|\Psi_i\|_{L^1}
  ).
  \]
\end{lemma*}
\begin{proof}
  The main result of \cite{nelson-theta-squared} gives an estimate nearly
  of the required shape, but instead with the error term
  $\Xi(s) \mathcal{S}(\phi_1') \mathcal{S}(\phi_2')
  \mathcal{S}^{\mathbf{X}}(h_1 \overline{h_2})$, where
  $\mathcal{S}, \mathcal{S}^{\mathbf{X}}$ are adelic Sobolev
  norms whose relevant properties we recall shortly.  By the cuspidality
  of $h_1, h_2$ and axioms (S3b) and (S4e) of
  \cite{michel-2009}, we may replace the expression
  $\mathcal{S}^{\mathbf{X}}(h_1 \overline{h_2})$ first with
  $\mathcal{S}^{\mathbf{X}}(h_1) \mathcal{S}^{\mathbf{X}}(h_2)$
  and then with $\mathcal{S}(h_1) \mathcal{S}(h_2)$.  Our task
  thereby reduces to showing for $i=1,2$ that
  $\mathcal{S}(\phi_i') \ll \mathcal{C}(\phi_i')$ and
  $\mathcal{S}(h_i) \ll \mathcal{C}(\phi_i'') \|\Xi\|_{L^1}$.

  To that end,
  we must recall something about the norms
  $\mathcal{S}$
  (see
  \cite[\S2]{michel-2009} and \cite[\S4.6,
  \S5.3]{nelson-theta-squared} for details).  They have the form
  $\mathcal{S}(v) = \|\Delta^d v\|$, where
  $d \in \mathbb{Z}_{\geq 0}$ is fixed but large enough,
  and
  $\Delta$ acts linearly
  on the space of smooth vectors in any
  unitary representation $\pi$ of $\Mp_2(\mathbb{A})$.
  If $\pi$
  factors as $\pi_\infty \otimes \pi_{\fin}$, then likewise
  $\Delta = \Delta_\infty \otimes \Delta_{\fin}$.
  If $\pi$ is
  the Weil representation on the Schwartz--Bruhat space
  $\mathcal{S}(V_\mathbb{A}) = \mathcal{S}(V_\infty) \otimes
  \mathcal{S}(\hat{V})$ of a quadratic space $V$ over $F$, then
  $\Delta_{\infty}$ is a finite order differential operator with
  polynomial coefficients, hence is continuous for the Schwartz
  topology; since any linear operator on $\mathcal{S}(\hat{V})$
  is continuous,
  it follows that $\Delta$
  defines a continuous operator on $\mathcal{S}(V_\mathbb{A})$
  for the Schwartz--Bruhat topology.

  The continuity of $\Delta$ implies that
  $\mathcal{S}(\phi_i')
  = \|\Delta^d \phi_i'\|
  \ll \mathcal{C}(\phi_i)$, giving one of the two required estimates.
  The operator $\Delta$ is moreover natural
  in that for an equivariant morphism $f : \pi \rightarrow \pi '$
  of
  $\Mp_2(\mathbb{A})$-representations,
  one has $\Delta \circ f = f \circ \Delta$.
  Thus
  \[
  \mathcal{S}(h_i)
  = \|\Delta^d h_i\|
  = \|
  \theta_{\psi^{\tau_i}}(\Delta^d \phi_i'', \Psi_i)
  \|.
  \]
  By Lemma \ref{lem:crude-bound-for-cuspidal-theta-functions} of
  \S\ref{sec-3-5-4},
  we have
  $\mathcal{S}(h_i) \ll \mathcal{C}(\Delta^d \phi_i'')
  \|\Psi_i\|_{L^1}$.
  The continuity of $\Delta$
  gives
  $\mathcal{C}(\Delta^d \phi_i'') \ll \mathcal{C} '(\phi_i'')$,
  and the required estimate follows.
\end{proof}

\subsubsection{Factorization}\label{sec:quantitative-factorization}
Let $V', V''$ be vector spaces over $F$ and
$V := V' \oplus V''$.

\begin{lemma*}\label{lem:quantitative-factorization}
  Let
  $\ell : \mathcal{S}(V_\mathbb{A} ') \otimes
  \mathcal{S}(V_\mathbb{A} '') \rightarrow \mathbb{C}$ be an
  algebraic linear functional on the algebraic tensor product of
  Schwartz--Bruhat spaces satisfying an estimate of the form
  \[
    \ell(\phi' \otimes \phi '')
    \ll \mathcal{C}'(\phi ') \mathcal{C}''(\phi '').
  \]
  Then $\ell$ extends to a continuous functional
  $\ell : \mathcal{S}(V_\mathbb{A}) \rightarrow \mathbb{C}$
  satisfying
  \[
    \ell(\phi) \ll \mathcal{C}(\phi)
  \]
  for all $\phi \in \mathcal{S}(V_\mathbb{A})$, where
  $\mathcal{C}$ depends only upon $\mathcal{C} '$
  and $\mathcal{C} ''$.
\end{lemma*}
\begin{proof}
  This is essentially the Schwartz kernel theorem, as
  extended by Bruhat \cite[\S5]{MR0140941}.  For concreteness,
  we sketch a proof.  There exists
  $\nu = \prod \nu_\mathfrak{p} \in C_c^\infty(V'_\mathbb{A})$
  such that $\sum_{\lambda \in V'} \nu(\lambda + x)^2 = 1$ for
  all $x \in V'_\mathbb{A}$, thus $\nu$ is a ``square-root of a
  partition of unity.''  Let $m \in \mathbb{A}^\times$ be large
  enough that the map
  $\supp(\nu) \rightarrow V'_\mathbb{A} / m V'$ is injective.
  Set $\Lambda_1 := V'$ and $\Lambda_2 := m^{-1} V'$.
  Fix non-degenerate
  bilinear forms
  $V'_\mathbb{A} \otimes V'_\mathbb{A}
  \rightarrow \mathbb{A}$
  and
  $V''_\mathbb{A} \otimes V''_\mathbb{A} \rightarrow
  \mathbb{A}$;
  denote them by
  $(x,y) \mapsto x \cdot y \in \mathbb{A}$.
  By Fourier inversion on
  the compact group $V'_\mathbb{A} / m V'$, there exists
  $c > 0$ so that for all $f \in C^\infty(V'_\mathbb{A})$ and
  $x \in V'_\mathbb{A}$,
  \begin{equation}\label{eq:fourier-expansion-of-half-of-partitino-of-unity}
    \nu(x)^2 f(x) = c \nu(x) \sum_{\lambda_2 \in \Lambda_2}
    \psi(\lambda_2 \cdot x) \int_{y \in V'_\mathbb{A}} \nu(y) f(y)
    \psi(-\lambda_2 \cdot y).
  \end{equation}
  For
  $t' \in V'_\mathbb{A}, t'' \in V''_\mathbb{A}$,
  we apply
  \eqref{eq:fourier-expansion-of-half-of-partitino-of-unity}
  to
  $f(x) := \phi( (t'' - \lambda_1) + x)$
  and set $x := t' + \lambda_1$,
  giving
  \begin{align}
    \phi(t' + t'')
    &=
      \sum_{\lambda_1 \in \Lambda_1}
      \nu(t' + \lambda_1)^2
      \phi( (t'' - \lambda_1) + (t' + \lambda_1))
    \\
    &=
      \sum_{\substack{
      \lambda_1 \in \Lambda_1 \\
    \lambda_2 \in \Lambda_2
    }
    }
    \phi_{\lambda_1,\lambda_2}'(t')
    \phi_{\lambda_1,\lambda_2}''(t''),
  \end{align}
  where
  \begin{equation}
    \phi_{\lambda_1,\lambda_2}'(t')
    :=
    c \nu (t' + \lambda_1)
    \psi(\lambda_2 \cdot (t' + \lambda_1)),
  \end{equation}
  \begin{equation}\label{eq:}
    \phi_{\lambda_1,\lambda_2}''(t'')
    = \int_{y \in \mathbb{A}}
    \nu(y) \phi(t'' - \lambda_1 + y) \psi(- \lambda_2 \cdot y).
  \end{equation}
  ``Integration by parts''
  gives readily that
  \begin{equation}
    \sum 
    |\ell(\phi'_{\lambda_1,\lambda_2} \otimes \phi''_{\lambda_1,\lambda_2})|
    \ll
    \sum
    |\mathcal{C} '(\phi'_{\lambda_1,\lambda_2})
    \mathcal{C} ''(\phi''_{\lambda_1,\lambda_2})|
    \ll
    \mathcal{C}(\phi)
  \end{equation}
  for suitable $\mathcal{C}(\phi)$.
  It follows that $\ell$ admits
  the required extension and satisfies
  the required estimate.
\end{proof}

\subsubsection{Main estimate: the general case\label{sec:main-estimate-for-equidistribution-pairs-theta}}
\label{sec-3-5-5}
Temporarily denote by $\mathcal{A}_0$
denote the space
of integrable functions
$\Psi : [\PB^\times] \rightarrow \mathbb{C}$  of mean zero.
Let $\mathcal{E}_{\tau_1,\tau_2} : \mathcal{S}(B_\mathbb{A}) \otimes
\mathcal{S}(B_\mathbb{A})
\otimes \mathcal{A}_0 \otimes \mathcal{A}_0 \rightarrow
\mathbb{C}$
denote the sesquilinear form given
for
$\phi_i = \phi_i' \otimes \phi_i'' \in
\mathcal{S}(B_\mathbb{A})$
with $\phi_1',\phi_2' \in \mathcal{S}(\mathbb{A}),
\phi_1'',\phi_2'' \in \mathcal{S}(B_\mathbb{A}^0)$
by
\begin{align*}
  \mathcal{E}_{\tau_1,\tau_2}(\phi_1,
  \phi_2,
  \Psi_1, \Psi_2)
  :=
  \langle
  \theta_1 h_1, \theta_2 h_2 \rangle
  - \langle \theta_1, \theta_2 \rangle
  \langle h_1, h_2 \rangle,
\end{align*}
where we abbreviate
$\theta_i := \theta_{\tau_i}(\phi_i')$ and
$h_i := \theta_{\tau_i}(\phi_i'',\Psi_i)$.
(The definition makes sense: \emph{a priori} estimates as in
\S\ref{sec-3-5-4} and the density of
$\mathcal{S}(\mathbb{A}) \otimes \mathcal{S}(B_\mathbb{A}^0)$ in
$\mathcal{S}(B_\mathbb{A})$ allow us to extend
$\mathcal{E}_{\tau_1,\tau_2}$ continuously
from its initial domain.)




\begin{proposition}\label{prop:main-error-estimate-global-adelic-general}
  For $\phi_1, \phi_2 \in \mathcal{S}(B_\mathbb{A}),
  \Psi_1,\Psi_2 \in \mathcal{A}_0$
  and $s \in \Mp_2(\mathbb{A})$,
  one has
  with $\rho_0^{\tau}(s) := \rho_{\Weil}^{\psi^{\tau}, B^0}(s)$
  the estimate
  \begin{equation}
    \mathcal{E}_{\tau_1,\tau_2}((1 \otimes \rho_0^{\tau_1}(s))\phi_1,
    (1 \otimes \rho_0^{\tau_2}(s))\phi_2,
    \Psi_1,\Psi_2)
    \ll \Xi(s)
    \prod_{j=1,2}
    \Sob(\phi_j)
    \|\Psi_j\|_{L^1}.
  \end{equation}
  The implied constant and the uniformity of the continuity of
  $\Sob(\phi_j)$
  depend at most upon
  $\psi,\tau_1,\tau_2,F,B$.
  The operator
  $1 \otimes \rho^{\tau_1}(s)$
  is defined as in \S\ref{sec:factorization-weil-repn}.
\end{proposition}
\begin{proof}
  The lemma of \S\ref{sec:quantitative-factorization}
  reduces the general case
  of Proposition
  \ref{prop:main-error-estimate-global-adelic-general}
  to the special case
  in which
  $\phi_i = \phi_i' \otimes \phi_i''$ for $i=1,2$,
  which follows
  from
  the lemma of \S\ref{sec-3-5-3}
  upon recalling from
  \eqref{eq:equivariance-ttf}
  that $\theta_{\psi_\tau}$
  intertwines $\rho_{0}^{\tau}$
  with $\rho_{\reg}$.
\end{proof}

\subsubsection{Invariance properties}\label{sec:equivariance-summary-for-E-tau-tau}
We record these for later use.
\begin{lemma*}\label{lem:equivariance-summary-for-E-tau-tau}
  For
  $g_1, g_2 \in \PB^\times_\mathbb{A}$
  and
  $s \in \Mp_2(\mathbb{A})$,
  one has
  \begin{align*}
    \mathcal{E}_{\tau_1,\tau_2}(\phi_1, \phi_2, \Psi_1, \Psi_2)
    &=
      \mathcal{E}_{\tau_1,\tau_2}(\mathfrak{S} \phi_1, \phi_2,
      \Psi_1, \Psi_2)
    \\
    &=
      \mathcal{E}_{\tau_1,\tau_2}(\phi_1, \mathfrak{S} \phi_2,
      \Psi_1, \Psi_2)
    \\
    &=
      \mathcal{E}_{\tau_1,\tau_2}(\Ad(g_1) \phi_1, \Ad(g_2) \phi_2,
      \rho_{\reg}(g_1) \Psi_1,       \rho_{\reg}(g_2)\Psi_2)
    \\
    &=
      \mathcal{E}_{\tau_1,\tau_2}(\rho_{\Weil}^{\psi^{\tau_1},B}(s) \phi_1, \rho_{\Weil}^{\psi^{\tau_2},B}(s) \phi_2,
        \Psi_1, \Psi_2).
    \end{align*}
  \end{lemma*}
\begin{proof}
  The first two identities follow from
  \eqref{eq:action-of-S-on-pure-tensors-for-B} and 
  \eqref{eqn:O1-invariance-elementary-theta-fn}, the remaining
  from
  \eqref{eq:equivariance-etf}, \eqref{eq:equivariance-ttf},
  \eqref{eq:equivariance-ttf-2},
  \eqref{eq:action-of-S-on-pure-tensors-for-B-0}, \eqref{eq:action-of-S-on-pure-tensors-for-B}
  and the translation invariance
  of the Petersson inner product.
\end{proof}


\subsection{Simillitude theta functions\label{sec:non-traditional-theta-lifts}}
\label{sec-3-2-9}

\subsubsection{Weil representation\label{sec:similitudes-global}}
\label{sec-3-2-4}
For each place $\mathfrak{p}$ of $F$, let
$\Omega_{\mathfrak{p}}$ denote the representation of
$\PGL_2(F_{\mathfrak{p}}) \times \GO(B_\mathfrak{p})$ attached
as in \S\ref{sec:defn-local-omega}  to the tuple
$(F_\mathfrak{p},B_{\mathfrak{p}},\psi_{\mathfrak{p}})$.  Let
$\Omega$ denote the restricted tensor product of the spaces
$\Omega_{\mathfrak{p}}$ with respect to the distinguished
elements, which we denote now by $\phi_{\mathfrak{p}}^0 \in
\Omega_{\mathfrak{p}}$.
We may and shall identify $\Omega$ with the space of
functions
$\phi : \mathbb{A}^\times \times B_\mathbb{A} \rightarrow
\mathbb{C}$ such that
\begin{itemize}
\item For each $t \in \mathbb{A}^\times$,
  the function
  $\phi[t] : B_\mathbb{A} \rightarrow \mathbb{C}$
  given by $\phi[t](x) := \phi(t,x)$
  belongs to the Schwartz--Bruhat space
  $\mathcal{S}(B_\mathbb{A})$;
\item $\phi(z^2 t, x) = \phi(t, z x)$
  for all $z,t \in \mathbb{A}^\times, x \in B_\mathbb{A}$.
\item There is a compact subset
  $C$ of $\mathbb{A}^\times / \mathbb{A}^{\times 2}$
  such that $\phi[t] = 0$ for all $t \notin C$
  (i.e., for all $t \in \mathbb{A}^\times$
  whose image in $\mathbb{A}^\times / \mathbb{A}^{\times 2}$
  lies outside $C$);
\item There is an open subgroup
  $U$ of $\mathbb{A}^\times / \mathbb{A}^{\times 2}$
  such that $\phi[t u] = \phi[t]$
  for all $t \in \mathbb{A}^\times, u \in U$.
\end{itemize}
We equip $\Omega$ with the invariant hermitian norm $\|.\|$
obtained by tensoring those
on the factors $\Omega_\mathfrak{p}$, thus
\begin{equation}\label{eqn:inner-product-on-Omega-1-adelic}
  \|\phi\|^2_{\Omega}
  :=
  \int_{t \in \mathbb{A}^\times / \mathbb{A}^{\times 2}}
  |t|^2
  \int_{x \in B_\mathbb{A}}
  |\phi|^2(t,x).
\end{equation}
The group $\PGL_2(\mathbb{A}) \times \GO(B_\mathbb{A})$ acts on
$\Omega$ by the representation $\rho_{\Weil}$ obtained as the
restricted tensor product of those defined in
\S\ref{sec:defn-local-omega}.  We define
$\mathfrak{S} : \Omega \rightarrow \Omega$
and
(for $g \in \PB^\times_\mathbb{A}$)
$\Ad(g) : \Omega \rightarrow \Omega$
as in
\S\ref{sec:defn-local-omega}.  Note that $\mathfrak{S}$ does
\emph{not} preserve pure tensors: for $\phi
= \otimes \phi_\mathfrak{p} \in \Omega$,
\begin{align*}
  \mathfrak{S} \phi (t,x)
  &= (\phi (t,x) + \phi (t,x -
    \nr(x)))/2,
    \\
  \otimes \mathfrak{S} \phi_\mathfrak{p}(x,t)
  &= \prod  (\phi (t_\mathfrak{p},x_\mathfrak{p}) + \phi (t_\mathfrak{p},x_\mathfrak{p} -
\nr(x_\mathfrak{p} )))/2,
\end{align*}
and these are not in general the same.
They \emph{do} coincide if $\# \{ \mathfrak{p} : \mathfrak{S} \phi_\mathfrak{p}
\neq \phi_\mathfrak{p} \}
\leq 1$.

\subsubsection{Theta functions}
For $\phi \in \Omega$,
$s \in \PGL_2(\mathbb{A}), g \in \GO(B_\mathbb{A})$,
set
\begin{equation}
  \label{defn:theta-knerle-for-Omega}
  \Theta(\phi;s,g)
  :=
  \frac{1}{2} \sum_{\tau \in F^\times / F^{\times 2}}
  \sum_{x \in B}
  \rho_{\Weil}(s,g)
  \phi(\tau,x).
\end{equation}
The sum is well-defined, converges absolutely
and defines a smooth function
$\Theta(\phi)$
on $[\PGL_2] \times [\GO(B)]$.
For a cusp form $\Psi : [\PB^\times] \rightarrow \mathbb{C}$
and $s \in \PGL_2(\mathbb{A})$,
set
\[
\Theta(\phi,\Psi;s) := \int_{g \in [\PB^\times]}
\Psi(g) \Theta(\phi;s,\Ad(g)).
\]
The integral (together with similar integrals below)
converges absolutely
and defines a cusp form
$\Theta(\phi,\Psi) : [\PGL_2] \rightarrow \mathbb{C}$.

\begin{remark*}
  $\Theta(\phi,\Psi)$
  is not a theta lift in the traditional sense:
  the integral
  in its definition is with respect to the orthogonal group of $B^0$ rather than that of $B$.
\end{remark*}


\subsubsection{Fourier expansion}\label{sec:four-expans}
Let $\phi \in \Omega$, and let $\Psi : [\PB^\times] \rightarrow
\mathbb{C}$
be a cusp form.
\begin{lemma*}
  For
  $x \in \mathbb{A}$,
  $y \in \mathbb{A}^\times$,
  one has
  \[
  \Theta(\phi,\Psi;n(x) a(y))
  =
  \sum_{\tau \in F^\times}
  \psi(\tau x)
  W(\Theta(\phi,\Psi), \tau y)
  \]
  where
  $W(\Theta(\phi,\Psi),y)
  :=
  \int_{g \in [\PB^\times]}
  \Psi(g)
  \sum_{\gamma \in \PB^\times}
  |y|
  \phi(y \nr(\gamma)^{-1}, g^{-1} \gamma g)$.
\end{lemma*}
\begin{proof}
  By direct unfolding
  as in \cite{MR0333081},
  one has for $g \in
  \PB^\times_{\mathbb{A}}$ that
  \begin{equation*}
    \begin{split}
      \Theta(\phi, n(x) a(y);\Ad(g))
      &=
      \frac{1}{2}
      \sum_{\tau \in F^\times / F^{\times 2}}
      |y| \phi(0,\tau y) \\
      &\quad  + \sum_{\tau \in F^\times}
      \psi(\tau x)
      W(\Theta(\phi),\Ad(g),\tau y),
    \end{split}
  \end{equation*}
  where
  $W(\Theta(\phi),\Ad(g),y)
  :=
  \sum_{\gamma \in \PB^\times}
  |y|
  \phi(y \nr(\gamma)^{-1}, g^{-1} \gamma g)$.
  We conclude by integrating against $\Psi$.
\end{proof}


\subsubsection{Restriction to $\SL_2$}
Let $\phi, \Psi$ be as above.

\begin{lemma*}
  Suppose
  that
  for each $y \in \mathbb{A}^\times$,
  one has
  $\phi [y] = \phi '[y] \otimes \phi ''[y]$
  for some
  $\phi '[y] \in \mathcal{S}(\mathbb{A})$,
  $\phi ''[y] \in \mathcal{S}(B_\mathbb{A}^0)$.
  Then for $y \in \mathbb{A}^\times$ and $s \in \SL_2(\mathbb{A})$,
  one has
  \begin{equation}\label{eqn:simliitude-theta-expand-as-sum-of-pure-tensors-and-over-a}
    \Theta(\phi,\Psi;s a(y))
    =
    \frac{1}{2}
    \sum_{\tau \in F^\times \backslash F^{\times 2}}
    |y|
    \theta_{\psi^\tau}(\phi'[\tau y];s)
    \theta_{\psi^\tau}(\phi''[\tau y],\Psi;s).
  \end{equation}
\end{lemma*}
\begin{proof}
  We derive first using \eqref{defn:theta-knerle-for-Omega}
  that for $g \in \O(B_{\mathbb{A}})$,
  \[
  \Theta(\phi;s a(y),g)
  =
  \frac{1}{2}
  \sum_{\tau \in F^\times \backslash F^{\times 2}}
  |y|
  \theta_{\psi^\tau}(\phi[a y];s,g),
  \]
  hence by \eqref{eqn:Factorization-of-theta-fns}
  that for $g \in \O(B_\mathbb{A}^0)$,
  \[
  \Theta(\phi;s a(y),g)
  =
  \frac{1}{2}
  \sum_{\tau \in F^\times \backslash F^{\times 2}}
  |y|
  \theta_{\psi^\tau}(\phi'[\tau y];s)
  \theta_{\psi^\tau}(\phi''[\tau y];s,g).
  \]
  We integrate against $\Psi$ to conclude.
\end{proof}

\subsubsection{Unfolding the inner
  product}\label{sec:unfold-inner-product-nontraditional-theta-over-sl2}
Let  $\phi_1, \phi_2 \in \Omega$.
Let $\Psi_1, \Psi_2 : [\PB^\times] \rightarrow \mathbb{C}$
be cusp forms.

\begin{lemma*}
  Suppose
  that
  for each $y \in \mathbb{A}^\times$,
  one has
  $\phi_i [y] = \phi_i '[y] \otimes \phi_i ''[y]$
  for some
  $\phi_i '[y] \in \mathcal{S}(\mathbb{A})$,
  $\phi_i ''[y] \in \mathcal{S}(B_\mathbb{A}^0)$.
  Abbreviate
  $\theta_i := \theta_{\psi^{\tau_i}}(\phi_i'[\tau_i y])$,
  $h_i := \theta_{\psi^{\tau_i}}(\phi_i''[\tau_i y],\Psi_i)$.
  Then the identity
  \begin{equation}\label{eq:inner-product-pgl2-made-into-sl2}
    \langle \Theta(\phi_1,\Psi_1),
    \Theta(\phi_2,\Psi_2) \rangle_{\PGL_2}
    =
    \int_{y \in \mathbb{A}^\times / F^\times \mathbb{A}^{\times 2}}
    |y|^2
    \frac{1}{2^2}
    \sum_{\tau_1,\tau_2 \in F^\times / F^{\times 2}}
    \langle \theta_1 h_1, \theta_2 h_2 \rangle_{\SL_2}
  \end{equation}
  holds, with both sides converging absolutely.
\end{lemma*}
\begin{proof}
  The LHS is an inner product of cusp forms, hence convergent.  On the RHS, we may
  replace the $y$-integral by a finite sum, since the domain
  $\mathbb{A}^{\times} / F^{\times } \mathbb{A}^{\times 2}$ is
  compact and the integrand is invariant under an open subgroup.
  For individual $y$, the sum over $\tau_1, \tau_2$ has only
  finitely many nonzero summands, each of which consists of an
  inner product whose convergence is clear
  (see \S\ref{sec-3-5-4}).  The expansion
  \eqref{eqn:simliitude-theta-expand-as-sum-of-pure-tensors-and-over-a}
  implies for $y \in \mathbb{A}^\times, s \in \SL_2(\mathbb{A})$
  that
  $\Theta(\phi_i,\Psi_i,a(y) s) = \frac{1}{2} \sum _{\tau_i \in
    F^\times / F^{\times 2} } |y| \theta_i(s) h_i(s)$,
  so the required identity follows from the formula
  \eqref{eq:sl2-vs-pgl2-integrals} relating integrals over
  $[\PGL_2]$ and $[\SL_2]$.
\end{proof}

\begin{remark*}
  In this paper, we consider several expressions shaped like the
  RHS of \eqref{eq:inner-product-pgl2-made-into-sl2}.  On a
  first (or perhaps on any) reading, one should focus on the
  contributions from $y = \tau_1 = \tau_2 = 1$; under some class
  number and unit group restrictions, these turns out to be the
  relevant ones for the proof of Theorem
  \ref{thm:main-result-for-microlocal-stuff}.  (We considered
  imposing such restrictions for the sake of presentation, but
  found that doing so obfuscated rather than simplified.)
\end{remark*}

\subsection{Inner product formulas}
\label{sec-3-7}

\subsubsection{Elementary theta functions}\label{sec:elem-theta-ipf}
We recall part of \cite[Thm 2]{nelson-theta-squared}.
\begin{lemma*}
  Suppose
  $\phi_1, \phi_2 \in \mathcal{S}(\mathbb{A})$
  satisfy $\phi_1(x) = \phi_1(-x), \phi_2(x) = \phi_2(-x)$.
  Let $\tau_1, \tau_2 \in F^\times$.
  Set $\theta_1 := \theta_{\psi^{\tau_1}}(\phi_1),
  \theta_2 := \theta_{\psi^{\tau_1}}(\phi_2)$.
  Then
  $\langle \theta_1, \theta_2 \rangle_{\SL_2}= 0$
  unless $\tau_1 = \tau_2$,
  in which case
  $\langle \theta_1, \theta_2 \rangle_{\SL_2}
  =
  2
  \langle \phi_1, \phi_2  \rangle_{L^2(\mathbb{A})}$.
\end{lemma*}
\subsubsection{Ternary theta lifts}\label{sec:ternary-theta-ipf}
As in
\cite[\S12.3]{nelson-variance-73-2},
we explicate
Gan--Takeda
\cite[Thm 6.6]{MR2837015}
(compare with \cite[Prop 2.8 (i)]{MR3291638}).
\begin{lemma*}
  Let $\pi_1, \pi_2 \subseteq L^2([\PB^\times])$ be
  cuspidal automorphic representations that are
  not one-dimensional.
  Let $\Psi_1 \in \pi_1, \Psi_2 \in \pi_2$
  and $\phi_1, \phi_2 \in \mathcal{S}(B_\mathbb{A}^0)$.
  Let $\tau \in F^\times$.
  Set $h_i :=   \theta_{\psi^{\tau}}(\phi_i,\Psi_i)$
  for $i=1,2$.
  \begin{enumerate}
  \item If $\pi_1 \neq \pi_2$, then
    $\langle h_1, h_2 \rangle_{\SL_2} = 0$.
  \item Suppose $\pi_1 = \pi_2 =: \pi$.  Let $S$ be a finite set of places of $F$
    that containing all
    archimedean places, as well as any finite places at which $B$
    ramifies, and that is sufficiently large in terms of  $\Psi_i, \phi_i$.
    Then
    $\langle h_1, h_2 \rangle_{\SL_2}$ equals
    \[
    \frac{L^{(S)}(\pi,\tfrac{1}{2})}{\zeta_F^{(S)}(2)}
    (\prod_{\mathfrak{p} \notin S}
    \vol(K_\mathfrak{p}) )
    \int_{g \in \PB^\times_S}
    \langle \Ad(g) \phi_1, \phi_2 \rangle_{L^2(B_\mathbb{A}^0)}
    \langle \pi(g) \Psi_1, \Psi_2 \rangle_{\PB^\times}
    \]
    with $L^{(S)}(\pi,\tfrac{1}{2})$
    as in 
    \S\ref{sec:standard-l-function}.
  \end{enumerate}
\end{lemma*}

\subsubsection{Induction to $\Omega$}\label{sec:induction-omega}
We now combine the previous two lemmas and sum them up.
Temporarily denote by $\mathcal{A}_0$ the space of cusp forms
$\Psi : [\PB^\times] \rightarrow \mathbb{C}$ that are
  orthogonal to all one-dimensional representations. 
  For $\tau_1, \tau_2 \in F^\times$, let
  $\mathfrak{m} : \Omega \otimes \Omega \otimes \mathcal{A}_0
  \otimes \mathcal{A}_0 \rightarrow \mathbb{C}$ denote the
  sesquilinear form given for $\phi_1, \phi_2 \in \Omega$
  admitting factorizations
  $\phi_i[y] = \phi_i'[y] \otimes \phi_i''[y]$ by
  \begin{equation}
    \mathfrak{m}(\phi_1,\phi_2,\Psi_1,\Psi_2)
    :=
    \int_{y \in \mathbb{A}^\times / F^{\times }
      \mathbb{A}^{\times 2}}
    |y|^2
    \frac{1}{2^2}
    \sum_{\tau_1,\tau_2 \in F^\times / F^{\times 2}}
    \langle \theta_1, \theta_2 \rangle
    \langle h_1, h_2 \rangle,
  \end{equation}
  where we abbreviate
  $\theta_i := \theta_{\psi^{\tau_i}}(\phi_i'[\tau_i y])$ and
  $h_i := \theta_{\psi^{\tau_i}}(\phi_i''[\tau_i y],\Psi_i)$.
  The relevance of $\mathfrak{m}$ may be inferred from
  \S\ref{sec:unfold-inner-product-nontraditional-theta-over-sl2}.

  The definition makes sense: as in the proof of
  \S\ref{sec:unfold-inner-product-nontraditional-theta-over-sl2},
  the $y$-integral is really a finite sum, and the sum over
  $\tau_1, \tau_2$ has only finitely many nonzero summands.
  Each
  summand defines a sesquilinear form on
  $\mathcal{S}(\mathbb{A}) \otimes \mathcal{S}(B_\mathbb{A}^0)
  \otimes \mathcal{A}_0$
  that extends continuously to
  $\mathcal{S}(B_\mathbb{A}) \otimes \mathcal{A}_0$ by the
  \emph{a priori} estimates of \S\ref{sec-3-5-4}.

\begin{lemma*}
  Let $\pi_1, \pi_2 \subseteq L^2([\PB^\times])$ be
  cuspidal automorphic representations that are
  not one-dimensional.
  Let $\Psi_1 \in \pi_1, \Psi_2 \in \pi_2$.
  Let $\phi_1, \phi_2 \in \Omega$.
  \begin{enumerate}
  \item If $\pi_1 \neq \pi_2$
    then $\mathfrak{m}(\phi_1,\phi_2,\Psi_1,\Psi_2) = 0$.
  \item Suppose $\pi_1 = \pi_2 =: \pi$.
    Let $S$ be a large enough finite set of places.
    Then
    $\mathfrak{m}(\phi_1,\phi_2,\Psi_1,\Psi_2)$ equals
    \[
    \frac{L^{(S)}(\pi,\tfrac{1}{2})}{\zeta_F^{(S)}(2)}
    (\prod_{\mathfrak{p} \notin S}
    \vol(K_\mathfrak{p}) )
    \int_{g \in \PB^\times_S}
    \langle \Ad(g) \mathfrak{S} \phi_1, \mathfrak{S} \phi_2 \rangle_{\Omega}
    \langle \pi(g) \Psi_1, \Psi_2 \rangle_{\PB^\times}.
    \]
  \end{enumerate}
\end{lemma*}
\begin{proof}
  It suffices to consider
  the case that $\phi_1, \phi_2$
  admit factorizations as in the definition
  of $\mathfrak{m}$.
  By \eqref{eqn:O1-invariance-elementary-theta-fn}
  and \eqref{eq:action-of-S-on-pure-tensors-for-B},
  we may assume that $\mathfrak{S} \phi_i = \phi_i$,
  or equivalently, that $\phi_i'(t,x) = \phi_i'(t,-x)$.
  By the lemmas of \S\ref{sec:elem-theta-ipf}
  and \S\ref{sec:ternary-theta-ipf},
  we have $\langle \theta_1, \theta_2 \rangle = 0$
  unless $\tau_1 = \tau_2$
  and then
  $\langle h_1, h_2 \rangle = 0$
  unless $\pi_1 = \pi_2$;
  in that case,
  the formulas from those lemmas
  and the identities
  \[
  \langle \phi_1'[y \tau], \phi_2'[y \tau] \rangle_{L^2(\mathbb{A})}
  \langle \Ad(g) \phi_1''[y \tau], \phi_2''[y \tau]
  \rangle_{L^2(B_\mathbb{A}^0)}
  \]
  \[
  = 
  \langle \Ad(g) \phi_1[y \tau], \phi_2[y \tau] \rangle_{L^2(B_\mathbb{A})}
  \]
  and (see \eqref{eq:integral-formula-involving-squares}, \eqref{eqn:inner-product-on-Omega-1-adelic})
  \[
  \int_{y \in \mathbb{A}^{\times} / F^\times \mathbb{A}^{\times
      2}
  }
  |y|^2
  \frac{1}{2}
  \sum_{\tau \in F^\times / F^{\times 2}}
  \langle \Ad(g) \phi_1[y \tau], \phi_2[y \tau]
  \rangle_{L^2(B_\mathbb{A})}
  \]
  \[
  =
  \int_{y \in \mathbb{A}^{\times} /  \mathbb{A}^{\times
      2}
  }
  |y|^2
  \langle \Ad(g) \phi_1[y], \phi_2[y]
  \rangle_{L^2(B_\mathbb{A})}
  =
  \langle \Ad(g) \phi_1, \phi_2 \rangle_{\Omega}
  \]
  combine to give the required conclusion.
\end{proof}

\section{Estimates for general quantum variance sums\label{sec:estimates-general-var}}
\label{sec-4}
In this section we introduce general families of quantum
variance sums, propose a candidate for their leading
asymptotics, and state a general ``estimate''
comparing the two.
\subsection{Notation}
\label{sec-4-1}
Let $F$ be a number field with adele ring $\mathbb{A}$.
Fix a nontrivial
unitary character $\psi$ of $\mathbb{A}/F$.
Let $B$ be a non-split quaternion algebra over $F$.
Fix a maximal order $R \subseteq B$
and a finite set $S$ of places of $F$,
containing
all archimedean places as well as any finite places at which $B$
ramifies.
Retain the (unsurprising) notation of \S\ref{sec-3-1}.

Since $B$ is non-split, the quotient $[\PB^\times]=
\PB^\times \backslash \PB^\times_\mathbb{A}$ is compact,
and so $L^2([\PB^\times])$ is completely reducible.
Let $A^{\flat}$ denote the set of irreducible subrepresentations
of the Hilbert space $L^2([\PB^\times])$.  For each
$\pi^{\flat} \in A^{\flat}$, let $\pi \leq \pi^{\flat}$ denote the subspace of
smooth vectors.
Set $A := \{\pi : \pi^\flat \in A^{\flat}\}$.
Let $\mathcal{A}$ denote the algebraic direct sum
$\oplus_{\pi \in A} \pi$,
regarded as a pre-unitary representation
of the group $\PB^\times_\mathbb{A}$.

We introduce the following additional notation:
\begin{itemize}
\item $K = \prod K_\mathfrak{p}$: a maximal compact
  subgroup of $\PB^\times_\mathbb{A}$.
  For $\mathfrak{p} \notin S$,
  we assume that $K_\mathfrak{p} \leq \PB^\times_{\mathfrak{p}}$ is the image
  of $R_\mathfrak{p}^\times$.


\item $\mathcal{A}_0 \leq \mathcal{A}$: the orthogonal
  complement of the one-dimensional subrepresentations.  (We had
  earlier, in \S\ref{sec-3-5} and \S\ref{sec:induction-omega}, used
  the same symbol to denote some \emph{larger} spaces
  than what we call here $\mathcal{A}_0$.  This abuse of
  notation should introduce no confusion.)
\item $A_0 := \{\pi \in A : \pi \subseteq \mathcal{A}_0\} =
  \{\pi \in A : \dim(\pi) > 1\}$,
  so that $\mathcal{A}_0 = \oplus_{\pi \in A_0} \pi$.
\item $\mathcal{A}^S := \{\varphi \in \mathcal{A} :
  \rho_{\reg}(k) \varphi = \varphi  \text{ for all }
  k \in K_\mathfrak{p}, \mathfrak{p} \notin S
  \}, \mathcal{A}_0^S := \mathcal{A}_0 \cap \mathcal{A}^S$:
  the ``unramified outside $S$''
  subspaces of $\mathcal{A}, \mathcal{A}_0$.
\item $A^S := \{\pi \in A : \pi \cap \mathcal{A}^S \neq \{0\}\},
  {A}_0^S := {A}^S \cap {A}_0$:
  the subsets consisting of those $\pi$ that are unramified outside $S$.
\item $\mathcal{B}(V)$, for $V$ a $K$-invariant subspace of $\mathcal{A}$:
  an orthonormal basis for the closure of $V$
  that consists 
  of $K$-isotypic elements of $V$.
  
\end{itemize}
Fix Haar measures on $\PB^\times_S$ and $[\PB^\times]$; we do
not require any compatibility between them.  Because $B$ is
non-split, each $\pi \in A_0$ is cuspidal.
Let $L^{(S)}(\pi,s)$,
$L^{(S)}(\ad \pi,s)$ 
be as in \S\ref{sec:standard-l-function}.
\subsection{Key definitions\label{sec:main-general-estimate-key-defns}}
\label{sec-4-2}
By
\eqref{eq:rapid-convergence-of-pretrace-formula},
\eqref{eq:polynomial-growth-of-reps}
and \eqref{eq:HL},
the sums considered in the definitions
to follow
converge absolutely.

\subsubsection{The basic distributions}
\label{sec:omega-pi}
For $\pi \in A^S$,
define $\omega_\pi : C_c^\infty(\PB^\times_S) \otimes \mathcal{A}^S
\rightarrow \mathbb{C}$
by
$\omega_\pi(f,\Psi)
:=
\sum_{\varphi \in \mathcal{B}(\pi \cap \mathcal{A}^S)}
\langle \varphi, \Psi \cdot \pi(f) \varphi  \rangle$.
The definition is independent of the choice of orthonormal
basis.

\begin{example*}\label{example:harmonic-sum-attached-to-orth-proj}
  If $\pi(f) = 0$,
  then $\omega_\pi(f,\Psi) = 0$.
  If $\pi(f)$
  is the orthogonal projector onto a
  one-dimensional subspace $\mathbb{C} \varphi$ of $\pi$ with
  unit basis vector $\varphi$,
  then
  $\omega_\pi(f,\Psi) = \langle \varphi, \Psi \varphi  \rangle$.
\end{example*}

\subsubsection{Quantum variance sums}
For $f \in C_c^\infty(\PB^\times_S)$,
define
the sesquilinear form $\mathcal{V}_f : \mathcal{A}_0^S
\otimes \mathcal{A}_0^S \rightarrow \mathbb{C}$
by
\[\mathcal{V}_f(\Psi_1,\Psi_2)
:=
\sum_{\pi \in A_0^S}
L^{(S)}(\ad \pi,1)
\omega_\pi( f,\Psi_1)
\overline{
  \omega_\pi( f,\Psi_2) }.
\]

\subsubsection{Proposed limiting variance}
For $f \in C_c^\infty(\PB^\times_S)$,
define the sesquilinear form
$\mathcal{M}_f : \mathcal{A}_0^S \otimes \mathcal{A}_0^S
\rightarrow \mathbb{C}$
by requiring for  $\Psi_1 \in \pi_1 \in A_0^S, \Psi_2 \in \pi_2
\in A_0^S$
that $\mathcal{M}_f(\Psi_1,\Psi_2) := 0$
unless $\pi_1 = \pi_2 =: \pi$,
in which case
\[
\mathcal{M}_f(\Psi_1,\Psi_2)
:=
c_3
L^{(S)}(\pi,\tfrac{1}{2})
\int_{g \in \PB^\times_S}
\langle \Ad(g) \mathfrak{S} f, \mathfrak{S} f \rangle_{L^2(\PB^\times_S)}
\langle \pi(g) \Psi_1, \Psi_2 \rangle_{\PB^\times}
\]
where
\begin{equation}\label{eq:defn-of-c}
  c_3 :=
  \zeta_F^{(S)}(2)  \vol([\PB^\times])^{-1}.
\end{equation}
The integral converges absolutely (see
\S\ref{sec:local-convergence-lemmas}).

\subsubsection{Thickening
  $\PB^\times$ inside $B$}\label{sec:heartsuit}

  Fix, once and for all, a nonzero element $W_S \in
  C_c^\infty(F_S^\times)$.
  For $\tau \in F^\times$,
  define the linear map
  $\heartsuit^{\tau} : C_c^\infty(\PB^\times_S)
  \rightarrow \mathcal{S}(B_S)$
  by
  \[
  \heartsuit^{\tau} f(x)
  :=
  \frac{W_S(\tau \nr(x))}{|\tau \nr(x)|_S}
  1_{B_S^\times}(x)
  f(\pr(x)),
  \]
  where $\pr : B_S^\times \rightarrow \PB^\times_S$
  denotes the natural projection.

\subsection{Statement of main result}
\label{sec-4-3}
The statement involves the metaplectic group
(\S\ref{sec:global-metaplectic-gp}) and the Weil representation
(\S\ref{sec:weil-repn-global}).
For $s \in \Mp_2(F_S)$,
we abbreviate
$\rho^{\tau}(s) := \rho_{\Weil}^{\psi^{\tau},B}(s)$
and
$\rho_0^\tau(s) := \rho_{\Weil}^{\psi^{\tau},B^0}(s)$;
these operators act respectively
on $\mathcal{S}(B_S)$
and $\mathcal{S}(B_S^0)$.
The operators $1 \otimes \rho_0^{\tau_i}(s)$
on $\mathcal{S}(B_S)$
are defined using the
decomposition $B_S = F_S \oplus B_S^0$,
as in \S\ref{sec:factorization-weil-repn}.
\begin{theorem}\label{thm:main-estimate-general-variance}
  There is a finite subset $X$ of $F^\times$ and a finite
  collection $(\eps_{\tau_1,\tau_2})_{\tau_1,\tau_2 \in X}$ of
  sesquilinear forms
  $\eps_{\tau_1,\tau_2} : \mathcal{S}(B_S) \otimes
  \mathcal{S}(B_S) \otimes \mathcal{A}_0^S \otimes
  \mathcal{A}_0^S \rightarrow \mathbb{C}$,
  depending only upon  $F$, $\psi$, $S$ and $W_S$,
  with the following properties:
  \begin{enumerate}
  \item {\bf Relevance.}
    For $f \in C_c^\infty(\PB^\times_S)$,
    one has the following identity of sesquilinear
    forms on $\mathcal{A}_0^S$:
    \begin{equation}\label{eq:relevance}
      \mathcal{V}_f
      = \mathcal{M}_f
      + \sum_{\tau_1,\tau_2 \in X} \eps_{\tau_1,\tau_2}(\heartsuit^{\tau_1} f, \heartsuit^{\tau_2} f, \cdot, \cdot).
    \end{equation}
  \item {\bf $\O_1(F)$-invariance.}
    \[
    \eps_{\tau_1,\tau_2}(\mathfrak{S} \phi_1, \phi_2, \Psi_1,
    \Psi_2)
    =
    \eps_{\tau_1,\tau_2}( \phi_1, \phi_2, \Psi_1,
    \Psi_2),
    \]
    \[
    \eps_{\tau_1,\tau_2}(\phi_1, \mathfrak{S}\phi_2, \Psi_1,
    \Psi_2)
    =
    \eps_{\tau_1,\tau_2}(\phi_1, \phi_2, \Psi_1,
    \Psi_2).
    \]
    
  \item {\bf $\SO(B_S^0)$-invariance.}
    For $g_1,g_2 \in \PB^\times_S$,
    \[\eps_{\tau_1,\tau_2}(\Ad(g_1) \phi_1, \Ad(g_2) \phi_2, \rho_{\reg}(g_1) \Psi_1,
    \rho_{\reg}(g_2) \Psi_2)
    \]
    \[
    =  \eps_{\tau_1,\tau_2}(\phi_1, \phi_2, \Psi_1,
    \Psi_2).
    \]
  \item {\bf Metaplectic invariance.}
    For $s \in \Mp_2(F_S)$,
    \[
    \eps_{\tau_1,\tau_2}(\rho^{\tau_1}(s) \phi_1, \rho^{\tau_2}(s) \phi_2, \Psi_1,
    \Psi_2)
    =
    \eps_{\tau_1,\tau_2}(\phi_1, \phi_2, \Psi_1,
    \Psi_2).
    \]
    
  \item {\bf Main estimate.}
    For $s \in \Mp_2(F_S)$,
    \[
    \eps_{\tau_1,\tau_2}((1 \otimes \rho_0^{\tau_1}(s)) \phi_1,
    (1 \otimes \rho_0^{\tau_2}(s)) \phi_2, \Psi_1,
    \Psi_2) \]
    \[
    \ll
    \Xi(s)
    \prod_{i=1,2}
    \Sob(\phi_i) \|\Psi\|_{L^1},
    \]
    where $\Xi$ denotes the Harish--Chandra function
    (\S\ref{sec:Xi-global})
    and $\Sob(\phi_i)$ denotes
    a quantity that varies
    continuously with $\phi_i$
    (see \S\ref{sec:some-asympt-notat}).
    The implied constants
    and the uniformity in the continuity of $\Sob(.)$
    depend at most upon $F,\psi,S,W_S$.
  \item  {\bf Construction.}
    $\eps_{\tau_1,\tau_2}$
    factors explicitly through the theta correspondence
    in the sense of \S\ref{sec-4-5-5}
    and the remark of \S\ref{sec-8-6}.
  \end{enumerate}
\end{theorem}

\begin{remark}
 For the application to Theorem
    \ref{thm:main-result-for-microlocal-stuff},
    the crucial assertions are the relevance,
    the
    $\SO(B_S^0)$-invariance, and the main estimate.
    The $\O_1(F)$-invariance and metaplectic
    invariance are employed
    to simplify the presentation of the proof.
    The construction is not applied
    in this paper,
    but may  be useful
    for further refinements and extensions.
  \end{remark}
  \begin{remark}
    One purpose of Part II is to
    give evidence that
    Theorem \ref{thm:main-estimate-general-variance} is
    useful.
  \end{remark}
  \begin{remark}
 The formulation of
    Theorem
    \ref{thm:main-estimate-general-variance} is independent
    of the choice of measures on $\PB_S^\times$
    and on $[\PB^\times]$.
  \end{remark}
  \begin{remark}
   Theorem \ref{thm:main-estimate-general-variance} minus the
    ``main estimate''
    is
    like a trace formula: $\mathcal{V}_f$ is a sum over
    automorphic forms, $\mathcal{M}_f$ is like the ``identity''
    contribution, and the $\eps_{\tau_1,\tau_2}$ are the
    ``interesting'' contributions which one would like in
    practice to show have negligible size.  One difference is
    that $\mathcal{V}_f$ has a quadrilinear (rather than
    bilinear) dependence upon the automorphic forms $\varphi$.
  \end{remark}
  \begin{remark}
 Theorem \ref{thm:main-estimate-general-variance} likely
    extends to the split case $B = M_2(F)$ after incorporating
    contributions from the continuous spectrum into the
    definitions of
    \S\ref{sec:main-general-estimate-key-defns}
    and replacing $\|\Psi_i\|_{L^1}$
    with $\|\htt^A \Psi_i\|_{L^1}$ for some fixed large enough $A > 0$.
  \end{remark}


\subsection{Proof of Theorem \ref{thm:main-estimate-general-variance}}\label{sec:proof-theor-refthm:m}

\subsubsection{Measures}
\label{sec-4-5-1}
With a view to applications, we have formulated Theorem
\ref{thm:main-estimate-general-variance}
in a
measure-independent fashion.
For the proof, it is convenient
to take on $[\PB^\times]$ the Tamagawa measure,
so that
\begin{equation}\label{eq:defn-of-c-2}
  c_3 = \frac{1}{2}  \zeta_F^{(S)}(2),
\end{equation}
and to fix measures
on $\PB^\times_{\mathbb{A}}, \PB^\times_\mathfrak{p}$
and hence on $\PB^\times_S = \prod_{\mathfrak{p} \in S}
\PB^\times_\mathfrak{p}$
as in \S\ref{sec:global-measures}.

\subsubsection{The $\heartsuit$ operator: local}
\label{sec-2-5}
Suppose 
temporarily (for \S\ref{sec-2-5} only)
that $k$ is a local field, $\psi$ is a nontrivial unitary character
of $k$,
$B$ is a quaternion algebra over $k$,
$G := B^\times / k^\times$,
and $W \in C_c^\infty(k^\times)$.
Recall from \S\ref{sec:defn-local-omega} the definition of
$\Omega$.
We define a linear map
$\heartsuit : C_c^\infty(G) \rightarrow \Omega$
by
\begin{equation}
  \heartsuit f(t,x)
  :=
  \frac{W(t \nr(x))}{|t \nr(x)|}
  1_{B^\times}(x)
  f(x).
\end{equation}
(By abuse of notation,
we write $f(x)$ for the value
taken by $f$ at
the image of $x$ under the natural projection
$B^\times \rightarrow G$.)

By inspecting 
the definitions, one has the identities  of maps
$C_c^\infty(G) \rightarrow \Omega$
\[
\mathfrak{S} \heartsuit = \heartsuit \mathfrak{S},
\quad
\Ad(g) \heartsuit = \heartsuit \Ad(g)
\text{ (for $g \in G$)}.
\]
By inspecting the definitions,
one has for $y \in k^\times, b \in B^\times$ that
\begin{equation}\label{eqn:local-heartsuit-formula}
  \rho_{\Weil}(a(y)) \heartsuit f(\nr(b)^{-1},b)
  =
  |y| \heartsuit f(y \nr(b)^{-1}, b)
  =
  \mathfrak{W}(y) f(b).
\end{equation}
By the formula
\eqref{eqn:inner-product-on-Omega-2} for $\|.\|_{\Omega}$,
one obtains
\begin{equation}
  \label{eqn:local-norm-heartsuit-f}
  \|\heartsuit f\|_{\Omega} =
  \|f\|_{L^2(G)} \,
  \|\mathfrak{W}\|_{L^2(k^\times, |x|^{-1} \, d x)}.
\end{equation}

\subsubsection{The $\heartsuit$ operator: global\label{sec:heartsuit-global}}
\label{sec-3-6}
We revert to the global setting
of \S\ref{sec:estimates-general-var}.
Let
$\pi \in A_0^S$.
Recall from \S\ref{sec:similitudes-global} the definition of
$\Omega$.
  We define a linear map
  $\heartsuit : C_c^\infty(\PB^\times_S) \rightarrow \Omega$ by   
  \[
  \heartsuit f(t,x) :=
  \frac{W_S(t_S \nr(x_S))}{|t_S \nr(x_S)|}
  1_{B_S^\times}(x_S)
  f(x_S)
  \prod_{\mathfrak{p} \notin S} \vol(K_{\mathfrak{p}})^{-1}
  \phi_\mathfrak{p}^0(t_\mathfrak{p},x_\mathfrak{p}),
  \]
  where $\phi_\mathfrak{p}^0 \in \Omega_\mathfrak{p}$
  is defined with respect to $R_\mathfrak{p}$
  (see \S\ref{sec:dist-elem}).
This definition
and that of \S\ref{sec:heartsuit}
are obviously similar; we record their
precise relationship below in \S\ref{sec-4-5-5}.

By
\eqref{eqn:local-heartsuit-formula}
and
  the lemma of
  \S\ref{sec:hecke-kernels-local},
  one has  for $y \in \mathbb{A}^\times$,
  $b \in B_\mathbb{A}^\times$ that
  \begin{equation}\label{eqn:formulas-for-heartsuit-f-global}
    |y| \heartsuit f(
    y \nr(b)^{-1}, b)
    =
    W_S(y_S) f(b_S)
    \prod_{\mathfrak{p} \notin S}
    T_{y_\mathfrak{p}}(b_\mathfrak{p}),
  \end{equation}
  with $T_{y_\mathfrak{p}}$ as in \S\ref{sec:hecke-kernels-local}.
By combining
\eqref{eqn:local-norm-heartsuit-f} 
with Lemma \ref{lem:norm-of-distinguished-vector-in-Omega-local}
of \S\ref{sec:dist-elem},
one obtains
\begin{equation}\label{eqn:norm-of-heartsuit-f-global}
  \|\heartsuit f\|^2_{\Omega}
  =
  \|f\|_{L^2(\PB^\times_S)}^2
  \int_{t \in F_S^\times}
  |\mathfrak{W}|^2(t)
  \, \frac{d t}{|t|}
  \prod_{\mathfrak{p} \notin S}
  \frac{\vol(R_\mathfrak{p})}{\vol(K_\mathfrak{p})^2}.
\end{equation}

\begin{lemma*}\label{lem:main-term-eval-for-heartsuit}
  Let $\pi \in A_0^S$.
Let $\Psi_1,\Psi_2 \in \pi$
be $\prod_{\mathfrak{p} \notin S} K_\mathfrak{p}$-invariant
vectors.
  For $f \in C_c^\infty(\PB^\times_S)$,
  the quantity $\mathfrak{m}(\heartsuit f, \heartsuit f, \Psi_1,
  \Psi_2)$ 
  (see \S\ref{sec:induction-omega}) equals
  \[
  c_2
  L^{(S)}(\pi,\tfrac{1}{2})
  \int_{g \in \PB_S^\times}
  \langle \Ad(g) \mathfrak{S} f, \mathfrak{S} f  \rangle_{L^2(\PB_S^\times)}
  \langle \pi(g) \Psi_1, \Psi_2 \rangle_{\PB^\times},
\]
where
\begin{equation}\label{eq:defn-c2-frak-W}
  c_2
  := 
  \frac{1}{\zeta^{(S)}_F(2)}
  (\prod_{\mathfrak{p} \notin S}
  \frac{\vol(R_\mathfrak{p})}{\vol(K_\mathfrak{p})})
  \int_{y \in F_S^\times}
  |W_S|^2(y)
  \, \frac{d y}{|y|}.
\end{equation}
\end{lemma*}
\begin{proof}
  By the lemma of \S\ref{sec:induction-omega},
  the polarization of \eqref{eqn:norm-of-heartsuit-f-global}
  and the commutativity $\heartsuit \Ad(g) = \Ad(g) \heartsuit$,
  the
  required identity holds if we replace $S$ with some possibly
  larger finite set of places $S' \supseteq S$.  To deduce the
  identity as written, we apply
  \eqref{eqn:local-rallis-ipf-unram-calc}
  (using
  \eqref{eqn:generic-at-each-place-if-not-1-diml} to verify its hypotheses).
\end{proof}

\subsubsection{A specific Eichler/Jacquet--Langlands lift}\label{sec:Phi-pi}
For $\pi \in A_0^S$,
let $\Phi_\pi \in \pi_{\JL}$ denote the element
of the Jacquet--Langlands lift of $\pi$
having the Fourier expansion
$\Phi_\pi(n(x) a(y)) =
\sum_{\tau \in F^\times} \psi(\tau x) W_\pi(\tau y)$, where
the Whittaker function
$W_\pi : \mathbb{A}^\times \rightarrow \mathbb{C}$ is given by
$W_\pi(y) := W_S(y_S) \prod_{\mathfrak{p} \notin S}
W_{\pi_\mathfrak{p}}^0(y_\mathfrak{p})$
(see \S\ref{sec:aut-forms-fourier-exp}, \S\ref{sec-3-3-2}).

\begin{lemma*}
  One has $\|\Phi_\pi\|^2 = c_1 L^{(S)}(\ad \pi,1)$, where
  \begin{equation}\label{eq:c-of-frak-W-defn}
    c_1
    :=
    \frac{2}{ \zeta_F^{(S)}(2)}
    (\prod_{\mathfrak{p} \notin S}
    \Delta_{\psi_{\mathfrak{p}}}^{-1/2})
    \int_{y \in F_S^\times} |W_S|^2(y) \, \frac{d y}{|y|}.
  \end{equation}
  If $\pi, \pi ' \in A_0^S$
  are distinct,
  then
  $\langle \Phi_{\pi}, \Phi_{\pi '} \rangle = 0$.
\end{lemma*}
\begin{proof}
  The conclusion in the case $\pi \neq \pi '$ is the
  multiplicity one theorem for $\PB^\times$ combined with the
  injectivity of $\pi \mapsto \pi_{\JL}$.  The formula \eqref{eq:c-of-frak-W-defn} is a
  consequence of
  the lemma of \S\ref{sec-3-3-2} and the corresponding local
  calculation (\S\ref{sec-2-3}).
\end{proof}


\subsubsection{The normalizing scalar}
Recall
from \eqref{eq:c-of-frak-W-defn},
\eqref{eq:defn-c2-frak-W}
and \eqref{eq:defn-of-c-2}
the scalars $c_1,c_2,c_3$.
By the local volume formulas of \S\ref{sec:local-vol-formulas},
\begin{equation}\label{eqn:relation-c1-c2-c}
  c_1^{-1} c_2 = c_3.
\end{equation}

\subsubsection{Application of the pretrace formula}
\label{sec-4-5-2}
Recall the theta functions $\Theta(\phi,\Psi)$
attached in \S\ref{sec:non-traditional-theta-lifts} to each
$\phi \in \Omega, \Psi \in \mathcal{A}_0$.
Let $f \in C_c^\infty(\PB^\times_S)$, $\Psi \in \mathcal{A}_0^S$.

\begin{lemma}\label{lem:abs-conv-of-silly-sum}
  $\sum_{\pi \in A_0^S} |\omega_\pi( f,\Psi)|
  \| \Phi_\pi \|_{L^p([\PGL_2])} < \infty$
  for $p=2,\infty$.
\end{lemma}
\begin{proof}
  By \eqref{eq:rapid-convergence-of-pretrace-formula}
  and
  \eqref{eq:polynomial-growth-of-reps},
  it suffices to show that
  $\| \Phi_\pi \|_{L^p([\PGL_2])} \ll C(\pi)^{O(1)}$.  The case
  $p=2$ follows from the lemma of \S\ref{sec:Phi-pi}
  and \eqref{eq:HL}.  The case $p=\infty$ reduces to the case
  $p=2$ by axioms (S2a) and (S3b) of \cite[\S2.4]{michel-2009},
  wherein the quantities $\mathcal{S}_d(\Phi_\pi)$ may be
  estimated using \cite[\S3.2.5]{michel-2009}.  A
  direct proof of this convergence also follows by a
  rearrangement of the arguments given below.
\end{proof}

\begin{lemma}\label{lem:fourier-coefficients-of-nontraditional-theta-lifts}
  $\Theta(\heartsuit f,\Psi)
  = \sum_{\pi \in A_0^S} \omega_\pi( f,\Psi)
  \Phi_\pi$.
\end{lemma}
\begin{proof}
  Set $\Phi_1 := \Theta(\heartsuit f,\Psi)$
  and $\Phi_2 := \sum_{\pi \in A_0^S} \omega_\pi( f,\Psi)
  \Phi_\pi$; we must show that $\Phi_1 = \Phi_2$.
  Since $\Phi_1, \Phi_2$ are cuspidal,
  it will suffice to demonstrate the equality of their Whittaker functions
  $W_1, W_2 : \mathbb{A}^\times \rightarrow \mathbb{C}$ as
  defined in \S\ref{sec:aut-forms-fourier-exp}.
  By the lemma of \S\ref{sec:four-expans} 
  and \eqref{eqn:formulas-for-heartsuit-f-global},
  we have
  \[
  W_1(y)
  = \sum_{\gamma \in \PB^\times}
  \int_{g \in [\PB^\times]}
  \Psi(g)
  W_S(y_S)
  f(g_S^{-1} \gamma g_S)
  \prod_{\mathfrak{p} \notin S}
  T_{y_\mathfrak{p}}(g_\mathfrak{p}^{-1} \gamma g_\mathfrak{p}).
  \]
  The definition of $\Phi_\pi$
  implies (using Lemma \ref{lem:abs-conv-of-silly-sum}
  to justify the interchange of summation with
  the Fourier integral over the compact group $\mathbb{A}/F$) that
  \[
  W_2(y)
  = \sum_{\pi \in A_0^S} \omega_\pi( f,\Psi) W_S(y_S)
  \prod_{\mathfrak{p} \notin S}
  W_{\pi_\mathfrak{p}}^0(y_\mathfrak{p}).
  \]
  Since $\Psi \in \mathcal{A}_0^S$,
  we have $\omega_\pi(f,\Psi) = 0$
  for all $\pi \in A^S$ with $\pi \notin A_0^S$,
  so it suffices to
  establish for all $y \in \mathbb{A}^\times$, $g \in \PB^\times_{\mathbb{A}}$ 
  the pointwise identity
  \[
  \sum_{\gamma \in \PB^\times}
  f(g_S^{-1} \gamma g_S)
  \prod_{\mathfrak{p} \notin S}
  T_{y_\mathfrak{p}}(g_\mathfrak{p}^{-1} \gamma g_\mathfrak{p})
  =
  \sum_{\pi \in A^S}
  (\sum_{\varphi \in \mathcal{B}(\pi \cap \mathcal{A}^S)}
  \overline{\varphi}(g)
  \pi(f) \varphi(g))
  \prod_{\mathfrak{p} \notin S}
  W_{\pi_\mathfrak{p}}^0(y_\mathfrak{p}),
  \]
  which follows
  from the pretrace formula (\S\ref{sec:pretrace-formula})
  and the identity
  $W_{\pi_\mathfrak{p}}^0(y_\mathfrak{p})
  = \lambda_{\pi_\mathfrak{p}}(T_{y_\mathfrak{p}})$
  (see \eqref{eqn:key-local-identity-relating-whittaker-and-hecke}).
\end{proof}

\begin{remark*}
  Lemma
  \ref{lem:fourier-coefficients-of-nontraditional-theta-lifts}
  and its proof are in the spirit of arguments of
  Shimizu \cite[\S4]{MR0333081}, but we were unable to relate
  them precisely (e.g., by deducing one from the other).
\end{remark*}

\subsubsection{Some sesquilinear forms}
  Define
  $\mathcal{V}, \mathcal{M}, \mathcal{E} : \Omega \otimes \Omega \otimes \mathcal{A}_0
  \otimes \mathcal{A}_0 \rightarrow \mathbb{C}$
  by requiring that
  for
  $\phi_1, \phi_2 \in \Omega$
  satisfying $\phi_i[y] = \phi_i'[y] \otimes \phi_i''[y]$
  with
  $\phi_i'[y] \in \mathcal{S}(\mathbb{A})$ and $
  \phi_i''[y] \in \mathcal{S}(B_\mathbb{A}^0)$
  for all $y \in \mathbb{A}^\times$,
  one has
  with the abbreviations
  $\theta_i := \theta_{\psi^{\tau_i}}(\phi_i'[\tau_i y])$ and
  $h_i := \theta_{\psi^{\tau_i}}(\phi_i''[\tau_i y],\Psi_i)$
  that
  \[
  \mathcal{V}(\phi_1,\phi_2,\Psi_1,\Psi_2)
  := c_1^{-1}
  \int_{y \in \mathbb{A}^\times / F^\times \mathbb{A}^{\times 2}} |y|^2
  \frac
  {
    1
  }
  {
    2^2
  }
  \sum_{\tau_1,\tau_2 \in F^\times / F^{\times 2}}
  \langle
  \theta_1 h_1, \theta_2 h_2 \rangle,
  \]
  \[
  \mathcal{M}(\phi_1,\phi_2,\Psi_1,\Psi_2)
  := c_1^{-1}
  \int_{y \in \mathbb{A}^\times / F^\times \mathbb{A}^{\times 2}} |y|^2
  \frac
  {
    1
  }
  {
    2^2
  }
  \sum_{\tau_1,\tau_2 \in F^\times / F^{\times 2}}
  \langle \theta_1, \theta_2 \rangle
  \langle h_1, h_2 \rangle
  \]
  and $\mathcal{E} := \mathcal{V} - \mathcal{M}$,
  or equivalently,
  \begin{equation}
    \mathcal{E}(\phi_1,\phi_2,\cdot ,\cdot )
    := c_1^{-1}
    \int_{y \in \mathbb{A}^\times / F^\times \mathbb{A}^{\times 2}} |y|^2
    \frac
    {
      1
    }
    {
      2^2
    }
    \sum_{\tau_1,\tau_2 \in F^\times / F^{\times 2}}
    \mathcal{E}_{\tau_1,\tau_2}(\phi_1[\tau_1 y], \phi_2[\tau_2
    y],\cdot,\cdot),
  \end{equation}
  where
  $\mathcal{E}_{\tau_1,\tau_2} : \mathcal{S}(B_\mathbb{A}) \otimes
  \mathcal{S}(B_\mathbb{A})
  \otimes \mathcal{A}_0 \otimes \mathcal{A}_0 \rightarrow
  \mathbb{C}$
  is as in
  \S\ref{sec:main-estimate-for-equidistribution-pairs-theta}.
  The definitions makes sense
  for the same reasons as in \S\ref{sec:induction-omega}.
The identity
\begin{equation}\label{eqn:c1-times-error-equals-bla}
  \mathcal{V}(\phi_1,\phi_2,\Psi_1,\Psi_2)
  = c_1^{-1} \langle \Theta(\phi_1, \Psi_1), \Theta(\phi_2, \Psi_2) \rangle.
\end{equation}
follows
from
\S\ref{sec:unfold-inner-product-nontraditional-theta-over-sl2}
when $\phi$ is a pure tensor,
hence in general by linearity.

\subsubsection{The main identities}\label{sec:main-identity}


\begin{proposition}\label{prop:after-extracting-main-term}
  Let $f \in C_c^\infty(\PB^\times_S)$
  and $\Psi_1,\Psi_2 \in \mathcal{A}_0^S$.
  Then
  \begin{equation}\label{eq:V-equals-Vf}
    \mathcal{V}(\heartsuit f,\heartsuit f,\Psi_1,\Psi_2)
    = \mathcal{V}_f(\Psi_1,\Psi_2),
  \end{equation}
  \begin{equation}\label{eq:M-equals-Mf}
    \mathcal{M}(\heartsuit f,\heartsuit f,\Psi_1,\Psi_2)
    = \mathcal{M}_f(\Psi_1,\Psi_2),
  \end{equation}
  \begin{equation}\label{eq:Vf-equals-Mf-plus-E}
    \mathcal{V}_f(\Psi_1,\Psi_2)
    =
    \mathcal{M}_f(\Psi_1,\Psi_2)
    +
    \mathcal{E}(\heartsuit f, \heartsuit f,
    \Psi_1, \Psi_2).
  \end{equation}
\end{proposition}
\begin{proof}
\eqref{eq:V-equals-Vf}:
    By \eqref{eqn:c1-times-error-equals-bla},
    Lemma
    \ref{lem:fourier-coefficients-of-nontraditional-theta-lifts}
    of \S\ref{sec-4-5-2},
    and
    the
    lemma
    of \S\ref{sec:Phi-pi},
    \begin{align*}
      \mathcal{V}(\heartsuit f,\heartsuit f,\Psi_1,\Psi_2)
      &=
        c_1^{-1} \langle   \Theta(\heartsuit f, \Psi_1),   \Theta(\heartsuit f,
        \Psi_2) \rangle
      \\
      &=
        c_1^{-1} \sum_{\pi_1,\pi_2 \in A_0^S}
        \left\langle
        \omega_{\pi_1}(f, \Psi_1),
        \omega_{\pi_2}(f, \Psi_2)
        \right\rangle
      \\
      &=
        \sum_{\pi \in A_0^S}
        L^{(S)}(\ad \pi,1)
        \omega_{\pi}( f, \Psi_1),
        \overline{\omega_{\pi}( f, \Psi_2)}
      \\
      &=
        V_f(\Psi_1,\Psi_2).
    \end{align*}

    \eqref{eq:M-equals-Mf}:
    by the lemma of \S\ref{sec-3-6}
    and
   \eqref{eqn:relation-c1-c2-c}.

   \eqref{eq:Vf-equals-Mf-plus-E}:
    by
    \eqref{eq:V-equals-Vf},
    \eqref{eq:M-equals-Mf}
    and the definition of $\mathcal{E}$.
\end{proof}

\subsubsection{Completion of the proof}
\label{sec-4-5-5}
We now apply Proposition
\ref{prop:main-error-estimate-global-adelic-general} (see
\S\ref{sec-3-5-5}) and Proposition
\ref{prop:after-extracting-main-term} to prove Theorem
\ref{thm:main-estimate-general-variance}.  This final part of
the argument is of a purely technical nature and involves no
major new ideas.  Indeed, its main purpose is to recast the
content of those propositions
in terms of $\mathcal{S}(B_S)$
rather than the less ``user-friendly'' space $\Omega$.

By weak approximation,
we may choose a compact fundamental domain $Y \subset
\mathbb{A}^\times / \mathbb{A}^{\times 2}$
for
$\mathbb{A}^\times / F^{\times} \mathbb{A}^{\times 2}$ with the
property that $y_\mathfrak{p} = 1$ for all $y \in Y$ and
$\mathfrak{p} \in S$.
Choose a finite set $X \subseteq F^\times$ of
representatives for
the finite set
\begin{equation}\label{eq:elements-such-that-multiplying-by-something-in-Y-gives-squares-mod-units-outside-S}
  \{\tau \in F^\times / F^{\times 2} :
  \text{
    there exists }
  y \in Y
  \text{ so that }
  y_\mathfrak{p} \tau \in F_\mathfrak{p}^{\times 2}  \mathcal{O}_\mathfrak{p}^\times
  \text{ for all }
  \mathfrak{p} \notin S
  \}.
\end{equation}
For $y \in Y, \tau \in X$, let
$\diamondsuit^{\tau y} : \mathcal{S}(B_S) \hookrightarrow
\mathcal{S}(B_\mathbb{A})$
denote the map
$\diamondsuit^{y \tau} \Phi := \Phi \otimes
(\otimes_{\mathfrak{p} \notin S} \phi_\mathfrak{p}^0[\tau
y_\mathfrak{p} ])$,
where $\phi_{\mathfrak{p}}^0 \in \Omega_{\mathfrak{p}}$ denotes as usual
the distinguished element.
For $f \in C_c^\infty(\PB_S^\times)$, one then has
$\diamondsuit^{\tau y} \heartsuit^{\tau} f = \heartsuit f[\tau y]$.
Observe that
for $\mathfrak{p} \notin S$
and $t \in F_\mathfrak{p}^\times$,
one has
$\phi_\mathfrak{p}^0[t] = 0$ unless
$t \in F_\mathfrak{p}^{\times 2}
\mathcal{O}_\mathfrak{p}^\times$.
It follows
that for $\tau \in F^\times$ and  $y \in Y$,
one has $\heartsuit f[\tau y] = 0$
unless $\tau$
belongs to the set
\eqref{eq:elements-such-that-multiplying-by-something-in-Y-gives-squares-mod-units-outside-S},
hence that
\begin{equation}\label{eq:penultimate-identity-before-proving-main-general-thm}
  \mathcal{E}(\heartsuit f, \heartsuit f, \cdot, \cdot)
  =
  \frac{c_1^{-1}}{2^2}
  \int_{y \in Y}
  |y|^2
  \sum_{\tau_1,\tau_2 \in X}
  \mathcal{E}_{\tau_1,\tau_2}(\diamondsuit^{\tau_1 y}
  \heartsuit^{\tau_1} f,
  \diamondsuit^{\tau_2 y} \heartsuit^{\tau_2} f,
  \cdot,\cdot).
\end{equation}
Define
$\eps_{\tau_1,\tau_2} : \mathcal{S}(B_S) \otimes \mathcal{S}(B_S)
\otimes \mathcal{A}_0^S
\otimes \mathcal{A}_0^S \rightarrow \mathbb{C}$
by
\begin{equation}\label{eq:defn-of-the-eps-guys-yay}
  \eps_{\tau_1,\tau_2}(\Phi_1,\Phi_2,\Psi_1,\Psi_2)
  :=
  \frac{c_1^{-1}}{2^2}
  \int_{y \in Y}
  |y|^2
  \mathcal{E}_{\tau_1,\tau_2}(\diamondsuit^{\tau_1 y}
  \Phi_1,
  \diamondsuit^{\tau_2 y} \Phi_2,
  \Psi_1,\Psi_2).
\end{equation}
We verify the assertions
made in Theorem
\ref{thm:main-estimate-general-variance}:
\begin{enumerate}
\item The ``relevance''
  follows from
  \eqref{eq:penultimate-identity-before-proving-main-general-thm},
  \eqref{eq:defn-of-the-eps-guys-yay}
  and Proposition \ref{prop:after-extracting-main-term}.
\item Since
  $\mathfrak{S} \phi_\mathfrak{p}^0 = \phi_\mathfrak{p}^0$, one
  has
  $\mathfrak{S} \diamondsuit^{\tau_i y} = \diamondsuit^{\tau_i
    y} \mathfrak{S}$.
  For $g \in \PB^\times_S$ and $s \in \Mp_2(F_S)$, one has
  $\Ad(g) \diamondsuit^{\tau_i y} = \diamondsuit^{\tau_i y}
  \Ad(g)$
  and
  $\rho^{\tau_i}(s) \diamondsuit^{\tau_i y} =
  \diamondsuit^{\tau_i y} \rho^{\tau_i}(s)$.
  Thus the ``$\O_1(F)$-invariance,''
  ``$\SO(B_S^0)$-invariance'' and ``metaplectic
  invariance'' follow from \S\ref{sec:equivariance-summary-for-E-tau-tau}.
\item The ``main estimate''
  is the content of Proposition
  \ref{prop:main-error-estimate-global-adelic-general}.
\end{enumerate}

\subsection{Classicalization}
\label{sec-4-4}

We begin discussing
how to relate
the setting of
Theorem
\ref{thm:main-estimate-general-variance}
to that of
Theorem \ref{sec-1}.
We complete this discussion
in \S\ref{sec:deduction-main-thm-microlocal}.



\subsubsection{Specialization to a single place\label{sec:general-estimates-specialized-single-place}}
\label{sec-4-4-1}
We specialize the definitions of
\S\ref{sec:main-general-estimate-key-defns} to the case that
ramification is concentrated at a single place $\mathfrak{q}$ of
$F$, finite or infinite.  This is the case required for the
proof of Theorem \ref{thm:main-result-for-microlocal-stuff}.

Assume that $S$ is the set of places $\mathfrak{p}$ for which
either
\begin{itemize}
\item $\mathfrak{p}$ is
  infinite,
\item  $\mathfrak{p}$ is a finite place
  at which $B$ ramifies, or
\item $\mathfrak{p} = \mathfrak{q}$. 
\end{itemize}
Assume that
for each $\mathfrak{p} \notin S - \{\mathfrak{q}\}$,
the completion
$B_\mathfrak{q}$ is non-split,
or equivalently,
that $\PB^\times_\mathfrak{q}$ is compact.
There are the following possibilities:
\begin{enumerate}
\item $\mathfrak{q}$ is real, in which case
  $F$ is totally real and $B$ ramifies at every infinite place other than $\mathfrak{q}$.
\item $\mathfrak{q}$ is complex, in which case $F$ is real and $B$ ramifies at every infinite place other than $\mathfrak{q}$.
\item $\mathfrak{q}$ is finite, in which case $F$ is totally
  real and $B$ is totally definite.
\end{enumerate}
For each place $\mathfrak{p}$, define the
compact open subgroup
$J_\mathfrak{p} \leq \PB^\times_\mathfrak{p}$ as in \S\ref{sec-2-4}  by
taking for $J_\mathfrak{p}$ the image of $R_\mathfrak{p}^\times$
if $\mathfrak{p}$ is finite and taking
$J_\mathfrak{p} := \PB^\times_\mathfrak{p}$ if $\mathfrak{p}$ is
infinite.  Set
$J := \prod_{\mathfrak{p} \neq \mathfrak{q}} J_\mathfrak{p}$.
In addition to the notation of \S\ref{sec-4-1}, we now introduce a
superscripted $J$, as in $\mathcal{A}^J, \mathcal{A}_0^J, \pi^J$
to denote the $J$-fixed subspace.
We denote by $\mathcal{A}^J_+ \subseteq \mathcal{A}^J$,
$\mathcal{A}^J_{0+} \subseteq \mathcal{A}_0^J$ the ``even''
subspaces consisting of $\varphi$ that are
$\PB^\times_{\mathfrak{p}}$-invariant for all
$\mathfrak{p} \in S - \{\mathfrak{q}\}$.
Thus, for instance,
$\mathcal{A}_{0+}^J
\subseteq \mathcal{A}_0^J
\subseteq \mathcal{A}_0^S
\subseteq \mathcal{A}_0 \subseteq \mathcal{A}$.
We denote by $A_0, A^J, A_0^J, A^J_+, A^J_{0+}$
the set of all $\pi \in A$ having nonzero
intersection with the space
having the corresponding scripted notation.

Set $G := \PB^\times_\mathfrak{q}$, and let
$f \in C_c^\infty(G)$.   For
$\mathfrak{p} \in S - \{q\}$, set
$e_{J_\mathfrak{p}} :=
\vol(J_\mathfrak{p})^{-1} 1_{J_\mathfrak{p}} \in
C_c^\infty(\PB^\times_\mathfrak{p})$.  Define
$\tilde{f} \in
C_c^\infty(\PB^\times_S)$
by the formula
\[
\tilde{f} (g)
:=
f(g_\mathfrak{q})
\prod_{\mathfrak{p} \in S - \{\mathfrak{q} \}}
e_{J_\mathfrak{p}}(g_\mathfrak{p}).
\]
For $\pi \in A_0^J$
and $\Psi \in \mathcal{A}_0^J$,
set
$\omega_\pi(f,\Psi) :=
\sum_{\varphi \in \mathcal{B}(\pi \cap \mathcal{A}^J)}
\langle \varphi, \Psi \cdot \pi(f) \varphi  \rangle$.
Then
$\omega_\pi(\tilde{f},\Psi) = \omega_\pi(f,\Psi)$
(see \S\ref{sec:omega-pi}).
Since $\dim(\pi_\mathfrak{p}) = 1$ for all $\mathfrak{p} \in S -
\{\mathfrak{q} \}$,
one has
for $\Psi \in \pi' \in \mathcal{A}_0^J$
that
$\omega_\pi(f,\Psi) = 0$
unless $\pi ' \in \mathcal{A}_{0+}^J$.

Let
$V_{f}, M_{f} : \mathcal{A}_{0+}^J
\otimes \mathcal{A}_{0+}^J \rightarrow \mathbb{C}$ denote the
sesquilinear forms obtained by restricting the forms
$\mathcal{V}_{\tilde{f}}, \mathcal{M}_{\tilde{f}} : \mathcal{A}_0^S \otimes
\mathcal{A}_0^S \rightarrow \mathbb{C}$.
Then
\begin{equation}
  V_{f}(\Psi_1,\Psi_2)
  =
  \sum_{\pi \in A_0^J}
  L^{(S)}(\ad \pi,1)
  \omega_\pi(f,\Psi_1)
  \overline{\omega_\pi(f,\Psi_2)}.
\end{equation}
By the observation that
$\mathfrak{S} e_{J_\mathfrak{p} } = e_{J_\mathfrak{p}}$ for
$\mathfrak{p} \in S - \{\mathfrak{q}\}$ and the local
calculations
\eqref{eqn:local-rallis-ipf-integral-nonsplit-finite}
and \eqref{eqn:local-rallis-ipf-integral-nonsplit-real}, we see for
$\Psi_1 \in \pi_1 \in A_{0+}^J$ and
$\Psi_2 \in \pi_2 \in A_{0+}^J$ that $M_{f}(\Psi_1,\Psi_2) = 0$
unless $\pi_1 = \pi_2 =: \pi$, in which case
\begin{equation}
  M_{f}(\Psi_1,\Psi_2)
  =
  c_4
  L^{(S)}(\pi,\tfrac{1}{2})
  \int_{g \in G}
  \langle \Ad(g) \mathfrak{S} f, \mathfrak{S} f \rangle_{L^2(G)}
  \langle \pi(g) \Psi_1, \Psi_2 \rangle_{\PB^\times}
\end{equation}
where
\begin{equation}\label{eq:defn-of-c4}
  c_4 := 2^t
  \zeta_F^{(S)}(2) \vol([\PB^\times])^{-1}.
\end{equation}
with
$t$ 
the number of finite primes $\mathfrak{p} \in S - \{\mathfrak{q}\}$.

\subsubsection{Strong approximation\label{sec:strong-approx}}
\label{sec-4-4-2}
Retaining the notation of \S\ref{sec:general-estimates-specialized-single-place},
we record here how the quotient
$[\PB^\times]/J$
unadelizes
under some assumptions.
Recall that $G := \PB^\times_{\mathfrak{q}}$.
Let $\Gamma \leq G$
denote the image of $\PB^\times \cap J$
under the inclusion $\PB^\times \hookrightarrow G$.
Then $\Gamma$ is a discrete cocompact
subgroup of $G$,
and the natural map
\begin{equation}\label{eq:map-defining-adelic-orbit}
  \Gamma \backslash G
  \xrightarrow{\iota}
  [\PB^\times]/J
\end{equation}
is injective.
\begin{lemma*}\label{lem:consequence-of-strong-approx}
  Suppose that $F$ has odd narrow class number and either that
  \begin{enumerate}
  \item  $B_\mathfrak{q}$ is split,
    or that
  \item  $\mathfrak{q}$ is infinite
    and $B$ has class number one.
  \end{enumerate}
  Then $\iota$ is bijective.
\end{lemma*}
\begin{proof}
  The class number assumption on $F$
  implies that
  \begin{equation}\label{eq:class-number-odd-definite-case}
    F_+^\times
    F_{\infty+}^\times 
    F_{\mathfrak{q}}^\times 
    \prod_{\mathfrak{p} < \infty}
    \mathcal{O}_\mathfrak{p}^\times
    \mathbb{A}^{\times 2}
    =
    \mathbb{A}_+^\times :=
    F_{\infty+}^\times
    \mathbb{A}_f^\times
  \end{equation}
  where $F_{\infty+}^\times$ denotes the connected component of
  the unit group of
  $F_{\infty} := \prod_{\mathfrak{p}|\infty} F_\mathfrak{p}$,
  $\mathbb{A}_f^\times := \prod_{\mathfrak{p} < \infty}
  F_\mathfrak{p}^\times$ denotes the group of finite ideles,
  and $F_+^\times := F^\times \cap F_{\infty+}^\times$.
  Set $J' := \PB^\times_{\mathfrak{p}_0} J$; it is open in $\PB^\times_{\mathbb{A}}$.
  Under the reduced norm map $\nr : \PB^\times_\mathbb{A}
  \rightarrow \mathbb{A}^\times_+/\mathbb{A}^{\times 2}$,
  we have
  $\nr(\PB^\times) = F_+^\times \mathbb{A}^{\times
    2}/\mathbb{A}^{\times 2}$
  and
  $\nr(J') = F_{\infty+}^\times F_{\mathfrak{p}_0}^\times
  \prod_{\mathfrak{p} < \infty}
  \mathcal{O}_{\mathfrak{p}}^\times
  \mathbb{A}^{\times 2} /   \mathbb{A}^{\times 2}$;
  from our assumption \eqref{eq:class-number-odd-definite-case},
  it follows
  that
  \begin{equation}\label{eq:norm-relation-relevant-for-strong-approx}
    \nr(\PB^\times) \nr(J') = \nr(\PB_\mathbb{A}^\times).
  \end{equation}
  If $B_\mathfrak{q}$ splits,
  then $\PB^\times_\mathfrak{q}$ is non-compact,
  so the strong approximation theorem
  and \eqref{eq:norm-relation-relevant-for-strong-approx}
  imply
  that $\PB^\times J' = \PB_{\mathbb{A}}^\times$,
  hence that $\iota$ is surjective.
  In the remaining case, the surjectivity of $\iota$
  holds by the definition of ``the class number of $B$.''
\end{proof}
\begin{remark*}\label{rmk:classical-volume-computation}
  Suppose the conclusion of the lemma holds and that
  $\mathfrak{q}$ is finite.  It is then natural to ask for the
  volume of $\Gamma \backslash G$ with respect to a Haar $\nu$
  on $\Gamma \backslash G$ obtained as the quotient of some given Haar $\mu$ on $G$.
  Using that $\vol([\PB^\times]) = 2$ with respect to Tamagawa
  measure, 
  one can show
  (as in \S\ref{sec:mean-statistics} below)
  that
  \[
  \nu(\Gamma \backslash G) = 2 \frac{
    \zeta_F(2)
    \Delta_B
    \Delta_F^{3/2}
  }{ (4 \pi^2)^{[F:\mathbb{Q}]} \prod_{\mathfrak{p} \in
        \ram_f(B)} \zeta_\mathfrak{p}(1)} \mu(K_{\mathfrak{q}}),\] where $K_{\mathfrak{q}} \leq G$
  denotes a maximal compact subgroup, $\Delta_F$
  the absolute discriminant, $\Delta_B$ the absolute
  reduced discriminant,
  and
  $\ram_f(B)$ the set of finite places
  at which $B$ ramifies.
\end{remark*}
\subsubsection{Hecke operators\label{sec:hecke-ops}}
\label{sec-4-4-3}
For each finite place $\mathfrak{p} \neq \mathfrak{q}$, let
$\mathcal{H}_{\mathfrak{p}}$ denote the Hecke algebra consisting
of compactly-supported bi-$J_{\mathfrak{p}}$-invariant
distributions on $\PB^\times_{\mathfrak{p}}$.  
It acts on $\mathcal{A}^J$.
It admits a
standard generator $T_{\mathfrak{p}}$ given by convolution
against
$|\varpi|^{-1} T_{\varpi} \in C_c^\infty(J_\mathfrak{p}^\times
\backslash \PB_\mathfrak{p}^\times / J_\mathfrak{p}^\times)$,
where $\varpi$ is a generator of the maximal ideal in
$\mathcal{O}_\mathfrak{p}$ and $T_{\varpi}$ is the normalized
Hecke kernel defined in \S\ref{sec:hecke-kernels-local}.
If
$B$ ramifies at $\mathfrak{p}$,
then
$T_{\mathfrak{p}}$ is an involution; otherwise, it has degree
$|\mathfrak{p}| + 1$,
where $|\mathfrak{p}|$ denotes the absolute norm.
\begin{lemma*}
  The following are equivalent for $\varphi \in \mathcal{A}^J$.
  \begin{enumerate}
  \item $\varphi$ generates an irreducible
    representation of $G$
    and is an eigenfunction of $T_\mathfrak{p}$ for each finite
    $\mathfrak{p} \neq \mathfrak{q}$.
  \item $\varphi$ generates an irreducible representation $\pi$
    of $\PB^\times_\mathbb{A}$.
  \end{enumerate}
  Assume that these conditions hold.
  Then $\pi \in A^J$,
  and the following are equivalent:
  \begin{enumerate}
  \item $T_\mathfrak{p} \varphi = \varphi$ for each finite prime
    $\mathfrak{p} \neq \mathfrak{q}$ at which $B$ ramifies.
  \item $\pi \in A^J_+$.
  \end{enumerate}
\end{lemma*}
In view of the lemma, the following definition faithfully
extends and makes precise Definition
\ref{defn:eigenfunctions-intro}
of \S\ref{sec:setting-overview}:
\begin{definition*}\label{defn:eigenfunction-general}
  An \emph{eigenfunction} is a nonzero element
  $\varphi \in \mathcal{A}^J$ that belongs to some
  $\pi \in A^J$; it is \emph{even} if $\pi \in A^J_+$ and
  \emph{strongly of mean zero}
  if $\pi \in A^J_{0}$ (and both if $\pi \in A^J_{0+}$).
\end{definition*}

\newpage
\part{Application to microlocal lifts}
Our aim is now to prove Theorem
\ref{thm:main-result-for-microlocal-stuff} by application of
Theorem \ref{thm:main-estimate-general-variance}.  We retain the
general notation of \S\ref{sec:general-notation}.
The following additional notation
is in effect for
\S\ref{sec:appl-micro-local-prelims}--\S\ref{sec:local-estimates-error}:
\begin{itemize}
\item $k$: a non-archimedean local field
of characteristic $\neq 2$.
\item $\mathfrak{o}, \mathfrak{q}, q, |.| = |.|_k, \ord = \ord_k$:
as in \S\ref{sec-2-1-1}.  
\item $B:= M_2(k)$: the algebra of $2 \times 2$ matrices,
  so that $\nr = \det : B \rightarrow k$.
\item $G := B^\times / k^\times = \PGL_2(k)$.  
\end{itemize}
Equip $G$ with any Haar
measure; we choose it more precisely when we must.

\section{Preliminaries\label{sec:appl-micro-local-prelims}}
\label{sec-5}

\subsection{Conductors}
\label{sec-5-1}
Let $c(\omega)$ denote
the (log-)conductor of a character
$\omega : \mathfrak{o}^\times \rightarrow \mathbb{C}^\times$,
that is, the smallest nonnegative integer $n$ for which $\omega$
has trivial restriction to
$\mathfrak{o}^\times \cap (1 + \mathfrak{q}^n)$.
Note that $c(\omega) = 0$ precisely when $\omega = 1$.
Let $c(\chi) := c(\chi|_{\mathfrak{o}^\times})$
denote the conductor of
a character
$\chi : k^\times \rightarrow \mathbb{C}^\times$.


\subsection{Principle series representations}
\label{sec-5-2}
For a character $\chi : k^\times \rightarrow \mathbb{C}^\times$,
we denote by $\chi \boxplus \chi^{-1}$ the corresponding induced
representation; it consists of smooth functions
$v : G \rightarrow \mathbb{C}$ satisfying
$v(n(x) a(y) g) = |y|^{1/2} \chi(y) v(g)$ for all
$x,y,g \in k,k^\times,G$,
with the group $G$ acting by right translation.
If $c(\chi^2) \neq 0$,
then $\chi \boxplus \chi^{-1}$ is irreducible and generic.

\subsection{Microlocal lifts}
\label{sec-5-3}
We recall the specialization to trivial central characters of
the main definition from \cite{nelson-padic-que}
and record some basic properties.

\subsubsection{Definition}
\label{sec-5-3-1}
For each integer $N > \ord_k(2)$, fix a decomposition
$N - \ord_k(2) = N_1 + N_2$ into nonnegative integers
$N_1, N_2$ that tend to $\infty$ with $N$.
For $N$ large enough,
we assume for the sake of consistency
with
\S\ref{sec:setting-overview}
that
$N_1 = \lfloor N/2 \rfloor - \ord_k(2), N_2 = \lceil N/2 \rceil$.

Let $\mathcal{A}$ be a representation
  of $G$.
  We say that a  vector $\varphi \in \mathcal{A}$
  is a \emph{microlocal lift}
  if
  \begin{enumerate}
  \item[(i)] $\varphi \neq 0$,
  \item[(ii)] $\varphi$ generates an irreducible subrepresentation
    $\pi \subseteq A$, and
  \item[(iii)]
    there is a character $\omega$ of $\mathfrak{o}^\times$
    whose square $\omega^2$ is nontrivial
    so that with  $N := c(\omega)$ and for all
    \begin{equation}\label{eq:microlocal-lift-subgroup-defn}
      g = \begin{pmatrix}
        a & b \\
        c & d
      \end{pmatrix}
      \in
      \GL_2(\mathfrak{o}) \cap \begin{pmatrix}
        \mathfrak{o} & \mathfrak{q}^{N_1} \\
        \mathfrak{q}^{N_2} & \mathfrak{o}
      \end{pmatrix},
    \end{equation}
    one has
    $g \varphi = \omega(a^2/\nr(g)) \varphi$.
  \end{enumerate}
  We call $\omega$ the \emph{orientation} of $\varphi$; it is
  determined by $\varphi$.
  
\begin{remark}
 We refer again to \cite[Thm 25, Rmk 26]{nelson-padic-que}
  for a discussion of the sense in which the microlocal lifts
  considered here are actually ``lifts.''
\end{remark}
\begin{remark}
The quantity denoted $N$ in \cite{nelson-padic-que} has been
  renamed here to $N - \ord_k(2)$ in order to make the formulas
  for the variance statistics hold more uniformly when
  $k$ extends $\mathbb{Q}_2$.  Amusingly,
  the convention chosen here is less natural from the
  perspective of linear statistics (see
  \S\ref{sec:mean-statistics}).
\end{remark}
\subsubsection{Classification}
\label{sec-5-3-2}
Let $\pi$ be an irreducible representation of $G$.
By \cite[Lem 22]{nelson-padic-que}, we have:
\begin{lemma*}\label{lem:determination-microlocal-lifts}~
  \begin{enumerate}
  \item Suppose that $\pi$ is isomorphic
    to $\chi \boxplus \chi^{-1}$ for some character $\chi$ of $k^\times$
    for which $c(\chi^2) \neq 0$.  Then the set of
    microlocal lifts in $\pi$ is a disjoint union
    $\mathbb{C}^\times \varphi_+ \bigsqcup \mathbb{C}^\times \varphi_-$,
    where $\varphi_{+}$ and $\varphi_-$
    are  microlocal lifts
    of orientations
    $\omega := \chi|_{\mathfrak{o}^\times}$
    and $\omega^{-1}$,
    respectively.
  \item Otherwise, $\pi$ contains no microlocal lifts.
  \end{enumerate}
\end{lemma*}

\begin{corollary*}\label{cor:get-nice-subspaces-of-microlocal-lifts-assuming-no-inverses}
  Let $X$
  be a set consisting of characters $\omega$ of
  $\mathfrak{o}^\times$
  for which $\omega^2 \neq 1$.
  Assume that $\omega \in X \implies \omega^{-1}
  \notin X$.
  The set of microlocal lifts in $\pi$
  with orientation in $X$
  is then either empty
  or of the form
  $\mathbb{C}^\times \varphi$
  for some $0 \neq \varphi \in \pi$.
\end{corollary*}

\subsubsection{Projectors}
\label{sec-5-3-3}
Let $\omega$ be a character of $\mathfrak{o}^\times$
for which $\omega^2$ is nontrivial.
Set $N := c(\omega)$.
Let $\widetilde{\mathfrak{J}}$
denote the subgroup of $\GL_2(\mathfrak{o})$ appearing
on the RHS of \eqref{eq:microlocal-lift-subgroup-defn}.
Let $\mathfrak{J} \leq G$
denote the image of  $\widetilde{\mathfrak{J}}$ under the projection $\GL_2(k)
\rightarrow G$.
By
matrix multiplication
and the inequalities $\min(N_1,N_2) \geq 0$,
$\max(N_1,N_2) \geq 1$,
one has for $x_1,x_2 \in k$
and $y \in k^\times$
that
\begin{equation}\label{eqn:explicit-description-of-J-bar}
  n'(x_2) n(x_1) a(y) \in
  \mathfrak{J}
  \iff
  x_1 \in \mathfrak{q}^{N_1},
  x_2 \in \mathfrak{q}^{N_2},
  y \in \mathfrak{o}^\times
\end{equation}
Define $f_\omega \in C_c^\infty(G)$ as follows, for $g \in G$:
\begin{itemize}
\item If $g \notin \mathfrak{J}$, then $f_\omega(g) := 0$.
\item If $g \in \mathfrak{J}$ is image of 
  $\tilde{g} = \begin{pmatrix}
    a & b \\
    c & d
  \end{pmatrix}
  \in  \widetilde{\mathfrak{J}}$, then
  $f_\omega(g) := \vol(\mathfrak{J})^{-1}
  \omega(\nr(\tilde{g})/a^2)$.
\end{itemize}
The definition is independent of the
  choice of $\tilde{g}$.
\begin{lemma*}\label{lem:harmonic-analytic-isolation-0}
  Let $\pi$ be an irreducible unitary representation of $G$.
  \begin{enumerate}
  \item If $\pi$ contains a microlocal lift $\varphi$ of
    orientation $\omega$, then $\pi(f_\omega)$ is the rank one
    orthogonal projector onto the subspace $\mathbb{C} \varphi$
    of $\pi$.
  \item Otherwise, $\pi(f_\omega) = 0$.
  \end{enumerate}
\end{lemma*}
\begin{proof}
  $f_\omega$ is supported on
  $\mathfrak{J}$.
  By direct calculation, the restriction of
  $\vol(\mathfrak{J}) f_\omega$ to $\mathfrak{J}$ defines a
  unitary character $\xi$ of $\mathfrak{J}$
  with the property that 
  condition (iii) in \S\ref{sec-5-3-1}
  reads ``$g \varphi =  \xi^{-1}(g) \varphi$ for all $g \in \mathfrak{J}$.''
  The conclusion follows
  from elementary Fourier analysis on the compact group
  $\mathfrak{J}$.
\end{proof}

\subsection{The convolution kernel of interest}

\subsubsection{Taylor expansion of the logarithm}
\label{sec:tayl-expans-logar}
Let $N,N_0$ be positive integers.
Assume that $N_0$ is large enough in terms of $\ord_k(2)$ and
that $N$  is large enough in terms of
$N_0$ that
\begin{equation}\label{eq:basic-N-N0-ord-2-assumptions}
  N > N_0  > \ord_k(2),
\end{equation}
\begin{equation}\label{eq:basic-N-N0-ord-2-assumptions-2}
  N \geq 2 N_0 + \ord_k(2).
\end{equation}
In applications, we take $N_0$ large enough but fixed, and let
$N \rightarrow \infty$.

Set $\mathfrak{o}^\times_0 := 1 + \mathfrak{q}^{N-N_0}$,
regarded as a subgroup of $\mathfrak{o}^\times$.
Then
\begin{equation}\label{eqn:index-of-O-0}
  [\mathfrak{o}^\times:\mathfrak{o}_0^\times]
  = q^{N-N_0} \zeta_k(1)^{-1}.
\end{equation}
Regard $1 + \mathfrak{q}^N$ as a subgroup of
$\mathfrak{o}^\times_0$.
Fix a uniformizer $\varpi$,
i.e., a generator of the $\mathfrak{o}$-ideal $\mathfrak{q}$.
\begin{lemma*}\label{lem:additive-vs-mult-over-local-field-local-taylor-sorta}
  The map
  \[
  \iota : 
  \mathfrak{o}^\times_0 / (1 + \mathfrak{q}^{N})
  \rightarrow 
  \mathfrak{q}^{-N_0} / \mathfrak{o}
  \]
  \[
  \iota(u) := \varpi^{-N} (u-1)
  \]
  is an isomorphism of groups.
\end{lemma*}
\begin{proof}
  By inspection, $\iota$ is well-defined
  and has a well-defined inverse $\iota^{-1}$,
  hence is bijective.
  For $x,y \in \mathfrak{q}^{-N_0}/\mathfrak{o}$,
  one has
  $\iota^{-1}(x) \iota^{-1}(y) = \iota^{-1}(x+y) + \eps$
  with
  $\eps := \varpi^{2 N} x y$.
  By \eqref{eq:basic-N-N0-ord-2-assumptions-2},
  one has
  $\eps \in \mathfrak{q}^{2 N - 2 N_0} \subseteq
  \mathfrak{q}^N$,
  so $\iota^{-1}$ is a homomorphism.
\end{proof}

\subsubsection{Partitioning the characters of given conductor}
\label{sec-5-4}
Motivated by the corollary of \S\ref{sec-5-3-2},
we record here
a partition of the characters of $\mathfrak{o}^\times$
having given conductor
into nice subsets that do not
contain the inverses of any characters that they contain.
Recall that a unitary character $\sigma : k \rightarrow
\mathbb{C}^{(1)}$
is \emph{unramified}
if it is trivial on $\mathfrak{o}$ but not on $\mathfrak{q}^{-1}$.
Let $\Sigma$
denote the set of equivalence classes of unramified
unitary characters $\sigma$ of $k$, with two such characters
declared equivalent if they have the same restriction to
$\mathfrak{q}^{-N_0}$.  For any $\sigma \in \Sigma$, the map
\[
\mathfrak{o}^\times / (1 + \mathfrak{q}^{N_0}) \rightarrow
\Sigma\]
\[
\xi \mapsto \text{ (the class of the character
  $[x \mapsto \sigma(\xi x)]$)}
\]
is bijective,
so $\Sigma$ is a
finite set of cardinality
\begin{equation}\label{eq:Sigma-cardinality}
  |\Sigma| = \zeta_k(1)^{-1} q^{N_0}.
\end{equation}

For each
$\sigma \in \Sigma$, let
$\omega_\sigma : \mathfrak{o}^\times_0 \rightarrow
\mathbb{C}^\times$ denote the function defined by
$\omega_\sigma(u) := \sigma (\iota(u))$.
Both $\omega_\sigma$ and the association $\Sigma \ni \sigma \mapsto \omega_{\sigma}$ are
then well-defined. By
the lemma of \S\ref{sec:tayl-expans-logar}, we see that
\begin{itemize}
\item $\omega_\sigma$
  is a character of $\mathfrak{o}_0^\times$
  of conductor $N$ (in the sense that it has trivial restriction
  to $1 + \mathfrak{q}^{N}$ but not to $1 + \mathfrak{q}^{N-1}$),
  and that
\item each character
  of $\mathfrak{o}^\times_0$
  of conductor $N$ 
  is of the form $\omega_\sigma$
  for some unique $\sigma \in \Sigma$.
\end{itemize}

Let $\mathcal{X}_N$ denote the set of characters $\omega$ of
$\mathfrak{o}^\times$ for which $c(\omega) = N$.
Since each such
$\omega$
restricts to a character of $\mathfrak{o}_0^\times$,
we have a partition
\begin{equation}\label{eq:partition-of-chars-into-packets}
  \mathcal{X}_N = \bigsqcup_{\sigma \in \Sigma}
  \mathcal{X}_N^\sigma
\end{equation}
where $\mathcal{X}_N^\sigma := \{\omega \in \mathcal{X}_N : \omega|_{\mathfrak{o}^\times_0} = \omega_{\sigma}\}$.
\begin{lemma*}\label{lem:properties-of-partition-of-characters}
  $\omega \in \mathcal{X}_N^\sigma \implies \omega^{-1} \notin
    \mathcal{X}_N^\sigma$.
\end{lemma*}
\begin{proof}
  If $\omega, \omega^{-1} \in \mathcal{X}_N^\sigma$ then
  $\omega_\sigma = \omega|_{\mathfrak{o}_0^\times} =
  \omega_\sigma^{-1}$, hence $\omega_\sigma^2 = 1$, hence
  $\sigma(2 x) = 1$ for all $x \in \mathfrak{q}^{-N_0}$,
  contrary to
  our assumptions that $\sigma$ is unramified and
  $N_0 > \ord_k(2)$.
\end{proof}

\subsubsection{Definition}\label{sec:element-attached-to-N-sigma}
By \emph{the element $f \in C_c^\infty(G)$
  attached to $(N,\sigma)$}
we shall mean
the function
\[f := \sum_{\omega \in \mathcal{X}_N^{\sigma}}
f_\omega,\]
where $f_\omega$ is as in \S\ref{sec-5-3-3}.
The element $f$ depends also upon $N_0$;
we regard that element as fixed in applications, and so omit
its dependence from the terminology.

\begin{proposition}\label{prop:harmonic-analytic-isolation-1}
  Let $f \in C_c^\infty(G)$
  be attached to $(N,\sigma)$.
  Let $\pi$ be an irreducible unitary representation of $G$.
  \begin{enumerate}
  \item If $\pi$ contains a microlocal lift $\varphi$
    of orientation $\omega$ for some $\omega \in \mathcal{X}_N^\sigma$,
    then $\pi(f)$ is the rank one orthogonal projector onto the
    subspace
    $\mathbb{C} \varphi$ of $\pi$.
  \item Otherwise, $\pi(f) = 0$.
  \end{enumerate}
\end{proposition}
\begin{proof}
  By combining the lemmas of
  \S\ref{lem:determination-microlocal-lifts},
  \S\ref{sec-5-3-3},
  \S\ref{sec-5-4}.
\end{proof}
\subsection{The $\heartsuit^{\tau}$ operator}
\label{sec:defn-heartsuit-local-again-without-Omega}
For $\tau \in k^\times$,
define $\heartsuit^{\tau} : C_c^\infty(G) \rightarrow
\mathcal{S}(B)$
by the formula $\heartsuit^{\tau} f (g) :=
1_{\mathfrak{o}^\times}(\tau \nr(g)) f(g)$.
(This is a variant of
the definition of \S\ref{sec:heartsuit},
adapted to the local setting.)

\section{Fourier transforms of convolution kernels\label{sec:fourier-transf-conv-kern}}
\label{sec-6}
We study here the ``Fourier transform''
of the $f \in C_c^\infty(G)$ attached to $(N,\sigma)$   (see \S\ref{sec:element-attached-to-N-sigma}).

This
section may be safely skipped on a first reading: the results
stated here are used in the proofs of the main results of
\S\ref{sec:local-estimates-main-statements} and \S\ref{sec:local-estimates-error}, but the details of
those proofs are not needed to understand the overall structure
of the proof given in \S\ref{sec:deduction-main-thm-microlocal}
of the main result of the article.

Throughout this section,
we fix $\sigma \in \Sigma$
and choose an arbitrary unramified unitary character $\psi : k
\rightarrow \mathbb{C}^{(1)}$
belonging to the class $\sigma$.
\subsection{Measures\label{sec:measures-for-fourier-computation}}
\label{sec-6-1}
We use $\psi$ to define measures
on $k, B, G$
as in \S\ref{sec:local-measures}.
The measures so obtained on $k$ and $B$ assign volume one to maximal compact subrings.
The measure on $G$ satisfies 
the integral formulas
\eqref{eq:integarl-formula-for-integrating-over-B}, \eqref{eqn:explicit-M2-integral-formula}.
\begin{lemma*}\label{lem:volume-of-J-bar}
  $\vol(\mathfrak{J}) = |2|_k^{-1} q^{-N} \zeta_k(1)^{-1}$.
\end{lemma*}
\begin{proof}
  This follows from the description of $\mathfrak{J}$
  given by \eqref{eqn:explicit-description-of-J-bar}, the integral formula
  \eqref{eqn:explicit-M2-integral-formula}, the local volume
  formulas
  (\S\ref{sec:local-vol-formulas}), and the
  consequence $q^{-N_1 - N_2} = |2|_k^{-1} q^{-N}$ of the definitions
  of $N_1, N_2$.
\end{proof}
\subsection{Coordinates}
\label{sec-6-2}
On $B$, we employ the coordinates
\begin{equation}\label{eq:x-xi-in-fancy-notation-etc}
  x = \begin{pmatrix}
    d - a/2 & b \\
    c & d + a/2
  \end{pmatrix},
  \quad
  \xi = \begin{pmatrix}
    \delta/2 + \alpha  & \beta  \\
    \gamma  & \delta/2 - \alpha 
  \end{pmatrix}.
\end{equation}
By a simple change of variables,
one verifies that the Haar measure on $B$ 
is given in either set of coordinates
\eqref{eq:x-xi-in-fancy-notation-etc}
by integrating over the coordinate variables
with respect to the chosen Haar measure on $k$.
The main involution $\iota : B \rightarrow B$
is
given  by
$(a,b,c,d) \mapsto (-a,-b,-c,d)$,
$(\alpha,\beta,\gamma,\delta) \mapsto
(-\alpha,-\beta,-\gamma,\delta)$
and the Fourier transform $\mathcal{F}  : \mathcal{S}(B) \rightarrow
\mathcal{S}(B)$,
defined as in \S\ref{sec:local-measures} or \S\ref{sec:local-quadratic-spaces}
by $\mathcal{F} \phi(\xi)
:=
\int_{x \in B}
\phi(x)
\psi(\langle x,\xi   \rangle)
\, d x$
with $\langle x, \xi  \rangle := \tr(x^{\iota} \xi)$,
by
\begin{equation}\label{eq:fourier-transform-definition-spelled-out}
  \mathcal{F} \phi(\xi )
  =
  \int_{a,b,c,d \in k}
  \phi(x)
  \psi(\alpha a + \delta d
  - \beta c - \gamma b).
\end{equation}

The general notation of \S\ref{sec:general-notation}
  gives us for each
  space $V \in \{C_c^\infty(G), \mathcal{S}(B)\}$
  maps
  $\mathfrak{S} : V \rightarrow V$
  and
  $\Ad(g) : V \rightarrow V$ for $g \in G$.
  These maps commute with each other
  and with $\mathcal{F}, \heartsuit^{\tau}$
  in every conceivable sense:
  $\mathfrak{S} \heartsuit^{\tau} = \heartsuit^{\tau} \mathfrak{S}$,
  $\mathfrak{S} \Ad(g) = \Ad(g) \mathfrak{S}$,
  $\heartsuit^{\tau} \Ad(g) = \Ad(g) \heartsuit^{\tau}$,
  $\mathcal{F} \Ad(g) = \Ad(g) \mathcal{F}$, $\mathcal{F}
  \mathfrak{S} =  \mathfrak{S} \mathcal{F}$.


\subsection{Statement of result}
\label{sec-6-3}
Asymptotic notation here refers to the $N \rightarrow \infty$ limit.
Implied constants depend at most upon the field $k$ and the integer $N_0$.
\begin{proposition}\label{prop:key-fourier-estimate-microlocal-kernel}
  Let $f \in C_c^\infty(G)$ be attached to $(N,\sigma)$   (see \S\ref{sec:element-attached-to-N-sigma}).
  Set
  \begin{equation}\label{eqn:phi-symmetrized-fourier-transform-of-f}
    \phi := \mathcal{F} \mathfrak{S} \heartsuit^1 f =
    \mathfrak{S} \mathcal{F} \heartsuit^1 f =
    \mathcal{F} \heartsuit^1 \mathfrak{S}  f
    \in \mathcal{S}(B).
  \end{equation}
  Define $\Phi : k^4 \rightarrow \mathbb{C}$
  in terms of the coordinates \eqref{eq:x-xi-in-fancy-notation-etc}
  by
  $\Phi(\alpha,\beta,\gamma,\delta) := \phi(\xi)$.
  One then has $\Phi(\alpha,\beta,\gamma,\delta) \neq 0$
  only if
  \begin{equation}\label{eq:asymptotic-properties-of-entries-of-support-of-fourier-transform}
    |\alpha| \asymp q^N, \quad 
    |\beta|, |\gamma| = o(q^N), \quad 
    |\delta| = O(1).
  \end{equation}
  One has
  \begin{equation}\label{eq:some-invariance-of-fourier-transform-udner-add-and-multiply}
    \Phi(u  \alpha,u \beta,u \gamma,\delta)
    = 
    \Phi(\alpha,\beta,\gamma,\delta) \text{ for all }
    u \in 1 + \mathfrak{q}^{N_0}.
  \end{equation}
  The function
  $I(\alpha,\delta) :=
  \int_{\beta,\gamma \in k} \Phi(\alpha,\beta,\gamma,\delta)$
  satisfies
  \begin{equation}\label{eq:weyl-invariance-fourier-transform-explicated-coordinates}
    I(\alpha,\delta) = I(-\alpha,\delta)
  \end{equation}
  and
  \begin{equation}\label{eq:I-rescaled-independent-of-N}
    q^{-N} I(\varpi^{-N} \alpha,\delta) \text{ is independent of $N$}.
  \end{equation}
  Moreover,
  \begin{equation}\label{eq:explicit-normalizing-computation}
    \int_{\alpha,\delta \in k}
    |2 \alpha|^{-2}
    \left\lvert
      I(\alpha,\delta)
    \right\rvert^2
    =
    C q^{N-N_0}
  \end{equation}
  with
  $C := (2 \zeta_k(1))^{-1}$.
\end{proposition}

\begin{remark*}
  We shall subsequently refer only to the properties laid out in the statement
  of  Proposition \ref{prop:key-fourier-estimate-microlocal-kernel},
  but it may be instructive to record that
  if $q$ is odd,
  one can establish
  (extending the proof
  of Proposition
  \ref{prop:key-fourier-estimate-microlocal-kernel})
  the explicit formula
  \[
  \Phi(\alpha,\beta,\gamma,\delta) =
  \frac{q^{- N_0}}{\zeta_k(1)}
  1_{\varpi^{-N} \mathfrak{o}^\times}(\alpha)
  1_{\mathfrak{q}^{-N_2}}(\beta)
  1_{\mathfrak{q}^{-N_1}}(\gamma)
  1_{\mathfrak{q}^{-N_0}}(\delta)
  \co(\delta / (\varpi^N \alpha)),
  \]
  where
  $\co(t) := (\psi(t) + \psi(-t))/2$.
  Otherwise, $k$ is a finite extension of $\mathbb{Q}_2$, and a similar
  but more complicated formula holds.
\end{remark*}

\subsection{Proofs\label{sec:main-comp-fourier-pfs}}
\label{sec-6-4}
The purpose of this section (which may be safely skipped on a first reading)
is to prove Proposition \ref{prop:key-fourier-estimate-microlocal-kernel}.
  \begin{lemma*}\label{lem:easy-formula-for-f-yay}
    Let $f$ be attached to $(N,\sigma)$.
    Let $x \in B$.
    In the coordinates \eqref{eq:x-xi-in-fancy-notation-etc},
    \begin{equation}\label{eq:easy-formula-for-f-yay}
      \heartsuit^1 f(x)
      =
      C_0
      1_{\mathfrak{o}^\times}(d)
      1_{\mathfrak{q}^{N-N_0}}(a)
      1_{\mathfrak{q}^{N_1}}(b)
      1_{\mathfrak{q}^{N_2}}(c)
      \psi (\frac{a d - b c}{\varpi^N d^2})
    \end{equation}
    where
    $C_0 := q^{2 N - N_0} |2|_k$.
  \end{lemma*}
  \begin{proof}
    By \eqref{eqn:index-of-O-0} and the lemma of \S\ref{lem:volume-of-J-bar},
    \begin{equation}\label{eqn:compute-volumes-J-etc}
      C_0 = [\mathfrak{o}^\times:\mathfrak{o}_0^\times]
    \vol(\mathfrak{J})^{-1}.
    \end{equation}
    By Fourier analysis on $\mathfrak{o}^\times$,  we have for $u \in \mathfrak{o}^\times$
   the expansion
  \begin{equation}\label{eq:fourier-expansion-of-zeta-local}
    [\mathfrak{o}^\times:\mathfrak{o}_0^\times]
       1_{\mathfrak{o}_0^\times}(u) \omega_{\sigma}(u)
      = \sum_{\omega \in \mathcal{X}_N^\sigma} \omega(u).
    \end{equation}
    Applying \eqref{eqn:compute-volumes-J-etc} and
    \eqref{eq:fourier-expansion-of-zeta-local} to the definition
    of $f$
    gives $\heartsuit^1 f(x) = C_0 \kappa \psi(\zeta)$, where
    \[
    \kappa := 
    1_{\mathfrak{q}^{N_1}}(b)
    1_{\mathfrak{q}^{N_2}}(c)
    1_{\mathfrak{o}^\times}(d - a/2)
    1_{\mathfrak{o}^\times}(d + a/2)
    1_{\mathfrak{q}^{N-N_0}}
    (\frac{\nr(g)}{(d-a/2)^2} - 1),
    \]
    \[
    \zeta
    :=
    \varpi^{-N}
    (
    \frac{\nr(x)}{(d-a/2)^2} - 1
    ).
    \]
    By our assumption \eqref{eq:basic-N-N0-ord-2-assumptions} 
    on the largeness of $N$ relative to $N_0$
    and identifies such as
    \begin{equation}\label{eq:first-determinant-expansion-in-annoying-f-initial-lemma}
      \frac{\nr(x)}{(d-a/2)^2} - 1
      = \frac{a}{d-a/2}
      - \frac{b c}{(d - a/2)^2},
    \end{equation}
    we verify directly that
    \begin{equation}\label{eq:basic-congruence-equivalence-initial-study-of-f}
      \kappa
      =
      1_{\mathfrak{o}^\times}(d)
      1_{\mathfrak{q}^{N-N_0}}(a)
      1_{\mathfrak{q}^{N_1}}(b)
      1_{\mathfrak{q}^{N_2}}(c).
    \end{equation}
    Similarly, we verify using \eqref{eq:basic-N-N0-ord-2-assumptions-2}
    that for $a,b,c,d$
    in the support of the RHS of \eqref{eq:basic-congruence-equivalence-initial-study-of-f},
    the congruences
    \[
    \frac{a}{d-a/2}
    \equiv \frac{a d}{d^2} \pmod{\mathfrak{q}^N},
    \quad
    \frac{b c}{(d-a/2)^2}
    \equiv \frac{b c}{d^2} \pmod{\mathfrak{q}^N}
    \]
    hold.
    It follows
    from these and
    \eqref{eq:first-determinant-expansion-in-annoying-f-initial-lemma}
    that
    \[
    \kappa \psi(\zeta)
    = \kappa \psi(\frac{a d  - b c}{2 \varpi^N d^2}).
    \]
    This completes the proof of \eqref{eq:easy-formula-for-f-yay}.
  \end{proof}
  \begin{corollary*}
    With notation as in the lemma,
    the quantity
    $f(x)$ depends only upon the congruence classes
    \[
    a \pmod{\mathfrak{q}^{N}},
    \quad 
    b \pmod{2 \mathfrak{q}^{N_1}},
    \quad 
    c \pmod{2 \mathfrak{q}^{N_2}},
    \quad
    d \pmod{\mathfrak{q}^{N_0}}.
    \]
  \end{corollary*}
  \begin{proof}
    Immediate;
    it may help to recall that
    $N = \ord_k(2) + N_1 + N_2$.
  \end{proof}

    We now prove Proposition \ref{prop:key-fourier-estimate-microlocal-kernel}.
    By the formula
    for $\heartsuit^1 f$ given in the lemma
    and the explication
    \eqref{eq:fourier-transform-definition-spelled-out}
    of the Fourier transform,
    we have
    \[
    \Phi(\alpha,\beta,\gamma,\delta)
    =
    \int_{a,b,c,d \in k}
    F(a,b,c,d)
    \psi(
    \alpha a + \delta d
    - \beta c - \gamma b),
    \]
    where
    \begin{equation}\label{eq:defn-F-explicated-f}
      F(a,b,c,d)
      :=
      C_0
      1_{\mathfrak{o}^\times}(d)
      1_{\mathfrak{q}^{N-N_0}}(a)
      1_{\mathfrak{q}^{N_1}}(b)
      1_{\mathfrak{q}^{N_2}}(c)
      \co(\frac{a d - b c}{\varpi^N d^2}),
    \end{equation}
    with
    $\co(x) := 
    ( \psi (x)
      +
      \psi (-x))/
      2$.
    The smoothness properties of $f$ (hence of $F$),
    as enunciated in the corollary,
    imply corresponding decay properties of its
    Fourier transform, namely that $\Phi(\alpha,\beta,\gamma,\delta)
    \neq 0$
    only if $\alpha \in \mathfrak{q}^{-N},
    \beta \in 2^{-1} \mathfrak{q}^{-N_2},
    \gamma \in 2^{-1} \mathfrak{q}^{-N_1},
    \delta \in \mathfrak{q}^{-N_0}$.
    We thereby obtain all assertions
    in
    \eqref{eq:asymptotic-properties-of-entries-of-support-of-fourier-transform}
    except for the lower bound on $|\alpha|$.
    To establish the latter,
    we compute for $d \in \mathfrak{o}^\times$ that
    \begin{equation}\label{eq:computation-relevant-for-alpha-support-of-fourier-transform}
      \int_{a \in k}
      1_{\mathfrak{q}^{N-N_0}}(a)
      \co (\frac{a d }{\varpi^N d^2})
      \psi(\alpha a)
      =
      q^{N_0-N}
      \rho(\alpha,d),
    \end{equation}
    where
    \[
    \rho(\alpha,d) :=
    \frac{1_{\mathfrak{q}^{N_0-N}}(\alpha + \frac{1}{
        \varpi^N d}) +
      1_{\mathfrak{q}^{N_0-N}}(\alpha - \frac{1}{\varpi^N d})
    }{2
    }.
    \]
    One has $\rho(\alpha,d) \neq 0$
    only if
    $\alpha \in \pm \frac{1}{\varpi^N d} + 
    \mathfrak{q}^{N_0 - N}
    \subseteq \frac{1}{\varpi ^N} \mathfrak{o}^\times$,
    in which case $|\alpha| \asymp q^N$.

    The invariance
    \eqref{eq:some-invariance-of-fourier-transform-udner-add-and-multiply}
    is equivalent to
    the identity $F(u a, u b, u c,d) = F(a,b,c,d)$
    for $u \in 1 + \mathfrak{q}^{N_0}$,
    which follows from \eqref{eq:defn-F-explicated-f}
    and the  congruence
    $u a d - u^2 b c \equiv a d - b c \pmod{\mathfrak{q}^N}$
    for $(a,b,c,d) \in \supp(F)$.

    We now verify
    \eqref{eq:weyl-invariance-fourier-transform-explicated-coordinates}.
    By Fourier inversion,
    \begin{equation}\label{eq:I-of-alpha-delta-after-Fourier-inversion}
      I(\alpha,\delta)
        =
        \int_{a,d \in k}
        F(a,0,0,d)
        \psi(\alpha a + \delta d).
      \end{equation}
      Thus \eqref{eq:defn-F-explicated-f}
      reduces to $F(a,0,0,d) = F(-a,0,0,d)$,
      follows from \eqref{eq:defn-F-explicated-f}.

      To establish \eqref{eq:I-rescaled-independent-of-N}, we 
      recall the definition of $C_0$ from the lemma and
      apply to
      \eqref{eq:defn-F-explicated-f} and
      \eqref{eq:I-of-alpha-delta-after-Fourier-inversion} 
      the change of variables $a \mapsto \varpi^N a$,
      giving
      \[
      I(\varpi^{-N} \alpha,\delta) =
      C'
      q^N
      \int_{a,d \in k}
      1_{\mathfrak{o}^\times}(d)
      1_{\mathfrak{q}^{-N_0}}(a)
      \co (a/d)
      \psi(\alpha a + \delta d)
      \]
      for some unimportant scalar $C' = |2| q^{-N_0}$
      that does not depend upon $N$.

      We turn finally to
      \eqref{eq:explicit-normalizing-computation}.
      Denote temporarily by $\mathcal{I}$ its LHS.
      By \eqref{eq:I-of-alpha-delta-after-Fourier-inversion}
      and Parseval,
      \begin{equation}\label{eq:normalizing-computation-consequence-of-parseval}
        \mathcal{I}  = 
        \int_{\alpha,d \in k}
        |2 \alpha|^{-2}
        \left\lvert
          \int_{a \in k}
          F(a,0,0,d)
          \psi(\alpha a)
        \right\rvert^2.
      \end{equation}
      We compute 
      with the help of
      \eqref{eq:computation-relevant-for-alpha-support-of-fourier-transform}
      that
      \[
      \int_{a \in k}
      F(a,0,0,d)
      \psi(\alpha a)
      =
      q^{N_0-N} C_0
      1_{\mathfrak{o}^\times}(d)
      \rho(\alpha,d).
      \]
      Substituting this and the identity  $q^{N_0-N} C_0 = q^N |2|_k$
      into
      \eqref{eq:normalizing-computation-consequence-of-parseval}
      gives
      \[
      \mathcal{I} =
      q^{2 N}
      \int_{\alpha,d \in k}
      |\alpha  |^{-2} 1_{\mathfrak{o}^\times}(d)
      |\rho(\alpha,d)|^2.
      \]
      We substitute $\alpha \mapsto \varpi^{-N} \alpha$;
      since $|\varpi^{-N}|^{-1} = q^{-N}$,
      we obtain
      \[
      \mathcal{I} =
      q^{N}
      \int_{\alpha,d \in k}
      |\alpha|^{-2}
      1_{\mathfrak{o}^\times}(d)
      \left\lvert \frac{1_{\mathfrak{q}^{N_0}}(\alpha + 1/d)
          +
          1_{\mathfrak{q}^{N_0}}(\alpha - 1/d)
        }{
          2}
      \right\rvert^2.
      \]
      By the consequence
      $2 \notin \mathfrak{q}^{N_0}$
      of the assumption
      \eqref{eq:basic-N-N0-ord-2-assumptions},
      we have
      \[1_{\mathfrak{o}^\times}(d)
      1_{\mathfrak{q}^{N_0}}(\alpha + 1/d)
      1_{\mathfrak{q}^{N_0}}(\alpha - 1/d) = 0,\]
      whence
      \begin{align*}
        \mathcal{I}
        &= q^N
          \int_{\alpha,d \in k}
          |\alpha|^{-2}
          1_{\mathfrak{o}^\times}(d)
          \frac{1_{\mathfrak{q}^{N_0}}(\alpha + 1/d)
          +
          1_{\mathfrak{q}^{N_0}}(\alpha - 1/d)
          }{
          2^2}
        \\
        &= 
          \frac{2 q^{N-N_0} ( 1 - \tfrac{1}{q}) }{2^2}
          = 
          \frac{q^{N-N_0}  }{2 \zeta_k(1)}
      \end{align*}
      We thereby arrive at
      \eqref{eq:explicit-normalizing-computation}
      with the constant
      \[
      C =
      q^{N_0-N}
      \cdot 
      \frac{q^{N-N_0}  }{2 \zeta_k(1)}
      = 
      \frac{1}{2 \zeta_k(1)},
      \]
      as required.
\section{Estimates for the main term\label{sec:local-estimates-main-statements}}
\label{sec-7}
\subsection{Statement of result}
\label{sec-7-1}
Let $E$ denote the diagonal subalgebra of $B$.
Let $H \leq G$ denote the image of $E^\times$.
Thus $H = \{a(y) : y \in k^\times \}$
is the diagonal split torus.
Let $N(H)$ denote the normalizer in $G$ of $H$.
Let
\[
  W := \left\{ \begin{pmatrix}
      1 &  \\
      & 1
    \end{pmatrix},
    \begin{pmatrix}
      & 1 \\
      1 & 
    \end{pmatrix}\right\}
\]
denote the Weyl group $W \cong N(H) / H$.
Equip $H$ and $N(H)$ the measures
\[
  \int_H f
  :=
  \int_{y \in k^\times}
  f(y) \, \frac{d y}{|y|},
  \quad
  \int_{N(H)}
  f :=
  \sum_{w \in W} \int_{h \in H}
  f(w h),
\]
where $d y$ denotes (as in \S\ref{sec:measures-for-fourier-computation}) the Haar measure
on $k$ assigning volume one to $\mathfrak{o}$.
\begin{proposition}\label{prop:desired-main-term-identity-do-it-up}
  Let $\Psi : G \rightarrow \mathbb{C}$
  be a function
  with the following properties:
  \begin{enumerate}
  \item There is an open subgroup $U$ of $G$ so that
    \begin{equation}\label{eq:matrix-coeff-bi-invariance}
      \Psi(u_1 g u_2)  = \Psi(g) \text{ for all $u_1,u_2 \in U$, $g \in G$.}
    \end{equation}
  \item One has (see \S\ref{sec:local-Xi} concerning $\Xi$)
    \begin{equation}\label{eq:matrix-coeff-decay-weak-generic}
      \Psi(g) \ll \Xi(g)^{\delta} \text{ for some } \delta > 0.
    \end{equation}
  \end{enumerate}
  Let $N$ be a positive integer taken sufficiently large in terms of $\Psi$.
  Let $f \in C_c^\infty(G)$
  be attached to $(N,\sigma)$
  for some $\sigma \in \Sigma$
  (see \S\ref{sec:element-attached-to-N-sigma}).
  Then
  \begin{equation}\label{eqn:main-term-what-we-wanna-show-just-after-reducing-to-local-problem}
    \int_{g \in G}
    \langle \Ad(g) \mathfrak{S} f, \mathfrak{S} f
    \rangle_{L^2(G)}
    \Psi(g)
    =
    q^{N-N_0} \frac{1}{2}
    \int_{N(H)} \Psi.
  \end{equation}
\end{proposition}

\begin{remark}
  The condition \eqref{eq:matrix-coeff-decay-weak-generic}
  implies the absolute convergence of both sides of
  \eqref{eqn:main-term-what-we-wanna-show-just-after-reducing-to-local-problem}
  (see \S\ref{sec:local-convergence-lemmas}).
\end{remark}
\begin{remark}\label{rmk:main-term-thm-applies-to-mx-coefs}
  Let $\pi$ be an irreducible unitary representation of
  $G$ with $\dim(\pi) > 1$, and let $v_1,v_2 \in \pi$.
  The hypotheses of Proposition
  \ref{prop:desired-main-term-identity-do-it-up}
  are then satisfied by $\Psi(g) := \langle g v_1, v_2 \rangle$
  (see \S\ref{sec:local-Xi}).
\end{remark}
\begin{remark}
  The LHS of
  \eqref{eqn:main-term-what-we-wanna-show-just-after-reducing-to-local-problem}
  is independent of the choice of Haar measure on $G$
  (noting that the definition of $f$ involves
  the factor $\vol(\mathfrak{J})^{-1}$).
\end{remark}
The proof of Proposition \ref{prop:desired-main-term-identity-do-it-up}
occupies the remainder of \S\ref{sec:local-estimates-main-statements}.

\subsection{Reduction to matrix calculus\label{sec:determine-main-term}}
\label{sec-7-2}
We fix measures on $G,B,k$ and define $\mathcal{F}$ as in
\S\ref{sec:measures-for-fourier-computation}.

\subsubsection{Application of Parseval}\label{sec:application-parseval}
Set
$\phi := \mathcal{F} \mathfrak{S} \heartsuit^1 f \in \mathcal{S}(B)$.
By Parseval,
\eqref{eq:integarl-formula-for-integrating-over-B} and
the volume formulas of \S\ref{sec:local-vol-formulas},
we have
for $g \in G$ that
\begin{equation}\label{eq:parseval-application-to-main-term}
  \langle \Ad(g) \mathfrak{S} f, \mathfrak{S} f \rangle_{L^2(G)}
  = \zeta_k(1)
  \langle \Ad(g) \phi, \phi  \rangle_{L^2(B)}.
\end{equation}

\begin{remark*}
  We may now informally explain
  \eqref{eqn:main-term-what-we-wanna-show-just-after-reducing-to-local-problem}
  as follows: The
  support properties
  \eqref{eq:asymptotic-properties-of-entries-of-support-of-fourier-transform}
  say that $\phi$ is concentrated on the subspace of diagonal
  matrices, whose normalizer is $N(H)$.  Thus
  $\langle \Ad(g) \phi, \phi \rangle$ should be small unless $g$
  is close to $N(H)$.  One has (morally) $\Ad(h) \phi \approx \phi$ for
  elements $h \in H$ of size $O(1)$.
  The identity
  \eqref{eq:weyl-invariance-fourier-transform-explicated-coordinates}
  says (morally) that $\Ad(w) \phi \approx \phi$ for all
  $w \in W$.  Thus the distribution
  $g \mapsto \langle \Ad(g) \phi, \phi \rangle$ should
  conceivably approximate some multiple of the Haar measure on
  $N(H)$.  
\end{remark*}

\subsubsection{Principal congruence subgroups}\label{sec:princ-congr-subgr}
For a positive integer $m$, we let $K[m]$ denote the
$m$th principal congruence subgroup of $G$;
we define this to mean the image in $G$ 
of the depth $m$ principal congruence subgroup
\[
  \begin{pmatrix}
    1 + \mathfrak{q}^m  & \mathfrak{q}^m \\
    \mathfrak{q}^m & 1 + \mathfrak{q}^m
  \end{pmatrix}
  \leq \GL_2(\mathfrak{o}).
  \]
  One has a diffeomorphism
  \begin{equation}\label{eqn:description-of-Km}
    \begin{split}
    \mathfrak{q}^m \times \mathfrak{q}^m \times (1 +
    \mathfrak{q}^m)
    &\xrightarrow{\cong} K[m] \\
    (x_1,x_2,y)
    &\mapsto n'(x_2) n(x_1) a(y).
  \end{split}
\end{equation}

\subsubsection{Properties of $\phi$}
Let $m$ be any positive integer
for which $m \geq N_0$.
Proposition \ref{prop:key-fourier-estimate-microlocal-kernel}
implies that
\begin{equation}\label{eq:unit-invariance-assumed-of-phi}
  \phi(\begin{pmatrix}
    \delta/2 + \alpha  & \beta  \\
    \gamma  & \delta/2 - \alpha 
  \end{pmatrix})
  =
  \phi(\begin{pmatrix}
    \delta/2 + u \alpha  & u \beta  \\
    u \gamma  & \delta/2 - u \alpha 
  \end{pmatrix})
  \text{ for all }
  u \in 1 + \mathfrak{q}^m
\end{equation}
and for $N$ large enough also that
\begin{equation}\label{eq:support-phi-contained-in-E-nought-of-m}
  \supp(\phi) \subseteq E(m),
\end{equation}
where $E(m)$ denotes the following set of ``near-diagonal'' matrices:
\begin{equation}\label{eq:definition-of-sort-of-neighborhood-of-diagonals}
  E(m) := \left\{
    \begin{pmatrix}
      \delta/2 + \alpha  & \beta  \\
      \gamma  & \delta/2 - \alpha 
    \end{pmatrix}
    :
    \delta \in k, \alpha \in k;
    \beta,\gamma \in 2 \alpha \mathfrak{q}^m
  \right\}.
\end{equation}
Recall also from
\eqref{eq:weyl-invariance-fourier-transform-explicated-coordinates}
that
\begin{equation}\label{eq:weyl-invariance-after-integrating-out-upper-and-lower-variables}
  \int_{\beta,\gamma \in k}
  \phi (
  \begin{pmatrix}    
    \delta/2 + \alpha  & \beta  \\
    \gamma  & \delta/2 - \alpha 
  \end{pmatrix}
  )
  = 
  \int_{\beta,\gamma \in k}
  \phi (
  \begin{pmatrix}
    \delta/2 - \alpha  & \beta  \\
    \gamma  & \delta/2 + \alpha 
  \end{pmatrix}
  ).
\end{equation}

\subsubsection{The key computation}\label{sec:key-computation}
Proposition \ref{prop:desired-main-term-identity-do-it-up} follows
immediately from \eqref{eq:parseval-application-to-main-term},
the normalizing computation \eqref{eq:explicit-normalizing-computation}
(giving the factor $(2 \zeta_k(1))^{-1}$)
and the following:
\begin{lemma*} 
  Let $\Psi : G \rightarrow \mathbb{C}$ 
  and $\phi \in \mathcal{S}(B)$ be arbitrary.
  Suppose there exists an open subgroup $U$ of $G$
  and a positive integer $m$
  so that
  \begin{equation}\label{eq:concentration-of-phi-quantified-eps-U}
    K[m] \leq U.
  \end{equation}
  and so that
  \eqref{eq:matrix-coeff-bi-invariance},
  \eqref{eq:matrix-coeff-decay-weak-generic},
  \eqref{eq:unit-invariance-assumed-of-phi},
  \eqref{eq:support-phi-contained-in-E-nought-of-m} and
  \eqref{eq:weyl-invariance-after-integrating-out-upper-and-lower-variables}
  hold.
  Then
  \[
    \int_{g \in G}
    \langle \Ad(g) \phi, \phi  \rangle
    \Psi(g)
    = 
  \]
  \[(\int_{N(H)} \Psi )
    \int_{\alpha,\delta \in k}
    |2 \alpha|^{-2}
    \left\lvert
      \int_{\beta,\gamma \in k}
      \phi (
      \begin{pmatrix}
        \delta/2 + \alpha  & \beta   \\
        \gamma  & \delta/2 - \alpha 
      \end{pmatrix}
      )
    \right\rvert^2.
  \]
\end{lemma*}
The proof involves
Hensel's lemma, the Weyl
integral formula, and related arguments; we
complete it in \S\ref{sec:appendix-proof-theorem}.

\begin{remark*}
  The lemma and its proof generalize readily to non-split
  quaternion algebras $B$ and/or non-split separable quadratic
  subalgebras $E \subseteq B$.
  We focus on the relevant split diagonal case
  for the sake of concreteness.
\end{remark*}

\subsection{Some matrix calculus\label{sec:some-mx-calc}}
\label{sec-7-3}
One purpose of this section
is to prove
the lemma of \S\ref{sec:key-computation}.
Another is to develop preliminaries for the proof of
Proposition \ref{prop:local-error-estimates-stmt}, below.

\subsubsection{Notation}
\label{sec-7-3-1}
We introduce on $B^0$ the coordinates
\[
[\alpha,\beta,\gamma] := \begin{pmatrix}
  \alpha  & \beta  \\
  \gamma  &  - \alpha 
\end{pmatrix}.
\]
The Weyl group $W \cong N(H)/H$
has a natural right action on $G/H$, denoted by juxtaposition.
Let $(G/H)[m]$ denote the image of $K[m]$
under the quotient map $G \rightarrow G/H$;
by \eqref{eqn:description-of-Km},
the map $\mathfrak{q}^m \times \mathfrak{q}^m
\ni (x_1,x_2) \mapsto n'(x_1) n(x_2) H \in (G/H)[m]$
is a bijection.
Let $E^0 := E \cap B^0 = \{[\alpha,0,0] : \alpha \in k\}$ denote the subspace
of traceless diagonal matrices.
Set
$E^0(m) :=
\{[\alpha,\beta,\gamma] : \alpha \in k;
\beta,\gamma \in 2 \alpha \mathfrak{q}^m  \}
\subseteq B^0$.
The sets $E^0(m) - \{0\}$ form a shrinking system of
open neighborhoods of $E^0 - \{0\}$. Their images
$\mathbb{P} E^0(m)$ in the projective plane $\mathbb{P} B^0$
form a fundamental system of open neighborhoods of the point
$\mathbb{P} E^0$.
\subsubsection{Measures}
\label{sec-7-3-2}
We fix measures on $G,B,B^0,k$ as in
\S\ref{sec:measures-for-fourier-computation}.
One then has
\[
\int_{B^0}
f
= \int_{\alpha,\beta,\gamma \in k}
f ([\alpha,\beta,\gamma]),
\]
\begin{equation}\label{eq:formula-relating-B-and-B0}
  \int_{B} f
  = \int_{\alpha,\beta,\gamma,\delta \in k}
  f (\begin{pmatrix}
    \delta/2 + \alpha  & \beta  \\
    \gamma  & \delta/2-\alpha 
  \end{pmatrix})
  =
  \int_{\delta \in k}
  (\int_{B_0} f_\delta),
\end{equation}
where
for
$f \in C_c(B)$
we define $f_\delta \in C_c(B^0)$
by $f_\delta(\xi) := f(\delta/2 + \xi)$.

We use the notation $\mathbb{E}$
to denote an average over a compact group or orbit thereof
with respect to the evident invariant probability measure.
Averages over $K[m]$
may be computed
by the formula (cf. \eqref{eqn:description-of-Km})
\begin{equation}\label{eq:formula-averages-over-Km}
  \mathbb{E}_{g \in K[m]}
  f(g)
  =
  \mathbb{E}_{\substack{
      x,y \in \mathfrak{q}^m \\
      z \in 1 + \mathfrak{q}^m
    }
  }
  f(n'(x) n(y) a(z)).
\end{equation}

Equip $G/H$ with the quotient Haar;
by \eqref{eqn:explicit-M2-integral-formula},
\begin{equation}\label{eq:integral-formula-G-mod-H}
  \int_{g \in G/H} f(g)
  = \int_{x_1,x_2 \in k}
  f(n'(x_2) n(x_1)),
\end{equation}
\begin{equation}\label{eq:integral-over-G-mod-H-m}
  \int_{x \in (G/H)[m]}
  f(x)
  :=  \int_{x \in G/H} 1_{(G/H)[m]}(x)
  f(x)
  =
  \int_{x_1,x_2 \in \mathfrak{q}^m}
  f(n'(x_1) n(x_2)),
\end{equation}
\begin{equation}\label{eq:integral-over-G-mod-H-m-averaged}
  \mathbb{E}_{x \in (G/H)[m]}
  f(x)
  =
  \mathbb{E}_{x_1,x_2 \in \mathfrak{q}^m}
  f(n'(x_1) n(x_2)).
\end{equation}

\subsubsection{Basic observations}

By inspection,
\begin{equation}\label{eq:weyl-invariance-of-cone-nbhd}
  \Ad(N(H)) E^0 = E^0,
  \quad 
  \Ad(W) E^0(m) = E^0(m).
\end{equation}
By direct calculations
such as
\begin{equation}\label{eq:basic-conjugation-action-on-diagonal}
  \Ad(n'(x) n(y) a(z) [1,0,0])
  =
  [1 + 2 x y, - 2 y, 2 x(1 + 2 y)].
\end{equation}
one verifies also that
\begin{equation}\label{eq:adjoint-action-small-elements-near-preserve-diagonal-neighborhood}
  \Ad(K[m]) E^0(m) = E^0(m).
\end{equation}

\subsubsection{Applications of Hensel's lemma}
\label{sec-7-3-3}

\begin{lemma}\label{lem:hensel-lemma-statement}
  Let $m,n \in \mathbb{Z}_{\geq 1}$.
  Set $M := (\mathfrak{q}^m)^{\oplus n}$.
  Let $f_1,\dotsc,f_n \in \mathfrak{o}[X_1,\dotsc,X_n]$ be
  polynomials in the variables $X_1,\dotsc,X_n$ with
  coefficients in $\mathfrak{o}$
  satisfying $f_1(0) = \dotsb = f_n(0) = 0$.
  Let $f : M \rightarrow k^{\oplus n}$
  denote the function given by
  $x \mapsto (f_1(x),\dotsc,f_n(x))$.  Set
  $J :=\det(\partial f_i / \partial x_j) \in
  \mathfrak{o}[X_1,\dotsc,X_n]$.  Assume that
  $J(x) \in \mathfrak{o}^\times$ for all $x \in M$.  Then $f$
  induces
  a diffeomorphism $f : M \rightarrow M$.
\end{lemma}
\begin{proof}
  One argues as in the proof of Hensel's lemma that $f$ is
  bijective.  The conclusion then follows from the inverse
  function theorem.
\end{proof}

\begin{lemma}\label{key-application-hensel-lemma}
  Let $\alpha_0 \in k - \{0\}$.
  Then the map
  \[
  (1 + \mathfrak{q}^m)
  \times 
  (G/H)[m]
  \rightarrow 
  [(1 + \mathfrak{q}^m) \alpha_0, 2 \mathfrak{q}^m \alpha_0, 2 \mathfrak{q}^m \alpha_0]
  \]
  \[
  (\lambda,x)
  \mapsto
  \lambda \Ad(x) [\alpha_0,0,0]
  \]
  is a well-defined diffeomorphism
  whose Jacobian has constant valuation. 
\end{lemma}
\begin{proof}
  We may assume that $\alpha_0 = 1$.
  For $x_1,x_2,x_3 \in \mathfrak{q}^m$,
  define $y_1,y_2,y_3 \in k$
  by
  $(1 + x_1) \Ad(n'(x_2) n(x_3)) [1,0,0]
  = [1 + y_1, 2 y_2, 2 y_3]$.
  By a calculation similar
  to \eqref{eq:basic-conjugation-action-on-diagonal},
  \begin{align*}
    y_1 &= x_1 + 2 x_2 x_3 + 2 x_1 x_2 x_3, \\
    y_2 &= - x_3 ( 1 + x_1), \\
    y_3 &= x_2 (1 + x_1) (1 + 2 x_3).
  \end{align*}
  Since $m \geq 1$, one sees that $y_1,y_2,y_3 \in \mathfrak{q}^m$
  and that the Jacobian of the map $(x_1,x_2,x_3) \mapsto
  (y_1,y_2,y_3)$
  belongs to $\mathfrak{o}^\times$ at every point of its domain.
  The conclusion follows now from Lemma \ref{lem:hensel-lemma-statement}. 
\end{proof}

\begin{corollary*}\label{cor:we-can-write-stuff-near-diagonal-nicely}
  For each $\xi = [\alpha,\beta,\gamma] \in E^0(m)$
  there exists $\alpha_0 \in \alpha(1 + \mathfrak{q}^m)$
  and $g \in K[m]$
  so that $\xi = \Ad(g) [\alpha_0,0,0]$.
\end{corollary*}

\begin{lemma}\label{lem:basic-integral-formula-without-hypotheses}
  Let $f \in C_c(B^0)$
  and
  $\xi_0 = [\alpha_0,\beta_0,\gamma_0] \in E^0(m) - \{0\}$.
  Then
  \[
  \mathbb{E}_{
    \substack{
      \lambda  \in 1 + \mathfrak{q}^m
      \\
      g \in K[m]
    }
  }
  f(\lambda \Ad(g) \xi_0)
  =
  \mathbb{E} _{
    \substack{
      \alpha \in (1 + \mathfrak{q}^m) \alpha_0 \\
      \beta, \gamma \in 2 \alpha_0 \mathfrak{q}^m
    }
  }
  f ([\alpha,\beta,\gamma]).
  \]
\end{lemma}
\begin{proof}
  By the corollary, we may
  write $\xi_0 = \Ad(g_0)[\alpha_0',0,0]$ with
  $\alpha_0' \in (1 + \mathfrak{q}^m) \alpha_0$ and
  $g_0 \in K[m]$.  By the change of variables
  $g \mapsto g g_0^{-1}$, we reduce to proving the required
  identity in the special case $\beta_0 = \gamma_0 = 0$.
  In that case,
  we replace the average over $g \in K[m]$
  with an average over $x \in (G/H)[m]$
  and appeal to
  Lemma \ref{key-application-hensel-lemma}
  and the change of variables formula.
\end{proof}

\begin{lemma}\label{lem:basic-integral-formula-for-diag-inv-f}
  Suppose $f \in C_c(B^0)$
  is supported on $E^0(m)$
  and satisfies
  $f(\lambda x) = f(x)$
  for all $x \in B^0$ and $\lambda \in 1 + \mathfrak{q}^m$.
  Let $\xi_0 \in B^0 - \{0\}$.
  Then
  \begin{align}\label{eqn:basic-integral-formula-for-diag-inv-f}
    \mathbb{E}_{g \in K[m]}
    f(\Ad(g) \xi_0)
    &=
      1_{E^0(m)}(\xi_0)
      \mathbb{E}_{\beta,\gamma \in 2 \mathfrak{q}^m \alpha_0} f
      ([\alpha_0,\beta,\gamma])
    \\
    \nonumber
    &=
      1_{\alpha_0 \neq 0}
      e_{2 \mathfrak{q}^m \alpha_0}(\beta_0)
      e_{2 \mathfrak{q}^m \alpha_0}(\gamma_0)
      \int_{\beta,\gamma \in k}
      f ([\alpha_0,\beta,\gamma,-\alpha_0])
  \end{align}
  where
  $e_\mathfrak{a} := \vol(\mathfrak{a})^{-1} 1_\mathfrak{a}$
  denotes
  the normalized characteristic function of a compact open subgroup $\mathfrak{a}$ of $k$.
\end{lemma}
\begin{proof}
  If $\xi_0 \notin E^0(m)$, then the vanishing of the LHS
  follows from
  \eqref{eq:adjoint-action-small-elements-near-preserve-diagonal-neighborhood},
  so suppose that $\xi_0 \in E^0(m)$.
  The $(1+\mathfrak{q}^m)$-invariance of $f$
  allows us to rewrite
  the LHS of \eqref{eqn:basic-integral-formula-for-diag-inv-f}
  as
  $\mathbb{E}_{\lambda \in 1 + \mathfrak{q}^m, g \in K[m]}
  f(\lambda \Ad(g) \xi_0)$.
  We apply Lemma
  \ref{lem:basic-integral-formula-without-hypotheses}
  to the latter,
  invoking again the $(1+\mathfrak{q}^m)$-invariance of $f$
  to simplify the conclusion.
\end{proof}

\begin{lemma}\label{lem:useful-for-main-term-evaluation-inverse-weyl-integral-formula}
  Let $f$ be as in Lemma
  \ref{lem:basic-integral-formula-for-diag-inv-f}
  and $\alpha_0 \in k - \{0\}$.
  Then
  \begin{equation}\label{eqn:useful-for-main-term-evaluation-inverse-weyl-integral-formula}
    \int_{x \in (G/H)[m]}
    f(\Ad(x) [\alpha_0,0,0])
    =
    |2 \alpha_0|^{-2}
    \int_{\beta,\gamma \in k}
    f([\alpha_0,\beta,\gamma]).
  \end{equation}
\end{lemma}
\begin{proof}
  Using the integral formulas
  \eqref{eq:formula-averages-over-Km}
  and \eqref{eq:integral-over-G-mod-H-m},
  the LHS of
  \eqref{eqn:useful-for-main-term-evaluation-inverse-weyl-integral-formula}
  may be rewritten
  $\vol(\mathfrak{q}^m)^2
  \mathbb{E}_{g \in K[m]} f(\Ad(g)[\alpha_0,0,0])$.
  The conclusion follows from Lemma
  \ref{lem:basic-integral-formula-for-diag-inv-f}
  upon noting that
  $\vol(\mathfrak{q}^m)^2 e_{2 \alpha_0 \mathfrak{q}_m}(0)^2
  = |2 \alpha_0|^{-2}$.
\end{proof}
\subsubsection{Expanding near the singular points}
\label{sec-7-3-4}
\begin{lemma}\label{lem:we-know-when-conjugation-nearly-stabilizes-diagonal}
  For $x \in G/H$,
  the following are equivalent:
  \begin{enumerate}
  \item[(i)] There exists $w \in W$
    so that $x w \in (G/H)[m]$.
  \item[(ii)] There exists $\tau \in E^0 - \{0\}$
    so that $\Ad(x) \tau \in E^0(m)$.
  \item[(iii)] For all $\tau \in E^0$,
    one has $\Ad(x) \tau \in E^0(m)$.
  \end{enumerate}
\end{lemma}
\begin{proof}
  (i)
  implies (ii):
  immediate from
  \eqref{eq:weyl-invariance-of-cone-nbhd} and
  \eqref{eq:adjoint-action-small-elements-near-preserve-diagonal-neighborhood}.
  (ii) implies (iii): immediate from the
  definitions and the fact that $\dim(E^0) = 1$.
  (iii) implies (i):
  By the corollary
  to Lemma \ref{key-application-hensel-lemma}
  of
  \S\ref{sec-7-3-3},
  we have $\Ad(x) [1,0,0] \in  \Ad(g) E^0$
  for some $g \in K[m]$;
  then $\Ad(g^{-1} x) E^0 = E^0$,
  hence $g^{-1} x \in N(H) = W H$, hence
  $x \in K[m] W H$.
\end{proof}

\begin{lemma}\label{lem:expansion-integrals-near-singular-pts}
  ~
  \begin{enumerate}
  \item[(i)]
    Suppose $f \in C_c(G/H)$
    is supported on $(G/H)[m] W$.
    Then
    \[
    \int_{G/H} f
    = \sum_{w \in W}
    \int_{x \in (G/H)[m]} f(x w).
    \]
  \item[(ii)]
    Suppose
    $f_1 \in C_c(B^0)$
    is supported on $E^0(m)$
    and $f_2 \in C_c(G/H)$ is arbitrary.
    Let $\tau \in E^0$.  Then
    \[
    \int_{x \in G/H}
    f_1(\Ad(x) \tau)  f_2(x)
    =
    \sum_{w \in W}
    \int_{x \in (G/H)[m]}
    f_1(\Ad(x w) \tau)
    f_2(x w).
    \]
  \end{enumerate}
\end{lemma}
\begin{proof}~
  \begin{enumerate}
  \item[(i)]
    The nontrivial Weyl element $w \in W$
    acts on $x = n'(x_1) n(x_2) H \in (G/H)[m]$
    by the formula (for $x_2 \neq 0$)
    $x w = n'(x_1 + 1/x_2) n(-x_2) H$,
    hence
    $(G/H)[m] \cap (G/H)[m] w = \emptyset$.
    The conclusion follows.
\item[(ii)]
  By Lemma
  \ref{lem:we-know-when-conjugation-nearly-stabilizes-diagonal},
  the function
  $f(x) := f_1(\Ad(x) \tau) f_2(x)$
  is supported on $(G/H)[m]$,
  so we may apply (i).
  \end{enumerate}
\end{proof}

\subsubsection{A variant of the Weyl integral formula}
\label{sec-7-3-5}
\begin{lemma*}\label{lem:weyl-integral-formula-variant}
  Let $f \in C_c(B^0)$ be supported on $\cup_{g \in G} g E^0 g^{-1}$.
  Then
  \[
  \int_{\alpha,\beta,\gamma \in k}
  f([\alpha,\beta,\gamma])
  =
  \frac{1}{|W|}
  \int_{\alpha \in k}
  |2 \alpha|^{-2}
  \int_{x \in G/H}
  f(\Ad(x) [\alpha,0,0]).
  \]
\end{lemma*}
We omit the proof. In the special case
that
$\supp(f) \subseteq (G/H)[m] W$,
the conclusion follows from Lemma
\ref{lem:useful-for-main-term-evaluation-inverse-weyl-integral-formula}
of \S\ref{sec-7-3-3}
and Lemma \ref{lem:expansion-integrals-near-singular-pts}
of \S\ref{sec-7-3-4};
this special case
suffices for our purposes and serves  also to check the normalization.
\subsubsection{Proof of the lemma of \S\ref{sec:key-computation}}
\label{sec-7-3-6}
\label{sec:appendix-proof-theorem}
Let $\phi \in \mathcal{S}(B)$.  By \eqref{eq:formula-relating-B-and-B0},
\[\langle \Ad(g) \phi, \phi  \rangle_{L^2(B)}
= \int_{\delta \in k}
\langle \Ad(g) \phi_\delta, \phi_\delta  \rangle_{L^2(B^0)},\]
where $\phi_\delta \in \mathcal{S}(B^0)$ is given by $\phi_\delta(\xi) := \phi(\delta/2 + \xi)$ for $\xi \in B^0$.
The proof of
the lemma of \S\ref{sec:key-computation}
thereby reduces to that of the following:
\begin{lemma*}
  Let $\Psi : G \rightarrow \mathbb{C}$
  and $\phi \in \mathcal{S}(B^0)$
  be arbitrary.
  Define $I : E^0 \rightarrow \mathbb{C}$ by
  $I([\alpha,0,0])
  :=
  \int_{\beta,\gamma \in k}
  \phi([\alpha,\beta,\gamma])$.
  Suppose there exists an open subgroup $U$ of $\PGL_2(k)$
  and a positive integer $m$ so that
  \begin{equation}\label{eq:Km-in-U-again-with-B0}
    K[m] \leq U
  \end{equation}
  \begin{equation}\label{eq:unit-invariance-assumed-of-phi-2}
    \text{
        $\phi(\lambda x) = \phi(x)$
  for all $x \in B^0$ and $\lambda \in 1 + \mathfrak{q}^m$.
}
  \end{equation}
  \begin{equation}\label{eq:support-phi-contained-in-E-nought-of-m-2}
    \supp(\phi) \subseteq E^0(m)
  \end{equation}
  \begin{equation}\label{eq:I-alpha-neg-alpha}
    I([\alpha,0,0]) = I([-\alpha,0,0]),
  \end{equation}
  and so that
  the decay and smoothness assumptions \eqref{eq:matrix-coeff-decay-weak-generic},
  \eqref{eq:matrix-coeff-bi-invariance} concerning $\Psi$ hold.
  Then
  \begin{equation}\label{eq:main-term-integral-so-annoying-without-cleverness-both-B0}
    \int_{g \in G}
    \Psi(g) \langle \Ad(g) \phi, \phi  \rangle = 
    (\int_{N(H)}
    \Psi)
    \int_{\alpha \in k}
    |2 \alpha|^{-2}
    |I([\alpha,0,0])|^2.
  \end{equation}
\end{lemma*}
\begin{proof}
  Note first that the RHS of 
  \eqref{eq:main-term-integral-so-annoying-without-cleverness-both-B0}
  converge absolutely, thanks to
  \eqref{eq:matrix-coeff-decay-weak-generic},
  Lemma
  \ref{lemma:convergence-Xi-along-G-and-H} of \S\ref{sec-2-1-5},
  and the compactness
  of the support of $\phi$.  We thereby reduce to establishing
  the claimed identity in the special case that $\Psi$ is the
  characteristic function of some $U \times U$-orbit; in
  particular, we may assume that $\Psi$ is compactly-supported.
  For this reason,
  we may neglect convergence issues in the arguments to
  follow.

  Equip $E^0$ with the measure
  $\int_{E^0} f := \int_{\alpha \in k} f([\alpha,0,0])$
  and define $D : E^0 \rightarrow \mathbb{R}_{\geq 0}$
  by $D([\alpha,0,0]) := |2 \alpha|^2$,
  so that the RHS of
  \eqref{eq:main-term-integral-so-annoying-without-cleverness-both-B0}
  reads
  \begin{equation}\label{eq:desired-RHS-for-weyl-int-formula-appl}
    \sum_{w \in W}
    \int_{h \in H}
    \Psi(w h)
    \int_{\tau \in E^0}
    D(\tau)^{-1} |I(\tau)|^2.
  \end{equation}
  We now
  successively transform
  the LHS of
  \eqref{eq:main-term-integral-so-annoying-without-cleverness-both-B0}.
  By expanding the definitions 
  and applying the integral formula of \S\ref{sec-7-3-5},
  we obtain
  \[
  \frac{1}{|W|}
  \int_{g \in G}
  \Psi(g)
  \int_{\tau \in E^0}
  D(\tau)
  \int_{x \in G/H}
  \overline{\phi(\Ad(g^{-1} x) \tau)}
  \phi(\Ad(x) \tau).
  \]
  We execute the (cosmetic) change of variables
  $x \mapsto g x$,
  swap orders of integration,
  and apply the (crucial)
  substitution
  $g \mapsto g x^{-1}$
  to arrive at
  \[
  \frac{1}{|W|}
  \int_{\tau \in E^0}
  D(\tau)
  \int_{x \in G/H}
  \overline{\phi(\Ad(x) \tau)}
  \int_{g \in G}
  \Psi(g x^{-1})
  \phi(\Ad(g) \tau).
  \]
  By factoring $g = y h$
  with $y \in G/H$, $h \in H$,
  we obtain
  \[
  \frac{1}{|W|}
  \int_{\tau \in E^0}
  D(\tau)
  \int_{x,y \in G/H}
  \overline{\phi(\Ad(x) \tau)}
  \phi(\Ad(y) \tau)
  \int_{h \in H}
  \Psi(y h x^{-1}).
  \]
  We apply Lemma \ref{lem:expansion-integrals-near-singular-pts}
  of \S\ref{sec-7-3-4}
  to the $x,y$ integrals,
  giving
  \begin{align*}
    \frac{1}{|W|}
    \sum_{w_1,w_2 \in W}
      \int_{\tau \in E^0}
      \int_{\substack{
      x,y \in (G/H)[m]
    }
    }
    D(\tau)
    \overline{\phi(\Ad(x w_1) \tau)}
    \phi(\Ad(y w_2) \tau)
    I'(x,y)
  \end{align*}
  where
  $I'(x,y)
    :=
    \int_{h \in H}
    \Psi(y w_2 h w_1^{-1} x^{-1})$.
  For each $\tau \in E^0 - \{0\}$,
  one has by our assumption \eqref{eq:Km-in-U-again-with-B0}
  that
  $I'(x,y) = \int_{h \in H} \Psi(w_2 h w_1^{-1})
    = \int_{h \in H} \Psi(w_2 w_1^{-1} h)$,
  by Lemma
  \ref{lem:useful-for-main-term-evaluation-inverse-weyl-integral-formula} 
  of \S\ref{sec-7-3-3} that
  $\int_{x \in (G/H)[m]} \phi(\Ad(x) \tau) = D(\tau)^{-1} I(\tau)$,
  and by our assumption \eqref{eq:I-alpha-neg-alpha}
  that $I(\Ad(w) \tau) = I(\tau)$ for all $w \in W$.
  We obtain
  \[
  \frac{1}{|W|}
  \sum_{w_1,w_2 \in W}
  \int_{h \in H}
  \Psi(w_2 w_1^{-1} h)
  \int_{\tau \in E^0}
  D(\tau)^{-1} |I(\tau)|^2,
  \]
  which simplifies to \eqref{eq:desired-RHS-for-weyl-int-formula-appl}.
\end{proof}

\section{Estimates for the error terms\label{sec:local-estimates-error}}
\label{sec-8}

\subsection{Statement of result}
\label{sec-8-2}
Recall the definitions of the Harish--Chandra function $\Xi$
(\S\ref{sec:local-Xi}) and the Weil representation
(\S\ref{sec:local-weil-reps}).
Let $\psi : k \rightarrow
\mathbb{C}^{(1)}$ be an unramified unitary character.
For $\tau \in k^\times$,
set
\[
\rho^\tau := \rho_{\Weil}^{\psi^\tau,B},
\quad 
\rho_0^\tau := \rho_{\Weil}^{\psi^\tau,B^0}.
\]
Let $\tau_1, \tau_2 \in k^\times$.
The relevance of the following definition may be inferred from 
the statement of Theorem \ref{thm:main-estimate-general-variance}.
\begin{definition*}\label{defn:good-sesquilinear-form}
  Let $\ell : \mathcal{S}(B) \otimes \mathcal{S}(B) \rightarrow
  \mathbb{C}$ be a sesquilinear form.  We say that $\ell$ is
  \emph{good}
  (relative to $\psi,\tau_1,\tau_2$)
  if:
  \begin{enumerate}
  \item
    There is an open subgroup $U$ of $G$
    so that
    for $g_1,g_2 \in U$,
    $\phi_1, \phi_2 \in \mathcal{S}(B)$ and $s \in \Mp_2(k)$, one
    has
    \begin{align}
      \ell(\phi_1, \phi_2)
      &=
        \label{eqn:symmetrization-of-ell}
        \ell(\mathfrak{S}  \phi_1, \phi_2)
        = \ell(\phi_1, \mathfrak{S}  \phi_2)
      \\
      &=
        \ell(\Ad(g_1) \phi_1, \Ad(g_2) \phi_2)
        \label{eqn:orthogonal-smoothness-of-ell}
      \\
      &=
        \ell(\rho^{\tau_1}(s) \phi_1, \rho^{\tau_2}(s) \phi_2)
        \label{eqn:metaplectic-invariance-of-ell}.
    \end{align}
  \item
    For all $\phi_1,\phi_2 \in
    \mathcal{S}(B)$
    there exists $C \geq 0$
    so that for all $s \in \Mp_2(k)$,
    \begin{equation}\label{eq:key-decay-for-ell}
      |\ell((1 \otimes \rho_0^{\tau_1}(s)) \phi_1,
      (1 \otimes \rho_0^{\tau_2}(s)) \phi_2)|
      \leq C \Xi(s).
    \end{equation}
  \end{enumerate}
\end{definition*}

\begin{proposition}\label{prop:local-error-estimates-stmt}
  For each good sesquilinear form
  $\ell : \mathcal{S}(B) \otimes \mathcal{S}(B) \rightarrow
  \mathbb{C}$ and $N_0 > \ord_k(2)$
  there exists $C \geq 0$
  so that for large positive integers $N$ and
  all $\sigma \in \Sigma$,
  the element $f \in C_c^\infty(G)$ attached to
  $(N,\sigma)$
  (see \S\ref{sec:element-attached-to-N-sigma})
  satisfies
  \begin{equation}\label{eqn:required-estimate-involving-ell}
    |\ell(\heartsuit^{\tau_1} f, \heartsuit^{\tau_2} f)| \leq C N.
  \end{equation}
\end{proposition}

The proof of Proposition \ref{prop:local-error-estimates-stmt}
occupies
\S\ref{sec-8-3}--\S\ref{sec-8-6}.
\subsection{Preliminary reduction}
\label{sec-8-3}
If Proposition
\ref{prop:local-error-estimates-stmt}
holds for $(\psi,\tau_1,\tau_2)$,
then it holds formally also for
$(\psi^{1/u}, \tau_1 u, \tau_2 u)$
for each $u \in \mathfrak{o}^\times$.
We may and shall thereby reduce
(for notational
convenience)
to the case
that $\psi \in \sigma$ (see \S\ref{sec-5-4}).

\subsection{Smoothing with respect to the adjoint action\label{sec:smoothing-wrt-adjoint}}
\label{sec-8-4}
Let $m$ be a positive integer.
Assume the following:
\begin{equation}\label{eq:m-geq-N0}
  m \geq N_0,
\end{equation}
\begin{equation}\label{eq:N-large-wrt-m}
  \text{$N$ is large enough in terms of $m$.}
\end{equation}
Let $U \leq G$
denote the $m$th principal congruence subgroup (see \S\ref{sec:princ-congr-subgr}).
Let $f \in C_c^\infty(G)$ be
attached to $(N,\sigma)$.
We study the effect of smoothing $f$
under the adjoint action of $U$.
\begin{lemma*}\label{lem:explicit-smoothed-out-variant-of-fourier-kernel}
  Set $\phi := \mathcal{F} \mathfrak{S} \heartsuit^1 f \in
  \mathcal{S}(B)$, as in \S\ref{sec-6-3}.
  Define $\phi^U \in \mathcal{S}(B)$
  by
  \begin{equation}\label{eq:defn-pih-m}
    \phi^U(x) := \mathbb{E}_{g \in U}
    \phi(\Ad(g) x).
  \end{equation}
  Then
  \begin{equation}\label{eq:desired-formula-for-phi-m}
    \phi^U (
    \begin{pmatrix}
      \delta/2 + \alpha  & \beta  \\
      \gamma  & \delta/2 - \alpha 
    \end{pmatrix}
    )
    =
    I(\alpha,\delta)
    e_{2 \alpha \mathfrak{q}^m}(\beta)
    e_{2 \alpha \mathfrak{q}^m}(\gamma).
  \end{equation}
  In particular,
  \begin{equation}\label{eq:eventual-N-stability-of-phi-m}
    q^{-N}
    \phi^U (\begin{pmatrix}
      \delta/2 + \varpi^{-N} \alpha  & \varpi^{-N} \beta  \\
      \varpi^{-N} \gamma  & \delta/2 - \varpi^{-N} \alpha 
    \end{pmatrix}
    )
    \text{ is independent of } N.
  \end{equation}
\end{lemma*}
\begin{proof}
  Proposition \ref{prop:key-fourier-estimate-microlocal-kernel}
  implies that for $N$ large enough,
  \begin{equation}\label{eq:support-condition-relevant-for-smoothing}
    \supp(\phi) \subseteq E(m)
  \end{equation}
  with
  $E(m)$ as in
  \eqref{eq:definition-of-sort-of-neighborhood-of-diagonals},
  and also that
  $\phi$ satisfies
  the smoothness property
  \eqref{eq:unit-invariance-assumed-of-phi}
  noted above.
  To deduce
  \eqref{eq:desired-formula-for-phi-m}
  from
  \eqref{eq:unit-invariance-assumed-of-phi} and
  \eqref{eq:support-condition-relevant-for-smoothing}
  is a calculus
  problem;
  to solve it, we apply Lemma
  \ref{lem:basic-integral-formula-for-diag-inv-f}
  of \S\ref{sec-7-3-3}
  to the functions $\phi_\delta \in \mathcal{S}(B^0)$
  attached to $\phi \in \mathcal{S}(B)$
  and $\delta \in k$
  by $\phi_\delta(\xi) := \phi(\delta/2 + \xi)$ for $\xi \in
  B^0$.
  The final assertion \eqref{eq:eventual-N-stability-of-phi-m} follows
  from \eqref{eq:I-rescaled-independent-of-N}.
\end{proof}

\subsection{Metaplectic interpretation\label{sec:error-estimates-metaplectic}}
\label{sec-8-5}
Retain the notation and setting of
\S\ref{sec:smoothing-wrt-adjoint}.  We now interpret
\eqref{eq:eventual-N-stability-of-phi-m} in terms of the Weil
representation.
Recall the double cover
$\pr : \Mp_2(k) \rightarrow \SL_2(k)$
and the elements $n(b), t(a), w \in \Mp_2(k)$
as in \S\ref{sec:local-weil-repn}.
To reduce clutter in the formulas
to follow, we introduce some notation
and terminology:
\begin{definition*}
  Let $a_1,a_2$ be quantities depending
  implicitly upon the large positive integer $N$ and a field element $\tau \in k^\times$  (as well as
  the field $k$ containing
  $\tau$, of course).
  Thus
  $a_i = a_i(\tau,N)$.
  Write $a_1 \approx a_2$ to denote that
  $a_1 = \gamma |\tau|^{c/4} b$,
  where
  \begin{itemize}
  \item  $\gamma \in \mathbb{C}^{(1)}$
    is
    an eighth root of unity
    that may depend upon $\tau$ and $N$,
    and
  \item  $c \in \mathbb{Z}$ depends neither upon
    $\tau$ nor upon $N$.
  \end{itemize}
  We say that a quantity $a$ is \emph{essentially independent of $N$} if
  $a \approx b$, where $b$ is independent of $N$.
\end{definition*}

\begin{lemma*}\label{lem:essential-independence-after-translation}
  Let $\tau \in k^\times$.  If $\tau \notin \mathfrak{o}^\times k^{\times 2}$,
  then $\heartsuit^{\tau} f = 0$.
  Otherwise there exists $\phi_0 \in \mathcal{S}(B)$
  that is essentially independent of $N$
  so that
  \begin{equation}\label{eqn:essential-indepdence-0}
    \Ad(e_U)
    \rho^{\tau}(w)
    \mathfrak{S} 
    \heartsuit^{\tau} f
    =
    q^{N/2}
    (1 \otimes \rho_0^{\tau}(t(\varpi^{N})))
    \phi_0.
  \end{equation}
\end{lemma*}
\begin{proof}
  The first assertion is immediate from the definitions of
  $\heartsuit^{\tau}$ and $f$.
  Abbreviate
  $f' := \mathfrak{S} f$.
  Since $\heartsuit^\tau f' = \mathfrak{S} \heartsuit^{\tau} f$,
  it suffices now
  to show
  for $\tau \in \mathfrak{o}^\times k^{\times 2}$
  that the required conclusion holds in the inverted form
  \begin{equation}\label{eqn:essential-indepdence-1}
    q^{-N/2} (1 \otimes \rho_0^{\tau}(t(\varpi^{-N}))) \Ad(e_U)
    \rho^{\tau}(w) \heartsuit^{\tau}  f'
    \text{
      is essentially independent of } N.
  \end{equation}
  The case $\tau = 1$
  of
  \eqref{eqn:essential-indepdence-1}
  follows from
  the lemma of \S\ref{sec-8-4}
  and the formulas describing of Weil representation.
  In general,
  we see
  by inspecting the definitions
  that
  $\heartsuit^{\tau} f'
  =
  \heartsuit^{\nu^2} f'
  \approx 
  \rho^{\tau}(t(\nu)) \heartsuit^1 f'$
  and
  that
  $\rho^\tau(w)
  \approx
  \rho^\tau(t(\nu))
  \rho^1(w)$,
  hence
  that
  \begin{equation}\label{eqn:essential-indepdence-2}
    \rho^{\tau}(w)
    \heartsuit^{\tau} f'
    \approx 
    \rho^\tau(t(\nu))
    \rho^1(w)
    \rho^{\tau}(t(\nu)) \heartsuit^1 f'
    \approx
    \rho^1(w) \heartsuit^1 f'.
  \end{equation}
  Moreover,
  $\rho_0^{\tau}(t(\varpi^{-N})) \approx
  \rho_0^1(t(\varpi^{-N}))$.
  Thus the identity \eqref{eqn:essential-indepdence-1}
  for general $\tau \in \mathfrak{o}^\times k^{\times 2}$
  reduces to the $\tau = 1$ case already established.
\end{proof}
\subsection{Completion of the proof}
\label{sec-8-6}
We now prove
Proposition \ref{prop:local-error-estimates-stmt}.
For $i=1,2$,
set $\phi_i := \mathfrak{S} \rho^{\tau_i}(w) \heartsuit^{\tau_i}
f$.
By \eqref{eqn:symmetrization-of-ell}
and
\eqref{eqn:metaplectic-invariance-of-ell}, we have
\begin{equation}
  \ell(\heartsuit^{\tau_1} f, \heartsuit^{\tau_2} f) = \ell(\phi_1,\phi_2).
\end{equation}
By averaging \eqref{eqn:orthogonal-smoothness-of-ell} over
$g_1,g_2 \in U$,
we have
\begin{equation}
  \ell(\phi_1,\phi_2) = \ell(\phi_1^U,\phi_2^U),
\end{equation}
with
notation as in \S\ref{sec:smoothing-wrt-adjoint}.
By shrinking $U \leq G$
as necessary,
we may assume that $U$ is the $m$th principal
congruence subgroup
for some $m \geq N_0$;
since $\ell$ and $N_0$ are independent
of $N$,
we may assume also that
$N$ is sufficiently large in terms of $m$.
By the lemma of \S\ref{sec-8-5},
we may suppose that
$\tau_1, \tau_2 \in \mathfrak{o}^\times k^{\times 2}$,
in which case
\[
\phi_i^U \approx q^{N/2} (1 \otimes
\rho_0^{\tau_i}(t(\varpi^N)) \phi_0,
\]
where $\phi_0 \in \mathcal{S}(B)$
is essentially independent of $N$ (see \S\ref{sec:error-estimates-metaplectic}).
Our task thereby reduces
to establishing the estimate
\begin{equation}
  \ell(
  (1 \otimes \rho_0^{\tau_1}(t(\varpi^N))) \phi_0,
  (1 \otimes \rho_0^{\tau_2}(t(\varpi^N))) \phi_0
  ) \ll N q^{-N},
\end{equation}
which follows finally from the condition
\eqref{eq:key-decay-for-ell}
in the definition of ``good.''

\begin{remark*}
    The sesquilinear forms $\ell$ to which we apply Proposition
    \ref{prop:local-error-estimates-stmt}
    below may be assumed to have an additional property
    (beyond being ``good'').
    That property is not directly relevant for the
    immediate purposes of this paper, but may be useful in future work,
    so we briefly record it:  There are irreducible unitary
    representations $\pi_1, \pi_2$ of $G$ of dimension $> 1$ (arising as local
    components of cuspidal automorphic representations) so that $\ell$ factors
    as $\ell = \tilde{\ell} \circ ((1 \otimes \theta_1) \otimes (1 \otimes
    \theta_2))$,
    where
    \begin{enumerate}
  \item $\sigma_i$ is the local $\psi^{\tau_i}$-theta lift of
    $\pi_i$
    as in \cite{MR1103429}, i.e., the
    irreducible representation of $\Mp_2(k)$ for which one
    has
    $\Hom_G(\rho_0^{\tau_i},\pi_i) =
    \Hom_\mathbb{C}(\sigma_i,\mathbb{C})$,
    \item
      $\tilde{\ell} :
    \mathcal{S}(k) \otimes \mathcal{S}(k) \otimes
    \sigma_1 \otimes \sigma_2
    \rightarrow \mathbb{C}$
    is a sesquilinear form
    invariant by the diagonal action of $\Mp_2(k)$,
  \item
    $\theta_i : \mathcal{S}(B^0) \rightarrow \sigma_i$
    is a basis element for the one-dimensional
    space of
    $G$-equivariant maps $\rho_0^{\tau_i} \rightarrow \sigma_i$, and
  \item
    $1 \otimes \theta_i : \mathcal{S}(B) = \mathcal{S}(k) \otimes
    \mathcal{S}(B^0) \rightarrow \mathcal{S}(k) \otimes \sigma_i$
    and
    $(1 \otimes \theta_1) \otimes (1 \otimes \theta_2) :
    \mathcal{S}(B) \otimes \mathcal{S}(B) \rightarrow
    \mathcal{S}(k) \otimes \mathcal{S}(k) \otimes \sigma_1 \otimes
    \sigma_2$ are the evident maps.
  \end{enumerate}
  The $\phi = \heartsuit f$ of interest in this paper
  concentrate on
  semisimple elements.  If one instead considers $\phi$ 
  supported close to the nilcone (as arise
  naturally when studying classical newvectors),
  then one may exploit bounds towards temperedness of the
  $\sigma_i$ and the above factorization of $\ell$ to produce
  estimates sharper than those that follow from $\ell$ being
  good.
\end{remark*}

\subsection{Bounds for partial orbital integrals}
\label{sec:bounds-part-orbit}
The contents of this short
miscellaneous section are used below to deduce the
assertions made in \S\ref{sec-1} concerning the family
cardinality and mean statistics;
they are directly related neither to
the rest of \S\ref{sec-8} nor to the main new ideas of this paper.

The parameter $N_0$ as in \S\ref{sec:element-attached-to-N-sigma} is regarded here as fixed
once and for all.  Recall that a \emph{regular semisimple}
element $\gamma \in B^\times$ is one which is diagonalizable
with distinct eigenvalues over an algebraic closure
$\overline{k}$ of $k$, or equivalently, for which
$\tr(\gamma)^2 \neq 4 \nr(\gamma)$.
\begin{lemma*}\label{lem:bounds-part-orbit}
  Let $\gamma \in B^\times$ be regular semisimple.  Let $U_2$ be a
  compact open subgroup of $G$ that is small enough in terms of
  $\gamma$.
  Let $N$ be large enough in terms of
  $(\gamma,U_2)$.
  Let $f \in C_c^\infty(G)$ be attached
  to $(N,\sigma)$ for some $\sigma \in \Sigma$.
  Then $\int_{u \in U_2} f(u^{-1} \gamma u) = 0$.
\end{lemma*}

\begin{proof}
  Let $M_1$ denote the semidirect product
  $k \rtimes k^\times$.
  Set $M := M_1 \times G$.
  The group $M$
  consists of triples $(x,y,z) \in k \times k^\times \times G$
  with the group law
  $(x_1,y_1,z_1)(x_2,y_2,z_2) = (x_1 + y_1 x_2, y_1 y_2, z_1
  z_2)$.
  For $(x,y,z) \in M$ and $b \in B$, set
  $(x,y,z) \cdot b := z (y b + x) z^{-1}$.  This formula defines
  an action of $M$ on $B$.

  Set $\tau := \nr(\gamma)^{-1}$ and
  $\phi := \heartsuit^{\tau} f \in \mathcal{S}(B)$.
  For $g \in G$, one then has $\tau \nr( g^{-1} \gamma g) = 1$
  and thus
  $\phi(g^{-1} \gamma g) = f(g^{-1} \gamma g)$.
  By the lemma of \S\ref{sec-6-4}
  (or directly from the definitions),
  there is an open subgroup $U_1 \leq M_1$,
  depending only upon $\tau$,
  so
  that $\phi(u_1 \cdot b) = \phi(b)$ for all
  $u_1 \in U_1, b \in B$.
  Setting
  $U := U_1 \times U_2 \leq M$,
  our task reduces
  to showing that
  \begin{equation}\label{eq:goal-vanishing-local-orbital-integral-for-phi}
    \mathbb{E}_{u \in U} \phi(u^{-1} \gamma u) = 0
  \end{equation}
  for $N$  large enough in terms of $(\gamma,U)$.
  To that end, observe that the orbit map $M \rightarrow B$ given by
  $m \mapsto m \cdot \gamma$ is submersive at the identity: it
  suffices to check this claim over $\overline{k}$, where it
  follows by direct calculation in the diagonal case.
  By Hensel's lemma as in \S\ref{sec-7-3-3}, together with the assumption that
  $U \leq M$ is small enough in
  terms of $\gamma$,
  we see
  that
  \begin{enumerate}
  \item the orbit $O := U \cdot \gamma \subseteq B$ is open,
  \item the orbit map $U \rightarrow O$
    is a diffeomorphism, and
  \item the pushforward of the probability Haar on $U$ under the
    orbit map is the probability measure on $O$ induced by the
    Haar on $B$.
  \end{enumerate}
  Our goal
  \eqref{eq:goal-vanishing-local-orbital-integral-for-phi} is
  thus equivalent to showing that
  $\int_O \phi = \langle \phi, 1_O \rangle_{L^2(B)} = 0$.  Since
  the Fourier transform of $1_O$ has compact support, we reduce
  by Parseval to showing that the support of $\mathcal{F} \phi$
  tends to $\infty$ as $N \rightarrow \infty$, which follows
  from the condition $\alpha \asymp q^{-N}$ of Proposition
  \ref{prop:key-fourier-estimate-microlocal-kernel} (arguing as
  in \S\ref{sec:error-estimates-metaplectic} to relate
  $\heartsuit^{\tau}, \heartsuit^{1}$).
\end{proof}

\section{Deduction of the main theorem\label{sec:deduction-main-thm-microlocal}}
\label{sec-9}

\subsection{Setting}\label{sec:9-setting}
We adopt here the notation and setting of
\S\ref{sec:general-estimates-specialized-single-place}, but
assume now that $\mathfrak{q}$ is a finite place and that
$B_\mathfrak{q}$ is split.  Our existing assumptions imply then
that $F$ is totally real and that $B$ is totally definite.

Set $k := F_\mathfrak{q}$, fix an identification
$\PB^\times_\mathfrak{q} = \PGL_2(k)$, and adopt the notation
$\mathfrak{o}, \mathfrak{q}, q$
from \S\ref{sec-2-1-1}.  Let $N$ be a large
positive integer, and let
$\omega : \mathfrak{o}^\times \rightarrow \mathbb{C}^\times$ be
a character of conductor $N$.  Let $\tilde{\mathcal{F}}_\omega$
denote the set of nonzero vectors $\varphi \in \mathcal{A}_0^J$
for which
\begin{itemize}
\item there exists $\pi \in A_0^J$
  so that $\varphi \in \pi$, and
\item $\varphi$ is a microlocal lift of orientation
  $\omega$ in the sense of \S\ref{sec-5-3}.
\end{itemize}
For
$\pi \in A_0^J$,
the set
$\pi \cap \tilde{\mathcal{F}}_\omega$ is either empty or of the
form
$\mathbb{C}^\times \varphi_1 \sqcup \mathbb{C}^\times \varphi_2$
for some $\varphi_1, \varphi_2 \in \pi^J$
(see \S\ref{lem:determination-microlocal-lifts}),
hence
$\tilde{\mathcal{F}}_\omega$ is a union of scaling classes
$\mathbb{C}^\times \varphi$.  Choose a set $\mathcal{F}_\omega$
consisting of one unit vector from each such scaling class,  so
that $\tilde{\mathcal{F}}_\omega = \bigsqcup_{\varphi \in
  \mathcal{F}_\omega} \mathbb{C}^\times \varphi$.
The
discussion of \S\ref{sec:setting-overview} and \S\ref{sec-1-4}
applies, giving us a
sequence of sets
$\mathcal{F}_N := \sqcup_{\omega \in \mathcal{X}_N}
\mathcal{F}_\omega \subset \mathcal{A}^J_0$
indexed by large positive integers $N$.

Define $\Sigma$ and $\mathcal{X}_N^{\sigma}$
as in \S\ref{sec-5-4}  with respect to some fixed but large enough natural number
$N_0$.
The partition
$\mathcal{X}_N = \bigsqcup_{\sigma \in \Sigma}
\mathcal{X}_N^\sigma$
of the group of characters of $\mathfrak{o}^\times$ of conductor $N$
induces a partition
$\mathcal{F}_N = \bigsqcup_{\sigma \in \Sigma}
\mathcal{F}_N^\sigma$ of the family of microlocal lifts.
We emphasize that $|\Sigma| = O(1)$.

\subsection{Mean statistics}\label{sec:mean-statistics}
We have included this section to complete the discussion of  \S\ref{sec-1};
it has nothing to do with the main new ideas of this paper.

Since smooth functions on $\mathbf{X}$ are uniformly
dense in the space of continuous functions,
the lemmas of \S\ref{sec-1-2} and \S\ref{sec-1-3} are consequences of the following:
\begin{lemma*}
  Fix $\Psi \in \mathcal{A}$.
  Assume that $N$ is large enough in terms of $\Psi$.
  Then
  \begin{equation}\label{eq:local-weyl-law-final-section}
    \sum_{\varphi \in \mathcal{F}_N}
    \langle \varphi, \Psi \varphi  \rangle
    =
    c q^{2 N}
    \frac{\langle 1, \Psi  \rangle}{ \langle 1, 1 \rangle}
  \end{equation}
  where with $\Delta_F, \Delta_B$ the absolute (reduced)
  discriminants
  and $\ram_f(B)$ the set of finite places
  at which $B$ ramifies,
  \[
  c :=
  2
  \frac
  {
    \zeta_F(2)
    \Delta_B
    \Delta_F^{3/2}
  }
  {
    (4 \pi^2)^{[F:\mathbb{Q}]}
    \prod_{\mathfrak{p} \in \ram_f(B)}
    \zeta_\mathfrak{p}(1)
  }
  \frac{
    |2|_k
  }
  {
    \zeta_k(1) \zeta_k(2)
  }.
  \]
\end{lemma*}

\begin{example*}
  Suppose that $F = \mathbb{Q}$, that $B$ is ramified precisely
  at $\{\infty,D\}$ for some prime $D \in \mathbb{Z}_{\geq 1}$,
  and that $\mathfrak{q}$ corresponds to some prime
  $q \in \mathbb{Z}_{\geq 1}$.
  By taking $\Psi = 1$ in the lemma
  and evaluating $c$,
  we obtain that for $N$ large enough,
  \[
  |\mathcal{F}_N|
  =
  q^{2 N}
  \frac{D-1}{12}
  (1 - \frac{1}{q})
  (1 - \frac{1}{q^2})
  \cdot \begin{cases}
    1/2 & q = 2 \\
    1 & q \text{ is odd.}
  \end{cases}
  \]
\end{example*}

\begin{proof}[Proof of the lemma]
  It will suffice to show for each $\sigma \in \Sigma$ that
  \begin{equation}\label{eq:local-weyl-law-final-section-lhs-rewritten}
    |\Sigma| \sum_{\varphi \in \mathcal{F}_N^\sigma} \langle \varphi, \Psi \varphi \rangle
  \end{equation}
  is eventually equal to the RHS of
  \eqref{eq:local-weyl-law-final-section}.
  For the remainder of the proof,
  set $\mathbf{G} := \bPB^\times$.  
  Let $f = \prod_\mathfrak{p} f_\mathfrak{p} \in
  C_c^\infty(G_\mathbb{A})$
  be given
  by $f_\mathfrak{p} := \vol(J_\mathfrak{p})^{-1} 1_{J_\mathfrak{p}}$
  for $\mathfrak{p} \neq \mathfrak{q}$
  (see \S\ref{sec:general-estimates-specialized-single-place})
  and by taking for $f_\mathfrak{q}  \in
  C_c^\infty(G_\mathfrak{q}) = C_c^\infty(\PGL_2(k))$
  the element attached to $(N,\sigma)$ as in \S\ref{sec:element-attached-to-N-sigma}.
  By Proposition \ref{prop:harmonic-analytic-isolation-1},
  the example of \S\ref{sec:omega-pi},
  and the pretrace formula (\S\ref{sec:pretrace-formula}),
  we have
  \begin{equation}\label{eqn:integrated-pretrace-formula}
    \sum_{\varphi \in \mathcal{F}_N^\sigma} \langle \varphi, \Psi
    \varphi \rangle
    =
    \int_{g \in [G]}
    \Psi(g)
    \sum_{\gamma \in G}
    f(g^{-1} \gamma g).
  \end{equation}
  As in the proof of the trace formula, the RHS of
  \eqref{eqn:integrated-pretrace-formula} may be folded as
  $\sum_{\{\gamma \}} I(\gamma)$, where $\gamma$ traverses
  a set of representatives
  for the $G$-conjugacy classes in $G$ and
  $I(\gamma) := \int_{h \in [G_\gamma]} \int_{g \in
    G_{\gamma,\mathbb{A}} \backslash \mathbf{G}_{\mathbb{A}}} \Psi(h g)
  f(g^{-1} \gamma g)$;
  here $\mathbf{G}_\gamma$ denotes the centralizer.  The function $f$ is
  supported in a fixed (i.e., independent of $N$) compact subset
  of $G_\mathbb{A}$, hence $I(\gamma) = 0$ for $\gamma$ outside
  some fixed finite collection of representatives.

  Let $\mu$ denote the Tamagawa measure
  on $\PB^\times_\mathbb{A}$.
  Recall the subgroup $\mathfrak{J} \leq G_\mathfrak{q}$
  arising in the definition of $f$.
  Set
  $\mathfrak{J} ' := \mathfrak{J}  \times \prod_{\mathfrak{p} \neq \mathfrak{q}} J_\mathfrak{p}$.
  Since $|\Sigma| \cdot |\mathcal{X}_N^\sigma| =
  |\mathcal{X}_N|$,
  one has
  \begin{equation}\label{eq:Sigma-I-of-1-rewrite}
    |\Sigma|
    I(1)
    =
    \langle \Psi, 1  \rangle
    |\Sigma| f(1)
    =
    \langle \Psi, 1  \rangle
    |\mathcal{X}_N|
    \mu(\mathfrak{J} ')^{-1}.
  \end{equation}
  To compute the latter volume, it is convenient to factor
  $\mu = \prod \mu_\mathfrak{p}$ into the local measures
  $\mu_\mathfrak{p}$ as defined in \S\ref{sec:local-measures} relative to the
  \emph{standard} nontrivial unitary character
  $\psi = \prod \psi_\mathfrak{p}$ of $\mathbb{A}/F$, i.e., that
  for which $\psi_\mathfrak{p}(x) = e^{2 \pi i x}$ for infinite
  places $\mathfrak{p}$.
  For a finite place $\mathfrak{p}$,
  $\Delta_{F_\mathfrak{p}} = \Delta_{\psi_\mathfrak{p}}$ is then
  the absolute discriminant of $F_\mathfrak{p}$.
  We record some consequences of
  the volume formulas of \S\ref{sec:local-vol-formulas}
  and \S\ref{sec:measures-for-fourier-computation}:
  \begin{enumerate}
  \item If $\mathfrak{p}$  is finite
    and $B_\mathfrak{p}$ splits,
    then
    $\mu_\mathfrak{p}(J_\mathfrak{p}) =
    \zeta_\mathfrak{p}(2)^{-1}
    \Delta_{B_\mathfrak{p}}^{-1}
    \Delta_{F_\mathfrak{p}}^{-3/2}$.
  \item If $\mathfrak{p}$  is finite
    and $B_\mathfrak{p}$ ramifies,
    then
    $\mu_\mathfrak{p}(J_\mathfrak{p}) =
    \zeta_\mathfrak{p}(1)
    \zeta_\mathfrak{p}(2)^{-1}    \Delta_{B_\mathfrak{p}}^{-1}
    \Delta_{F_\mathfrak{p}}^{-3/2}$.
  \item If $\mathfrak{p}$ is infinite,
    then $\mu_\mathfrak{p}(J_\mathfrak{p}) =
    \mu_\mathfrak{p}(G_\mathfrak{p})
    = 4 \pi^2$.
  \item 
  $\mu_{\mathfrak{q}}(\mathfrak{J}) /
    \mu_{\mathfrak{q}}(J_\mathfrak{p})
    = |2|_k^{-1} q^{-N} \zeta_k(1)^{-1} \zeta_k(2)$.
  \end{enumerate}
  Therefore
  \begin{equation}
    \mu(\mathfrak{J} ')^{-1}
    =
    \frac{
      \zeta_F(2)
      \Delta_B
      \Delta_F^{3/2}
    }
    {
      (4 \pi^2)^{[F:\mathbb{Q}]}
      \prod_{\mathfrak{p} \in \ram_f(B)}
      \zeta_\mathfrak{p}(1)
    }
    \frac{
      |2|_k
      q^N
      \zeta_k(1)}{\zeta_k(2)}.
  \end{equation}
  Since $|\mathcal{X}_N| = q^N / \zeta_k(1)^2$
  and $\langle 1,1  \rangle = \mu([\PB^\times]) = 2$,
  we obtain
  \begin{equation}\label{eq:Sigma-I-of-1-rewrite-2}
    |\mathcal{X}_N|
    \mu(\mathfrak{J} ')^{-1}
    = c q^{2 N}/ 2 =  c q^{2 N} / \langle 1,1 \rangle.
  \end{equation}
  By \eqref{eq:Sigma-I-of-1-rewrite}
  and \eqref{eq:Sigma-I-of-1-rewrite-2},
  the contribution from $I(1)$
  to \eqref{eq:local-weyl-law-final-section-lhs-rewritten}
  gives the required RHS of \eqref{eq:local-weyl-law-final-section}.

  To complete the proof, it suffices now to show for each
  fixed $1 \neq \gamma \in G$ that $I(\gamma) = 0$ for $N$
  large enough.
  Let $U$ be a small enough but fixed compact open subgroup
  of $G_\mathfrak{q}$
  under which $\Psi$ is invariant.
  By a change of variables in the definition of $I(\gamma)$,
  our task reduces to showing for all $g \in G_{\mathbb{A}}$
  that
  \begin{equation}\label{eq:penultimate-goal-for-linear-stats}
    \int_{u \in U}
    f(u^{-1} g^{-1} \gamma g u) = 0 \text{ for $N$ large enough.}
  \end{equation}
  Fix a compact set $E$ (independent of $N$) containing the
  support of $f$, and let
  $j : G_{\mathbb{A}} \rightarrow G_{\mathbb{A}}$ denote the
  orbit map $j(g) := g^{-1} \gamma g$.  Since $B$ is non-split,
  the group element $\gamma$ is regular semisimple
  (see \S\ref{sec:bounds-part-orbit}).  The map
  $G_{\gamma,\mathbb{A}} \backslash G_{\mathbb{A}} \rightarrow
  G_{\mathbb{A}}$
  induced by $\gamma$ is thus proper, and so the set $j^{-1}(E)$
  meets only a fixed finite collection of double cosets
  $G_{\gamma,\mathbb{A}} g U$.  For this reason, it suffices to
  establish \eqref{eq:penultimate-goal-for-linear-stats} for
  each \emph{fixed} $g \in G_{\mathbb{A}}$.
  The conclusion follows in that case from
  the lemma of \S\ref{sec:bounds-part-orbit}.
\end{proof}

\subsection{Variance statistics}\label{sec:variance-statistics}

\subsubsection{The sums}
Define the sesquilinear
forms
$V_N : \mathcal{A}_{0+}^J \otimes \mathcal{A}_{0+}^J
\rightarrow \mathbb{C}$
by
\[
V_N(\Psi_1,\Psi_2)
=
\sum_{\varphi \in \mathcal{F}_N}
L^{(S)}(\ad \varphi,1)
\langle
\varphi, \Psi_1 \varphi 
\rangle
\langle \Psi_2 \varphi, \varphi 
\rangle.
\]
We have written $L^{(S)}(\ad \varphi,1) := L^{(S)}(\ad \pi,1)$
for $\varphi \in \pi \in A_0^J$.

\subsubsection{The leading constant}
Set $\mathbf{X} := [\PB^\times]/J$.
Equip it with the quotient measure
induced by any Haar measure on $[\PB^\times]$;
the formulation of our results
will not depend upon this choice.
Set
\begin{equation}\label{eq:defn-of-c0}
  c_0 := 2^{\# \ram_f(B)}
  \zeta_F^{(S)}(2)
  \vol(\mathbf{X})^{-1} \frac{1}{2 \zeta_k(1)}.
\end{equation}

\begin{example*}
  In the setting of the example of \S\ref{sec:mean-statistics},
  suppose that we identify
  $\mathbf{X}$ with $\Gamma \backslash \PGL_2(\mathbb{Q}_q)$ as
  in \S\ref{sec:strong-approx}
  and equip $\mathbf{X}$
  with the quotient measure induced by
  the Haar on $\PGL_2(\mathbb{Q}_q)$
  assigning volume one to $\PGL_2(\mathbb{Z}_q)$.
  Then
  $\vol(\mathbf{X})= (D-1)/12$ (see the remark of \S\ref{sec:strong-approx}).
  After some calculation, we obtain
  \[
  c_0
  =
  2 \pi^2
  \frac{2 (1 + D^{-1})}{D}
  \frac{(1 - q^{-1}) (1 - q^{-2})}{2}
  \]
\end{example*}

\subsubsection{The proposed limit}
Let
$V_\infty : \mathcal{A}_{0+}^J \otimes \mathcal{A}_{0+}^J
\rightarrow \mathbb{C}$ denote the sesquilinear form given for
$\Psi_i \in \pi_i \in A_{0+}^J$ ($i=1,2$) by
$V_\infty(\Psi_1,\Psi_2) := 0$ unless $\pi_1 = \pi_2 =: \pi$, in
which case
\[
V_\infty(\Psi_1,\Psi_2) :=
c_0 L^{(S)}(\pi,\tfrac{1}{2})
\int_{h \in N(H)}
\langle \pi(h) \Psi_1, \Psi_2 \rangle_{L^2(\mathbf{X})},
\]
where the measure on $N(H)$ is as in
\S\ref{sec:local-estimates-main-statements}.

\subsubsection{Main result}
In view of the discussion of \S\ref{sec:strong-approx} and \S\ref{sec:hecke-ops},
the following result
makes precise
(and mildly generalizes) Theorem \ref{thm:main-result-for-microlocal-stuff}:
\begin{theorem}\label{thm:main-variance-microlocal-adelic-formulation}
  Let $\Psi_1, \Psi_2 \in \mathcal{A}_{0+}^J$ be fixed  (i.e., independent of $N$).
  Then
  \begin{equation}q^{-N} V_N(\Psi_1,\Psi_2)
    = V_\infty(\Psi_1,\Psi_2)
    + O(N q^{-N}).
  \end{equation}
\end{theorem}
\begin{proof}
  The proof divides into five steps:
  \begin{enumerate}
  \item 
    The partition
    $\mathcal{F}_N = \bigsqcup_{\sigma \in \Sigma}
    \mathcal{F}_N^\sigma$ of the family
    induces a decomposition
    $V_N = \sum_{\sigma \in \Sigma} V_N^\sigma$ of the variance.
    Since $|\Sigma| \asymp 1$, it will suffice to show that
    \begin{equation}q^{-N} V_N^{\sigma}(\Psi_1,\Psi_2)
      = |\Sigma|^{-1} V_\infty(\Psi_1,\Psi_2)
      + O(N q^{-N}).\end{equation}
    
  \item   Let $f \in C_c^\infty(\PB^\times_\mathfrak{q})$ be attached to $(N,\sigma)$
    (see \S\ref{sec:element-attached-to-N-sigma}).
    Recall from
    \S\ref{sec:general-estimates-specialized-single-place}
    the definitions of $V_f,M_f$.
    By Proposition \ref{prop:harmonic-analytic-isolation-1}
    and
    the example of \S\ref{sec:omega-pi},
    we have $V_N^\sigma = V_f$.
  \item 
    By \eqref{eq:defn-of-c4} and \eqref{eq:Sigma-cardinality},
    we have
    $c_4 q^{N-N_0} \frac{1}{2}
    =
    q^N |\Sigma|^{-1} c_0$.
    Feeding  this calculation into Proposition
    \ref{prop:desired-main-term-identity-do-it-up}
    gives
    \begin{equation}M_f(\Psi_1,\Psi_2)
      = q^N |\Sigma|^{-1} V_\infty(\Psi_1,\Psi_2) \end{equation}
    for $N$ large enough in terms of $\Psi_1,\Psi_2$.
    We reduce to showing that
    \begin{equation}\label{eqn:desired-error-estimate--do-it-1}
      V_f(\Psi_1,\Psi_2) - M_f(\Psi_1,\Psi_2) \ll N.
    \end{equation}

  \item   Fix a nontrivial unitary character $\psi$ of
    $\mathbb{A}/F$
    whose component $\psi_\mathfrak{q}$ is unramified.
    Fix a nonzero element $W_S = \prod_{\mathfrak{p} \in S}
    W_\mathfrak{p} \in C_c^\infty(F_S^\times)$
    for which $W_\mathfrak{q} := 1_{\mathfrak{o}^\times}$.
    We apply Theorem \ref{thm:main-estimate-general-variance}
    with respect to $\psi$ and $W_S$; our task thereby reduces to showing for fixed
    $\tau_1,\tau_2 \in F^\times$ that
    \begin{equation}\label{eq:penultimate-task-almost-there}
      \eps_{\tau_1,\tau_2}(\heartsuit^{\tau_1} \tilde{f},
      \heartsuit^{\tau_2} \tilde{f},
      \Psi_1,\Psi_2)
      \ll N,
    \end{equation}
    where $\tilde{f} \in C_c^\infty(\PB^\times_S)$ is as in \S\ref{sec:general-estimates-specialized-single-place}
    and
    $\heartsuit^{\tau} \tilde{f} \in \mathcal{S}(B_S)$
    as in \S\ref{sec:heartsuit}.

  \item   One has
    $\heartsuit^{\tau} \tilde{f}
    = (\heartsuit^{\tau} f) \otimes \phi^{\tau}$
    where $\heartsuit^{\tau} f \in \mathcal{S}(B_\mathfrak{q})$ is
    as in \S\ref{sec:defn-heartsuit-local-again-without-Omega} and
    $\phi^{\tau} = \otimes_{\mathfrak{p} \in S - \{\mathfrak{q} \}}
    \phi_\mathfrak{p}^{\tau}$ with $\phi_\mathfrak{p}^{\tau} \in
    \mathcal{S} (B_\mathfrak{p} )$
    independent of $N$.
    The LHS of \eqref{eq:penultimate-task-almost-there}
    is thus equal to $\ell(\heartsuit^{\tau_1} f,
    \heartsuit^{\tau_2} f)$,
    where $\ell : \mathcal{S}(B_\mathfrak{q}) \otimes
    \mathcal{S}(B_\mathfrak{q})
    \rightarrow \mathbb{C}$
    denotes the sesquilinear form given by
    \[
    \ell(\Phi_1,\Phi_2)
    :=
    \eps_{\tau_1,\tau_2}( \Phi_1
    \otimes \phi^{\tau_1},
    \Phi_2
    \otimes \phi^{\tau_2},
    \Psi_1,\Psi_2).
    \]
    Observe that $\ell$ is independent of $N$.
    Theorem \ref{thm:main-estimate-general-variance} says that
    $\ell$ is good in the sense of \S\ref{sec-8-2} with respect to
    $(\psi_\mathfrak{q},\tau_1,\tau_2)$
    (taking for $U \leq \PB^\times_\mathfrak{q}$
    any open subgroup under which $\Psi_1,\Psi_2$ are invariant).
    Proposition \ref{prop:local-error-estimates-stmt}
    implies
    finally that $\ell(\heartsuit f, \heartsuit f) \ll N$,
    giving \eqref{eq:penultimate-task-almost-there},
    as required.
  \end{enumerate}
\end{proof}

\subsection*{Acknowledgements}
We gratefully acknowledge the support of
NSF grant OISE-1064866
and SNF grant SNF-137488 during the work leading
to this paper.

\bibliography{refs}{}
\bibliographystyle{plain}
\end{document}